\documentclass[reqno]{amsart}%
\usepackage{amsmath}
\usepackage{amsfonts}
\usepackage{amssymb}
\usepackage{graphicx}%
\pdfoutput=1
\newtheorem{theorem}{Theorem}

\newtheorem{conclusion}[theorem]{Conclusion}

\newtheorem{corollary}[theorem]{Corollary}

\newtheorem{definition}[theorem]{Definition}

\newtheorem{lemma}[theorem]{Lemma}
\newtheorem{notation}[theorem]{Notation}

\newtheorem{proposition}[theorem]{Proposition}
\newtheorem{remark}[theorem]{Remark}

\newtheorem{summary}[theorem]{Summary}
\newtheorem{assumption}[theorem]{Assumption}
\numberwithin{theorem}{section}
\numberwithin{equation}{section}
\begin{document}
\title[The Brown measure of the sum]{The Brown measure of the sum of a self-adjoint element and an imaginary
multiple of a semicircular element}
\author{Brian C.\ Hall}
\address{Department of Mathematics, University of Notre Dame, Notre Dame, IN 46556, USA}
\email{bhall@nd.edu}
\author{Ching-Wei Ho}
\address{Institute of Mathematics, Academia Sinica, Taipei 10617, Taiwan; Department of
Mathematics, University of Notre Dame, Notre Dame, IN 46556, USA}
\email{cho2@nd.edu}
\thanks{Hall's research supported in part by a grant from the Simons Foundation}

\begin{abstract}
We compute the Brown measure of $x_{0}+i\sigma_{t}$, where $\sigma_{t}$ is a
free semicircular Brownian motion and $x_{0}$ is a freely independent
self-adjoint element that is not a multiple of the identity. The Brown measure
is supported in the closure of a certain bounded region $\Omega_{t}$ in the
plane. In $\Omega_{t},$ the Brown measure is absolutely continuous with
respect to Lebesgue measure, with a density that is constant in the vertical
direction. Our results refine and rigorize results of Janik, Nowak, Papp,
Wambach, and Zahed and of Jarosz and Nowak in the physics literature.

We also show that pushing forward the Brown measure of $x_{0}+i\sigma_{t}$ by
a certain map $Q_{t}:\Omega_{t}\rightarrow\mathbb{R}$ gives the distribution
of $x_{0}+\sigma_{t}.$ We also establish a similar result relating the Brown
measure of $x_{0}+i\sigma_{t}$ to the Brown measure of $x_{0}+c_{t}$, where
$c_{t}$ is the free circular Brownian motion.

\end{abstract}
\maketitle
\tableofcontents

\section{Introduction}

\subsection{Sums of independent random matrices}

A fundamental problem in random matrix theory is to understand the eigenvalue
distribution of sums of independent random matrices. When the random matrices
are Hermitian, the subordination method, introduced by Voiculescu \cite{Voi3}
and further developed by Biane \cite{BianeFreeIncr} and Voiculescu \cite{Voi4}
gives a powerful method of analyzing the problem in the setting of free
probability. (See Section \ref{bianeResult.sec} for a brief discussion of the
subordination method.) For related results in the random matrix setting, see,
for example, works of Pastur and Vasilchuk \cite{PV} and of Kargin
\cite{Kargin}.

A natural next step would be to consider non-normal random matrices of the
form $X+iY$ where $X$ and $Y$ are independent Hermitian random matrices.
Although a general framework has been developed for analyzing combinations of
freely independent elements in free probability (see works of Belinschi, Mai,
and Speicher \cite{BMS} and Belinschi, \'{S}niady, and Speicher \cite{BSS}),
it does not appear to be easy to apply this framework to get analytic results
about the $X+iY$ case.

The $X+iY$ problem has been analyzed at a nonrigorous level in the physics
literature. A highly cited paper of Stephanov \cite{stephanov} uses the case
in which $X$ is Bernoulli and $Y$ is GUE to provide a model of QCD. In the
case that $Y$ is GUE, work of Janik, Nowak, Papp, Wambach, and Zahed
\cite{JNPWZ} identified the domain into which the eigenvalues should cluster
in the large-$N$ limit. Then work of Jarosz and Nowak \cite{JN1,JN2} analyzed
the limiting eigenvalue distribution for general $X$ and $Y,$ with explicit
computations of examples when $Y$ is GUE and $X$ has various distributions
\cite[Section 6.1]{JN1}.

In this paper, we \textit{compute} the Brown measure of $x_{0}+i\sigma_{t},$
where $\sigma_{t}$ is a semicircular Brownian motion and $x_{0}$ is an
arbitrary self-adjoint element freely independent of $\sigma_{t}.$ This Brown
measure is the natural candidate for the limiting eigenvalue distribution of
random matrices of the form $X+iY$ where $X$ and $Y$ are independent and $Y$
is GUE. We also \textit{relate} the Brown measure of $x_{0}+i\sigma_{t}$ to
the distribution of $x_{0}+\sigma_{t}$ (without the factor of $i$). Our
computation of the Brown measure of $x_{0}+i\sigma_{t}$ refines and rigorizes
the results of \cite{JNPWZ} and \cite{JN1,JN2}, using a different method,
while the relationship between $x_{0}+i\sigma_{t}$ and $x_{0}+\sigma_{t}$ is a
new result. See Section \ref{previous.sec} for further discussion of these
works and Sections \ref{jnpwz.sec} and \ref{jn.sec} for a detailed comparison
of results.

Our work extends that of Ho and Zhong \cite{HZ}, which (among other results)
computes the Brown measure of $x_{0}+i\sigma_{t}$ in the case $x_{0}%
=y_{0}+\tilde{\sigma}_{t},$ where $\tilde{\sigma}_{t}$ is another semicircular
Brownian motion, freely independent of both $y_{0}$ and $\sigma_{t}.$ In this
case, $x_{0}+i\sigma_{t}$ has the form of $y_{0}+c_{2t},$ where $c_{t}$ is a
free circular Brownian motion.

Our results are based on the PDE method introduced in \cite{DHKBrown}. This
method has been used in subsequent works by Ho and Zhong \cite{HZ}, Demni and
Hamdi \cite{DH}, and Hall and Ho \cite{HHmult}, and is discussed from the
physics point of view by Grela, Nowak, and Tarnowski in \cite{GNT}. See also
the expository article \cite{PDEmethods} of the first author for an
introduction to the PDE method. Similar PDEs, in which the regularization
parameter in the construction of the Brown measure becomes a variable in the
PDE, have appeared in the physics literature in the work of Burda, Grela,
Nowak, Tarnowski, and Warcho\l \ \cite{BGNTW1,BGNTW2}.

Since this article was posted on the arXiv, three papers have appeared that
extend our results have appeared. The paper \cite{Ho1} of Ho examines in
detail the case in which $x_{0}$ is the sum of a self-adjoint element and a
freely independent semicircular element, so that $x_{0}+i\sigma_{t}$ becomes
the sum of a self-adjoint element and a freely independent elliptic element.
The paper \cite{Ho2} extends the results of the present paper by allowing
$x_{0}$ to be unbounded. Finally, the paper \cite{Zhong2} of Zhong analyzes
the Brown measure of $x_{0}+g,$ where $g$ is a twisted elliptic element and
$x_{0}$ is freely independent of $g$ but otherwise arbitrary. In the case that
$x_{0}$ is self-adjoint and $g$ is an imaginary multiple of a semicircular
element, Zhong's results reduce to ours.

\subsection{Statement of results\label{statements.sec}}

Let $\sigma_{t}$ be a semicircular Brownian motion living in a tracial von
Neumann algebra $(\mathcal{A},\tau)$ and let $x_{0}$ be a self-adjoint element
of $\mathcal{A}$ that is freely independent of every $\sigma_{t},$ $t>0.$ (In
particular, $x_{0}$ is a \textit{bounded} self-adjoint operator.) Throughout
the paper, we let $\mu$ be the distribution\ of $x_{0},$ that is, the unique
compactly supported probability measure on $\mathbb{R}$ such that%
\begin{equation}
\int_{\mathbb{R}}x^{n}~d\mu(x)=\tau(x_{0}^{n}),\quad\text{for all }%
n\in\mathbb{N}. \label{muDef}%
\end{equation}
Our goal is then to compute the Brown measure of the element%
\begin{equation}
x_{0}+i\sigma_{t} \label{x0PlusSigma}%
\end{equation}
in $\mathcal{A}.$ (See Section \ref{Brown.sec} for the definition of the Brown
measure.) Throughout the paper, we impose the following standing assumption
about $\mu.$

\begin{assumption}
\label{notDirac.assumption}The measure $\mu$ is not a $\delta$-measure, that
is, not supported at a single point.
\end{assumption}

Of course, the case in which $\mu$ is a $\delta$-measure is not hard to
analyze---in that case, $x_{0}+i\sigma_{t}$ has the form $a+i\sigma_{t},$ for
some constant $a\in\mathbb{R},$ so that the Brown measure is a semicircular
distribution on a vertical segment through $a.$ But this case is
\textit{different}; in all other cases, the Brown measure is absolutely
continuous with respect to the Lebesgue measure on a two-dimensional region in
the plane. Thus, our main results do not hold as stated in the case that $\mu$
is a $\delta$-measure.

The element (\ref{x0PlusSigma}) is the large-$N$ limit of the following random
matrix model. Let $Y^{N}$ be an $N\times N$ random variable distributed
according to the Gaussian unitary ensemble. Let $X^{N}$ be a sequence of
self-adjoint random matrices that are independent of $Y^{N}$ and whose
eigenvalue distributions converge almost surely to the law $\mu$ of $x_{0}.$
(The $X^{N}$'s may, for example, be chosen to be deterministic diagonal
matrices, which is the case in all the simulations shown in this paper.) Then
the random matrices
\begin{equation}
X^{N}+i\sqrt{t}Y^{N} \label{approxMatrices}%
\end{equation}
will converge in $\ast$-distribution to $x_{0}+i\sigma_{t}.$

In this paper we compute the Brown measure of $x_{0}+i\sigma_{t}.$ This Brown
measure is the natural candidate for the limiting empirical eigenvalue
distribution of the random matrices in (\ref{approxMatrices}). Our main
results are summarized briefly in the following theorem.

\begin{theorem}
\label{intro.thm}For each $t>0,$ there exists a continuous function
$b_{t}:\mathbb{R}\rightarrow\lbrack0,\infty)$ such that the following results
hold. Let%
\[
\Omega_{t}=\left\{  \left.  a+ib\in\mathbb{C}\right\vert ~\left\vert
b\right\vert <b_{t}(a)\right\}  .
\]
Then the Brown measure of $x_{0}+i\sigma_{t}$ is supported on the closure of
$\Omega_{t}$ and $\Omega_{t}$ itself is a set of full Brown measure. Inside
$\Omega_{t}$, the Brown measure is absolutely continuous with a density that
is \emph{constant in the vertical directions}. Specifically, the density
$w_{t}(a+ib)$ is independent of $b$ in $\Omega_{t}$ and has the form%
\[
w_{t}(a+ib)=\frac{1}{2\pi t}\left(  \frac{da_{0}^{t}(a)}{da}-\frac{1}%
{2}\right)  ,\quad a+ib\in\Omega_{t}.
\]
for a certain function $a_{0}^{t}.$
\end{theorem}

See Figures \ref{bernoullidensity.fig} and \ref{quaddensity.fig}.%

\begin{figure}[ptb]%
\centering
\includegraphics[
height=3.1687in,
width=3.5276in
]%
{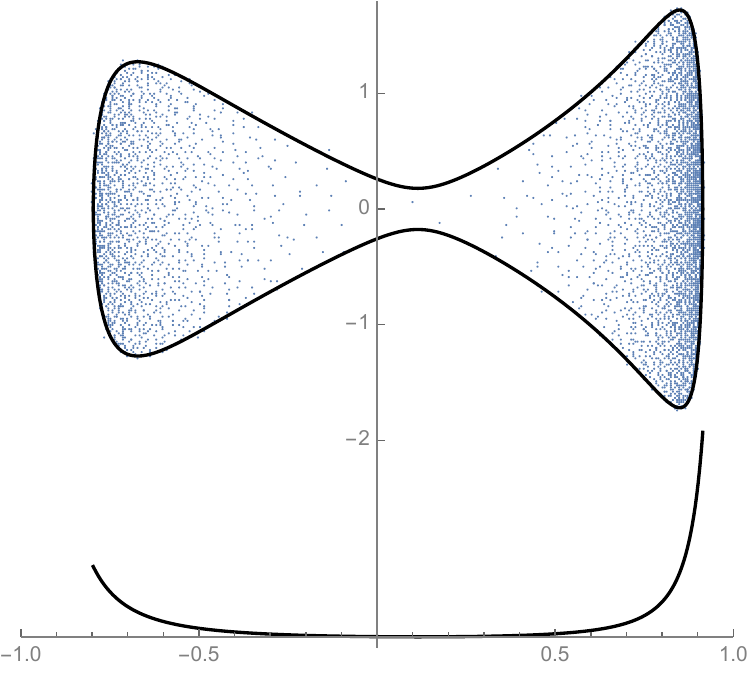}%
\caption{The top of the figure shows the domain $\Omega_{t}$ for the case
$\mu=\frac{1}{3}\delta_{-1}+\frac{2}{3}\delta_{1}$ and $t=1.05,$ together with
a simulation of the corresponding random matrix model. The bottom of the
figure shows the density (in $\Omega_{t}$) of the Brown measure as a function
of $a.$}%
\label{bernoullidensity.fig}%
\end{figure}

\begin{figure}[ptb]%
\centering
\includegraphics[
height=2.284in,
width=3.5276in
]%
{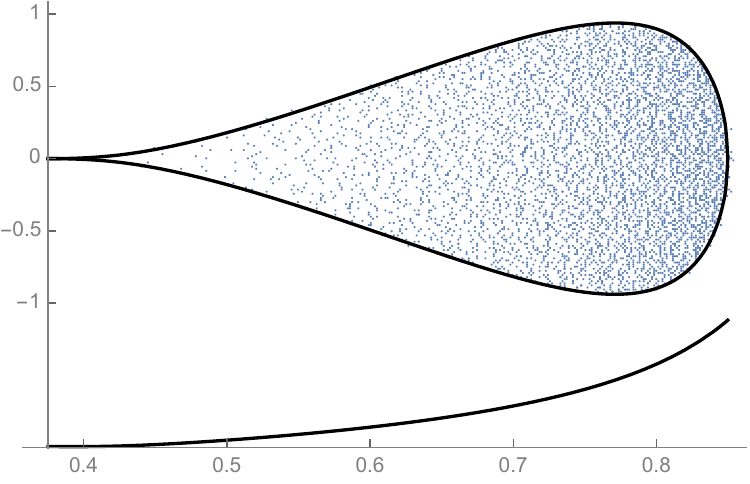}%
\caption{The top of the figure shows the domain $\Omega_{t}$ for the case in
which $\mu$ has density $3x^{2}$ on $[0,1]$ and $t=1/4,$ together with a
simulation of the corresponding random matrix model. The bottom of the figure
shows the density (in $\Omega_{t}$) of the Brown measure as a function of
$a.$}%
\label{quaddensity.fig}%
\end{figure}

We now describe how to compute the functions $b_{t}$ and $a_{0}^{t}$ in
Theorem \ref{intro.thm}. Recall that $\mu$ is the law of $x_{0},$ as in
(\ref{muDef}). We then fix $t>0$ and consider two equations:%
\begin{align}
\int_{\mathbb{R}}\frac{1}{(a_{0}-x)^{2}+v^{2}}~d\mu(x)  &  =\frac{1}%
{t}\label{p0EqIntro}\\
\int_{\mathbb{R}}\frac{x}{(a_{0}-x)^{2}+v^{2}}~d\mu(x)  &  =\frac{a}{t},
\label{p1EqIntro}%
\end{align}
where we look for a solution with $v>0$ and $a_{0}\in\mathbb{R}.$ We will show
in Section \ref{sect:surjectivity} that there can be at most one such pair
$(v,a_{0})$ for each $a\in\mathbb{R}.$ If, for a given $a\in\mathbb{R},$ we
can find $v>0$ and $a_{0}\in\mathbb{R}$ solving these equations, we set%
\begin{equation}
a_{0}^{t}(a)=a_{0} \label{aotDef}%
\end{equation}
and%
\begin{equation}
b_{t}(a)=2v. \label{btDef}%
\end{equation}
If, on the other hand, no solution exists, we set $b_{t}(a)=0$ and leave
$a_{0}^{t}(a)$ undefined. (If $b_{t}(a)=0$, there are no points of the form
$a+ib$ in $\Omega_{t}$ and so the density of the Brown measure is undefined.)

The equations (\ref{p0EqIntro}) and (\ref{p1EqIntro}) can be solved explicitly
for some simple choices of $\mu$, as shown in Section \ref{examples.sec}. For
any reasonable choice of $\mu,$ the equations can be easily solved numerically.

We now explain a connection between the Brown measure of $x_{0}+i\sigma_{t}$
and two other models. In addition to the semicircular Brownian motion
$\sigma_{t},$ we consider also a circular Brownian motion $c_{t}.$ This may be
constructed as%
\[
c_{t}=\sigma_{t/2}+i\tilde{\sigma}_{t/2},
\]
where $\sigma_{\cdot}$ and $\tilde{\sigma}_{\cdot}$ are two freely independent
semicircular Brownian motions. We now describe a remarkable direct connection
between the Brown measure of $x_{0}+i\sigma_{t}$ and the Brown measure of
$x_{0}+c_{t},$ and a similar direct connection between the the Brown measure
of $x_{0}+i\sigma_{t}$ and the law of $x_{0}+\sigma_{t}.$ We remark that a
fascinating indication of a connection between the behavior of $x_{0}%
+\sigma_{t}$ and the behavior of $x_{0}+i\sigma_{t}$ were given previously in
the work of Janik, Nowak, Papp, Wambach, and Zahed, discussed in Section
\ref{jnpwz.sec}. Note that since $\sigma_{t}$ has the same law as
$\sigma_{t/2}+\tilde{\sigma}_{t/2},$ we can describe the three random
variables in question as
\begin{align*}
x_{0}+\sigma_{t}  &  \equiv x_{0}+\sigma_{t/2}+\tilde{\sigma}_{t/2}\\
x_{0}+c_{t}  &  \equiv x_{0}+\sigma_{t/2}+i\tilde{\sigma}_{t/2}\\
x_{0}+i\sigma_{t}  &  \equiv x_{0}+i\sigma_{t/2}+i\tilde{\sigma}_{t/2},
\end{align*}
where the notation $A\equiv B$ means that $A$ and $B$ have the same $\ast
$-distribution and therefore the same Brown measure.

The Brown measure of $x_{0}+c_{t}$ was computed by the second author and Zhong
in \cite{HZ}. They also established that the Brown measure of $x_{0}+c_{t}$ is
related to the law of $x_{0}+\sigma_{t}$. We then show that the Brown measure
of $x_{0}+i\sigma_{t}$ is related to the Brown measure of $x_{0}+c_{t}.$ By
combining our this last result with what was shown in \cite[Prop. 3.14]{HZ},
we obtain the following result.

\begin{theorem}
The Brown measure of $x_{0}+c_{t}$ is supported in the closure of a certain
domain $\Lambda_{t}$ identified in \cite{HZ}. There is a homeomorphism $U_{t}$
of $\overline{\Lambda}_{t}$ onto $\overline{\Omega}_{t}$ with the property
that the push-forward of $\mathrm{Brown}(x_{0}+c_{t})$ under $U_{t}$ is equal
to $\mathrm{Brown}(x_{0}+i\sigma_{t}).$ Furthermore, there is a continuous map
$Q_{t}:\overline{\Omega}_{t}\rightarrow\mathbb{R}$ such that the push-forward
of $\mathrm{Brown}(x_{0}+i\sigma_{t})$ under $Q_{t}$ is the law of
$x_{0}+\sigma_{t},$ as computed by Biane.
\end{theorem}

The maps $U_{t}$ and $Q_{t}$ are described in Sections \ref{sect:surjectivity}
and \ref{pushforward.sec}, respectively. The map $U_{t}$ has the property that
vertical line segments in $\overline{\Lambda}_{t}$ map linearly to vertical
line segments in $\overline{\Omega}_{t},$ while the map $Q_{t}$ has the
property that vertical line segments in $\overline{\Omega}_{t}$ map to single
points in $\mathbb{R}.$ (See Figures \ref{utmap.fig} and \ref{qmap.fig}.) The
map $Q_{t}$ is computed by first applying the inverse of the map $U_{t}$ and
then applying the map denoted as $\Psi_{t}$ in Point 3 of Theorem 1.1 in
\cite{HZ}.%

\begin{figure}[ptb]%
\centering
\includegraphics[
height=2.9032in,
width=4.3448in
]%
{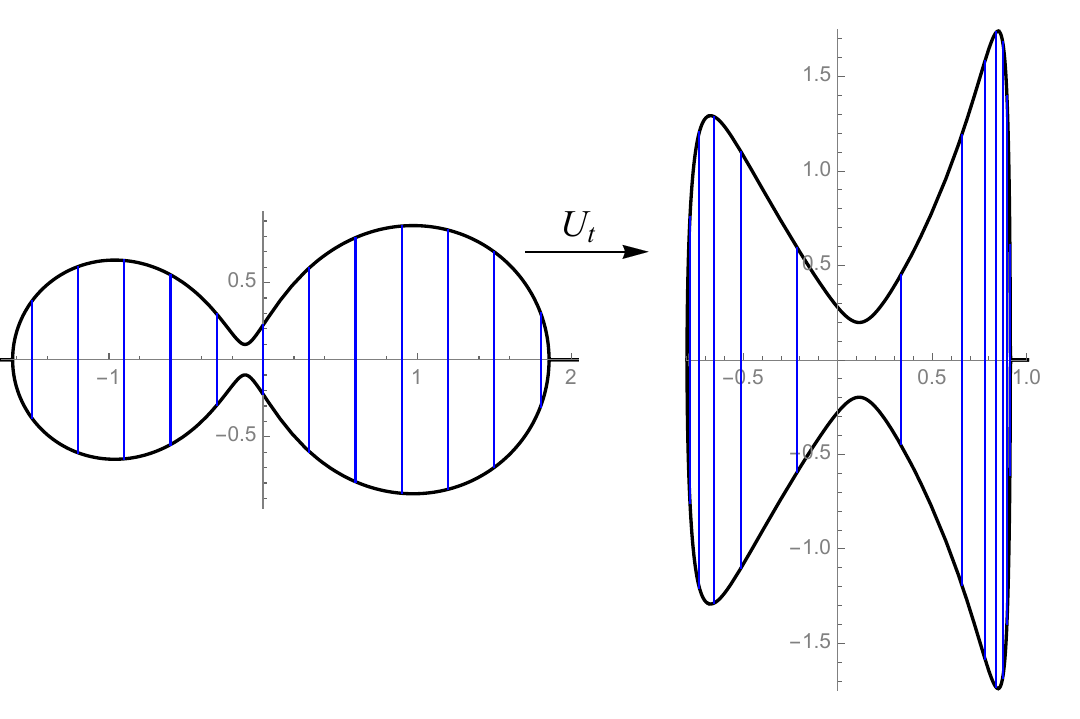}%
\caption{A visualization of the map $U_{t}:\overline{\Lambda}_{t}%
\rightarrow\overline{\Omega}_{t}.$ The map takes vertical segments in
$\Lambda_{t}$ linearly to vertical segments in $\Omega_{t}.$ Shown for
$\mu=\frac{1}{3}\delta_{-1}+\frac{2}{3}\delta_{1}$ and~$t=1.05.$}%
\label{utmap.fig}%
\end{figure}

\begin{figure}[ptb]%
\centering
\includegraphics[
height=2.7769in,
width=4.3811in
]%
{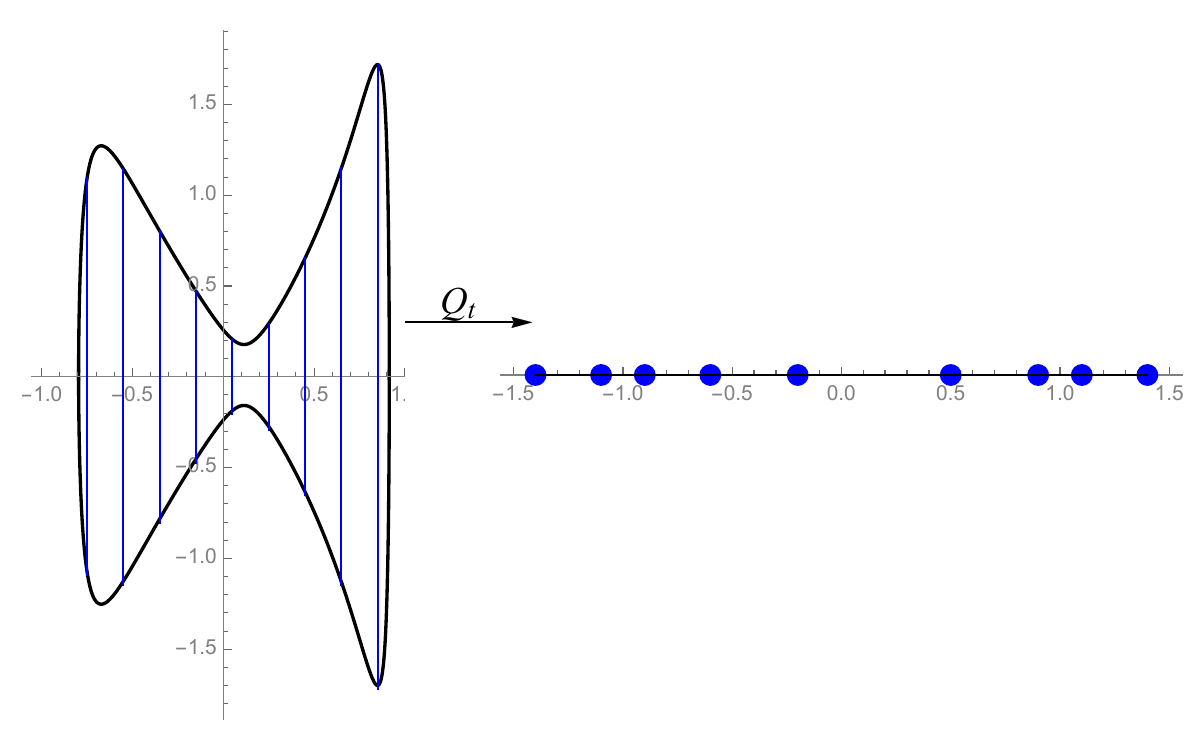}%
\caption{A visualization of the map $Q_{t}:\overline{\Omega}_{t}%
\rightarrow\mathbb{R}.$ The map takes vertical segments in $\overline{\Omega
}_{t}$ to single points in $\mathbb{R}.$ Shown for $\mu=\frac{1}{3}\delta
_{-1}+\frac{2}{3}\delta_{1}$ and~$t=1.05.$}%
\label{qmap.fig}%
\end{figure}

\subsection{Method of proof}

Our proofs are based on the PDE method developed in \cite{DHKBrown} and used
also in \cite{HZ} and \cite{DH}. (See also \cite{PDEmethods} for a gentle
introduction to the method.) For any operator $A$ in a tracial von Neumann
algebra $(\mathcal{A},\tau),$ the Brown measure of $A,$ denoted
$\mathrm{Brown}(A),$ may be computed as follows. (See Section \ref{Brown.sec}
for more details.) Let%
\[
S(\lambda,\varepsilon)=\tau\lbrack\log((A-\lambda)^{\ast}(A-\lambda
)+\varepsilon)]
\]
for $\varepsilon>0.$ Then the limit%
\[
s(\lambda):=\lim_{\varepsilon\rightarrow0^{+}}S(\lambda,\varepsilon)
\]
exists as a subharmonic function. The Brown measure is then defined as%
\[
\mathrm{Brown}(A)=\frac{1}{4\pi}\Delta s,
\]
where the Laplacian is computed in the distributional sense. The general
theory then guarantees that $\mathrm{Brown}(A)$ is a probability measure
supported on the spectrum of $A.$ (The closed support of $\mathrm{Brown}(A)$
can be a proper subset of the spectrum of $A.$)

In our case, we take $A=x_{0}+i\sigma_{t}$, so that $S$ also depends on $t.$
Thus, we consider the functions%
\begin{equation}
S(t,\lambda,\varepsilon)=\tau\lbrack\log((x_{0}+i\sigma_{t}-\lambda)^{\ast
}(x_{0}+i\sigma_{t}-\lambda)+\varepsilon)] \label{Sintro}%
\end{equation}
and%
\[
s_{t}(\lambda)=\lim_{\varepsilon\rightarrow0^{+}}S(t,\lambda,\varepsilon).
\]
Then%
\[
\mathrm{Brown}(x_{0}+i\sigma_{t})=\frac{1}{4\pi}\Delta s_{t}(\lambda),
\]
where the Laplacian is taken with respect to $\lambda$ with $t$ fixed.

Our first main result (Theorem \ref{PDE.thm}) is that the function $S$ in
(\ref{Sintro}) satisfies a first-order nonlinear PDE of Hamilton--Jacobi type,
given in Theorem \ref{PDE.thm}. Our goal is then to solve the PDE for
$S(t,\lambda,\varepsilon)$, evaluate the solution in the limit $\varepsilon
\rightarrow0,$ and then take the Laplacian with respect to $\lambda.$ We use
two different approaches to this goal, one approach outside a certain domain
$\Omega_{t}$ and a different approach inside $\Omega_{t},$ where the Brown
measure turns out to be zero outside $\Omega_{t}$ and nonzero inside
$\Omega_{t}.$ See Sections \ref{outside.sec} and \ref{inside.sec}.

\subsection{Comparison to previous results\label{previous.sec}}

A different approach to the problem was previously developed in the physics
literature by Jarosz and Nowak \cite{JN1,JN2}. Using linearization and
subordination functions, they propose an algorithm for computing the Brown
measure of $H_{1}+iH_{2},$ where $H_{1}$ and $H_{2}$ are arbitrary freely
independent Hermitian elements. (See, specifically, Eqs. (75)--(80) in
\cite{JN2}.) Section 6 of \cite{JN1} presents examples in which one of $H_{1}$
and $H_{2}$ is semicircular and the other has various distributions.

Although the method of \cite{JN1,JN2} is not rigorous as written, it is
possible that the strategy used there could be made rigorous using the
general\ framework developed by Belinschi, Mai, and Speicher \cite{BMS}. (See,
specifically, the very general algorithm in Section 4 of \cite{BMS}. See also
\cite{BSS} for further rigorous developments in this direction.) We emphasize,
however, that it would require considerable effort to get analytic results for
the $H_{1}+iH_{2}$ case from the general algorithm of \cite{BMS}. In any case,
we show in Section \ref{jn.sec} that our results are compatible with those
obtained by the algorithm of Jarosz and Nowak.

In addition to presenting a rigorous argument, we provide information about
the Brown measure of $x_{0}+i\sigma_{t}$ that is not found in \cite{JN1,JN2}.
First, we highlight the crucial result that the density of the Brown measure,
inside its support, is always constant in the vertical direction. Although
this result certainly follows from the algorithm of Jarosz and Nowak (and is
reflected in the examples in \cite[Sect. 6]{JN1}), it is not explicitly stated
in their work. Second, we give significantly more explicit formulas for the
support of the Brown measure and for its density when $x_{0}$ is arbitrary.
Third, we obtain (Section \ref{pushforward.sec}) a direct relationship between
the Brown measure of $x_{0}+i\sigma_{t}$ and the distribution of $x_{0}%
+\sigma_{t}$ that is not found in \cite{JN1} or \cite{JN2}.

Meanwhile, in Section \ref{domain.sec}, we also confirm a separate,
nonrigorous argument of Janik, Nowak, Papp, Wambach, and Zahed predicting the
domain on which the Brown measure is supported.

Finally, as mentioned previously, Section 3 of the paper \cite{HZ} of the
second author and Zhong computed the Brown measure of $y_{0}+c_{t},$ where
$c_{t}$ is the free circular Brownian motion (large-$N$ limit of the Ginibre
ensemble). Now, $c_{t}$ can be constructed as $c_{t}=\tilde{\sigma}%
_{t/2}+i\sigma_{t/2},$ where $\sigma_{\cdot}$ and $\tilde{\sigma}_{\cdot}$ are
two freely independent semicircular Brownian motions. Thus, the results of the
present paper in the case where $x_{0}$ is the sum of a self-adjoint element
$y_{0}$ and a freely independent semicircular element fall under the results
of \cite{HZ}. But actually, the connection between the present paper and
\cite{HZ} is deeper than that. For any choice of $x_{0},$ the region
$\Lambda_{t}$ in which the Brown measure of $x_{0}+c_{t}$ is supported shows
up in the computation of the Brown measure of $x_{0}+i\sigma_{t},$ as the
\textquotedblleft domain in the $\lambda_{0}$-plane\textquotedblright%
\ (Section \ref{lambda0Domain.sec}). And then we show that the Brown measure
of $x_{0}+i\sigma_{t}$ is the pushforward of the Brown measure of $x_{0}%
+c_{t}$ under a certain map (Section \ref{pushforward.sec}). Thus, one of the
notable aspect of the results of the present paper is the way they illuminate
the deep connections between $x_{0}+c_{t}$ and $x_{0}+i\sigma_{t}$.

\section{The Brown measure formalism\label{Brown.sec}}

We present here general results about the Brown measure. For more information,
the reader is referred to the original paper \cite{Br} of Brown and to Chapter
11 of the monograph of Mingo and Speicher \cite{MS}.

Let $(\mathcal{A},\tau)$ be a \textbf{tracial von Neumann algebra}, that is, a
finite von Neumann algebra $\mathcal{A}$ with a faithful, normal, tracial
state $\tau:\mathcal{A}\rightarrow\mathbb{C}.$ Thus, $\tau$ is a norm-1 linear
functional with the properties that $\tau(A^{\ast}A)>0$ for all nonzero
elements of $\mathcal{A}$ and that $\tau(AB)=\tau(BA)$ for all $A,B\in
\mathcal{A}.$ For any $A\in\mathcal{A},$ we define a function $S$ by
\[
S(\lambda,\varepsilon)=\tau\lbrack\log((A-\lambda)^{\ast}(A-\lambda
)+\varepsilon)],\quad\lambda\in\mathbb{C},~\varepsilon>0.
\]
It is known that
\[
s(\lambda):=\lim_{\varepsilon\rightarrow0^{+}}S(\lambda,\varepsilon)
\]
exists as a subharmonic function on $\mathbb{C}.$ Then the \textbf{Brown
measure} of $A$ is defined in terms of the distributional Laplacian of $s$:%
\[
\mathrm{Brown}(A)=\frac{1}{4\pi}\Delta s.
\]

The motivation for this definition comes from the case in which $\mathcal{A}$
is the algebra of all $N\times N$ matrices and $\tau$ is the normalized trace
($1/N$ time ordinary trace). In this case, if $A$ has eigenvalues $\lambda
_{1},\ldots,\lambda_{N}$ (counted with their algebraic multiplicities), then
the function $s$ may be computed as%
\[
s(\lambda)=\frac{2}{N}\sum_{j=1}^{N}\log\left\vert \lambda-\lambda
_{j}\right\vert .
\]
That is to say, $s$ is $2/N$ time the logarithm of the absolute value of the
characteristic polynomial of $A.$ Since $\frac{1}{2\pi}\log\left\vert
\lambda\right\vert $ is the Green's function for the Laplacian on the plane,
we find that%
\[
\mathrm{Brown}(A)=\frac{1}{N}\sum_{j=1}^{N}\delta_{\lambda_{j}}.
\]
Thus, the Brown measure of a matrix is just its empirical eigenvalue distribution.

If a sequence of random matrices $A^{N}$ converges in $\ast$-distribution to
an element $A$ in a tracial von Neumann algebra, one generally expects that
the empirical eigenvalue distribution of $A^{N}$ will converge almost surely
the Brown measure of $A.$ But such a result does not \textit{always} hold and
it is a hard technical problem to prove that it does in specific examples.
Works of Girko \cite{GirkoCircular}, Bai \cite{Bai}, and Tao and Vu \cite{TV}
(among others)\ on the circular law provide techniques for establish such
convergence results, while a somewhat different approach to such problems was
developed by Guionnet, Krishnapur, and Zeitouni \cite{GKZsingleRing}.

\section{The differential equation for $S$}

Let $\sigma_{t}$ be a free semicircular Brownian motion and let $x_{0}$ be a
Hermitian element freely independent of each $\sigma_{t},$ $t>0$. The main
result of this section is the following.

\begin{theorem}
\label{PDE.thm}Let
\[
S(t,\lambda,\varepsilon)=\tau\lbrack\log((x_{0}+i\sigma_{t}-\lambda)^{\ast
}(x_{0}+i\sigma_{t}-\lambda)+\varepsilon)]\quad\lambda\in\mathbb{C}%
,~\varepsilon>0
\]
and write $\lambda$ as $\lambda=a+ib$ with $a,b\in\mathbb{R}.$ Then the
function $S$ satisfies the PDE
\begin{equation}
\frac{\partial S}{\partial t}=\frac{1}{4}\left(  \left(  \frac{\partial
S}{\partial a}\right)  ^{2}-\left(  \frac{\partial S}{\partial b}\right)
^{2}\right)  +\varepsilon\left(  \frac{\partial S}{\partial\varepsilon
}\right)  ^{2} \label{thePDE}%
\end{equation}
subject to the initial condition%
\[
S(0,\lambda,\varepsilon)=\tau\lbrack\log((x_{0}-\lambda)^{\ast}(x_{0}%
-\lambda)+\varepsilon)].
\]

\end{theorem}

We use the notation
\begin{align*}
x_{t}  &  :=x_{0}+i\sigma_{t}\\
x_{t,\lambda}  &  :=x_{t}-\lambda.
\end{align*}
Then the free SDE's of $x_{t,\lambda}$ and $x_{t,\lambda}^{\ast}$ are
\begin{equation}
dx_{t,\lambda}=i\,d\sigma_{t},\qquad dx_{t,\lambda}^{\ast}=-i\,d\sigma_{t}.
\label{eq:dxtlambda}%
\end{equation}

The main tool of this section is the free It\^{o} formula. The following
theorem is a simpler form of Theorem 4.1.2 of \cite{BS1} which states the free
It\^{o} formula. The form of the It\^{o} formula used here is similar to what
is in Lemma 2.5 and Lemma 4.3 of \cite{KempLargeN}. For a \textquotedblleft
functional\textquotedblright\ form of these free It\^{o} formulas, see Section
4.3 of \cite{Nik}.

\begin{theorem}
\label{thm:Ito} Let $(\mathcal{A}_{t})_{t\geq0}$ be a filtration such that
$\sigma_{t}\in\mathcal{A}_{t}$ for all $t$ and $\sigma_{t}-\sigma_{s}$ is free
with $\mathcal{A}_{s}$ for all $s\leq t$. Also let $f_{t}$, $g_{t}$ be two
free It\^{o} processes satisfying the free SDEs
\begin{align}
df_{t}  &  =\sum_{k=1}^{n}a_{t}^{k}\,~d\sigma_{t}~\,b_{t}^{k}+c_{t}%
\,~dt\label{SDEdft}\\
dg_{t}  &  =\sum_{k=1}^{n}\tilde{a}_{t}^{k}\,~d\sigma_{t}~\,\tilde{b}_{t}%
^{k}+\tilde{c}_{t}\,~dt. \label{SDEdgt}%
\end{align}
for some continuous adapted processes $\{a_{t}^{k},b_{t}^{k},c_{t},\tilde
{a}_{t}^{k},\tilde{b}_{t}^{k},\tilde{c}_{t}\}_{k=1}^{n}.$ Then $f_{t}g_{t}$
satisfies the free SDE
\begin{equation}
d(f_{t}g_{t})=\sum_{k=1}^{n}(a_{t}^{k}\,~d\sigma_{t}~\,b_{t}^{k}g_{t}%
+f_{t}\tilde{a}_{t}^{k}~\,d\sigma_{t}~\,\tilde{b}_{t}^{k})+\left(  c_{t}%
g_{t}+f_{t}\tilde{c}_{t}+\sum_{j,k=1}^{n}\tau\lbrack b_{t}^{k}\tilde{a}%
_{t}^{j}]a_{t}^{k}\tilde{b}_{t}^{j}\right)  \,dt. \label{Dftgt}%
\end{equation}
That is, $d(f_{t}g_{t})$ can be informally computed using the free It\^{o}
product rule:
\[
d(f_{t}g_{t})=df_{t}\,g_{t}+f_{t}\,dg_{t}+df_{t}~dg_{t},
\]
where $df_{t}~dg_{t}$ is computed using the rules
\begin{align}
\,d\sigma_{t}~\,\theta_{t}~\,dt  &  =dt~\theta_{t}~d\sigma_{t}=dt~d\theta
_{t}~dt=0,\label{Ito3}\\
d\sigma_{t}\,~\theta_{t}\,~d\sigma_{t}  &  =\tau\lbrack\theta_{t}]\,dt
\label{Ito.dxdx}%
\end{align}
for any continuous adapted process $\theta_{t}.$

Furthermore, if a process $f_{t}$ satisfies an SDE as in (\ref{SDEdft}), then
$\tau\lbrack f_{t}]$ satisfies%
\[
d\tau\lbrack f_{t}]=\tau\lbrack c_{t}]~dt.
\]
This result can be expressed informally as saying $d$ commutes with $\tau$ and
that
\begin{equation}
\tau\lbrack\theta_{t}~d\sigma_{t}]=0 \label{Ito.dx}%
\end{equation}
for any continuous adapted process $\theta_{t}.$
\end{theorem}

The theorem stated above is applicable to our current situation. Let
$\mathcal{A}_{0}$ be the von Neumann algebra generated by $x_{0}$, and
$\mathcal{B}_{t}$ be the von Neumann algebra generated by $\{\sigma_{r}:r\leq
t\}$. Then we apply Theorem \ref{thm:Ito} with $\mathcal{A}_{t}=\mathcal{A}%
_{0}\ast\mathcal{B}_{t}$, the reduced free product of $\mathcal{A}_{0}$ and
$\mathcal{B}_{t}$.

We shall use the free It\^{o} formula to compute a partial differential
equation that $S$ satisfies. Our strategy is to first do a power series
expansion of the logarithm and then apply the free It\^{o} formula to compute
the partial derivative of the powers of $x_{t,\lambda}^{\ast}x_{t,\lambda}$
with respect to $t$. We start by computing the time derivatives of
$\tau\lbrack(x_{t,\lambda}^{\ast}x_{t,\lambda})^{n}].$

\begin{lemma}
\label{Dt:moments} We have
\begin{equation}
\frac{\partial}{\partial t}\tau\lbrack(x_{t,\lambda}^{\ast}x_{t,\lambda})]=1.
\label{DtN1}%
\end{equation}
When $n\geq2$,
\begin{equation}%
\begin{split}
\frac{\partial}{\partial t}\tau\lbrack(x_{t,\lambda}^{\ast}x_{t,\lambda}%
)^{n}]=  &  -\frac{n}{2}\sum_{m=1}^{n-1}\tau\lbrack x_{t,\lambda}^{\ast
}(x_{t,\lambda}^{\ast}x_{t,\lambda})^{m-1}]\tau\lbrack x_{t,\lambda}^{\ast
}(x_{t,\lambda}^{\ast}x_{t,\lambda})^{n-m-1}]\\
&  -\frac{n}{2}\sum_{m=1}^{n-1}\tau\lbrack x_{t,\lambda}(x_{t,\lambda}^{\ast
}x_{t,\lambda})^{m-1}]\tau\lbrack x_{t,\lambda}(x_{t,\lambda}^{\ast
}x_{t,\lambda})^{n-m-1}]\\
&  +n\sum_{m=1}^{n}\tau\lbrack(x_{t,\lambda}^{\ast}x_{t,\lambda})^{n-m}%
]\tau\lbrack(x_{t,\lambda}^{\ast}x_{t,\lambda})^{m-1}].
\end{split}
\label{DtExpand2}%
\end{equation}

\end{lemma}

\begin{proof}
For $n=1$, we apply the free It\^{o} formula to get
\[
d(x_{t,\lambda}^{\ast}x_{t,\lambda})=x_{t,\lambda}^{\ast}(i\,d\sigma
_{t})+(-i\,d\sigma_{t})x_{t,\lambda}+d\sigma_{t}\cdot d\sigma_{t}%
=ix_{t,\lambda}^{\ast}\,d\sigma_{t}-i\,d\sigma_{t}\,x_{t,\lambda}+dt
\]
which gives (\ref{DtN1}), after taking trace on both sides.

Now, we assume $n\geq2$. When we apply Theorem \ref{thm:Ito} repeatedly to
obtain results for the product of several free It\^{o} processes. When
computing $d\tau\lbrack(x_{t,\lambda}^{\ast}x_{t,\lambda})^{n}],$ we obtain
four types of terms, as follows.

\begin{enumerate}
\item Terms involving only one differential, either of $x_{t,\lambda}^{\ast}$
or of $x_{t,\lambda}$.

\item Terms involving two differentials of $x_{t,\lambda}$.

\item Terms involving two differentials of $x_{t,\lambda}^{*}$.

\item Terms involving a differential of $x_{t,\lambda}^{*}$ and a differential
of $x_{t,\lambda}$.
\end{enumerate}

We now compute $d\tau\lbrack(x_{t,\lambda}^{\ast}x_{t,\lambda})^{n}]$ by
moving the $d$ inside the trace and then applying Theorem \ref{thm:Ito}. By
(\ref{Ito.dx}), the terms in Point 1 will not contribute.

We then consider the terms in Point 2. There are exactly $n$ factors of
$x_{t,\lambda}$ in $(x_{t,\lambda}^{\ast}x_{t,\lambda})^{n}$. Since the terms
in Point 2 involve exactly two $dx_{t,\lambda}$'s, there are precisely
${\binom{n}{2}}$ terms in Point 2. For the purpose of computing these terms,
we label all of the $x_{t,\lambda}$'s by $x_{t,\lambda}^{(k)}$ for
$k=1,\ldots,n$. We view choosing two $x_{t,\lambda}$'s as first choosing an
$x_{t,\lambda}^{(i)}$, then another $x_{t,\lambda}^{(j)}$. We then cyclically
permute the factors until $dx_{t,\lambda}^{(i)}$ is at the beginning. Using
the free stochastic equation (\ref{eq:dxtlambda}) of $x_{t,\lambda}$, this
term has the form
\begin{align*}
&  \tau\lbrack dx_{t,\lambda}^{(i)}\,(x_{t,\lambda}^{\ast}x_{t,\lambda}%
)^{m}x_{t,\lambda}^{\ast}\,dx_{t,\lambda}^{(j)}\,(x_{t,\lambda}^{\ast
}x_{t,\lambda})^{n-m-2}x_{t,\lambda}^{\ast}]\\
&  =-\tau\lbrack(x_{t,\lambda}^{\ast}x_{t,\lambda})^{m}x_{t,\lambda}^{\ast
}]\tau\lbrack(x_{t,\lambda}^{\ast}x_{t,\lambda})^{n-m-2}x_{t,\lambda}^{\ast
}]~dt
\end{align*}
where $m=j-i-1\operatorname{mod}n$ and we omit the labeling of all
$x_{t,\lambda}$'s except $x_{t,\lambda}^{(i)}$ and $x_{t,\lambda}^{(j)}$.

If we then sum over all $j\neq i,$ we obtain
\[
-\sum_{m=0}^{n-2}\tau\lbrack x_{t,\lambda}^{\ast}(x_{t,\lambda}^{\ast
}x_{t,\lambda})^{m}]\tau\lbrack x_{t,\lambda}^{\ast}(x_{t,\lambda}^{\ast
}x_{t,\lambda})^{n-m-2}]~dt.
\]
Since this expression is independent of $i,$ summing over $i$ produces a
factor of $n$ in front. But then we have counted every term exactly twice,
since we can choose the $i$ first and then the $j$ or vice versa. Thus, the
sum of all the terms in Point 2 is
\begin{equation}
-\frac{n}{2}\sum_{m=0}^{n-2}\tau\lbrack(x_{t,\lambda}^{\ast}(x_{t,\lambda
}^{\ast}x_{t,\lambda})^{m}]\tau\lbrack(x_{t,\lambda}^{\ast}(x_{t,\lambda
}^{\ast}x_{t,\lambda})^{n-m-2}]~dt. \label{eq:Point2}%
\end{equation}
By a similar argument, the sum of all the terms in Point 3 is
\begin{equation}
-\frac{n}{2}\sum_{m=0}^{n-2}\tau\lbrack x_{t,\lambda}(x_{t,\lambda}^{\ast
}x_{t,\lambda})^{m}]\tau\lbrack x_{t,\lambda}(x_{t,\lambda}^{\ast}%
x_{t,\lambda})^{n-m-2}]~dt. \label{eq:Point3}%
\end{equation}

We now compute the terms in Point 4. We can cyclically permute the factors
until $dx_{t,\lambda}^{\ast}$ is at the beginning. Thus, each of the terms in
Point 4 can be written as
\begin{align}
&  \tau\lbrack dx_{t,\lambda}^{\ast}(x_{t,\lambda}x_{t,\lambda}^{\ast}%
)^{m}\,dx_{t,\lambda}\,(x_{t,\lambda}^{\ast}x_{t,\lambda})^{n-m-1}]\nonumber\\
&  =\tau\lbrack(x_{t,\lambda}^{\ast}x_{t,\lambda})^{m}]\tau\lbrack
(x_{t,\lambda}^{\ast}x_{t,\lambda})^{n-m-1}]~dt, \label{eq:Point4.iFixed}%
\end{align}
where $m=0,\ldots,n-1$. Now, there are a total of $n^{2}$ terms in Point 4,
but from (\ref{eq:Point4.iFixed}), we can see that there are only $n$
\textit{distinct} terms, each of which occurs $n$ times, so that the sum of
all terms from Point 4 is
\begin{equation}
n\sum_{m=0}^{n-1}\tau\lbrack(x_{t,\lambda}^{\ast}x_{t,\lambda})^{m}%
]\tau\lbrack(x_{t,\lambda}^{\ast}x_{t,\lambda})^{n-m-1}]~dt. \label{eq:Point4}%
\end{equation}

We now obtain (\ref{DtExpand2}) by adding (\ref{eq:Point2}), (\ref{eq:Point3}%
), and (\ref{eq:Point4}) and making a change of index.
\end{proof}

\begin{proposition}
\label{DS.Dt} The function $S$ satisfies the equation
\begin{align}
\frac{\partial S}{\partial t}  &  =\frac{1}{2}\tau\lbrack x_{t,\lambda
}(x_{t,\lambda}^{\ast}x_{t,\lambda}+\varepsilon)^{-1}]^{2}\nonumber\\
&  +\frac{1}{2}\tau\lbrack x_{t,\lambda}^{\ast}(x_{t,\lambda}^{\ast
}x_{t,\lambda}+\varepsilon)^{-1}]^{2}+\varepsilon\tau\lbrack(x_{t,\lambda
}^{\ast}x_{t,\lambda}+\varepsilon)^{-1}]^{2}. \label{eq:DS.Dt}%
\end{align}

\end{proposition}

\begin{proof}
We first show that (\ref{eq:DS.Dt}) holds for all $\varepsilon>\Vert
x_{t,\lambda}^{\ast}x_{t,\lambda}\Vert$. Let $\varepsilon>\Vert x_{t,\lambda
}^{\ast}x_{t,\lambda}\Vert$. We write $\log(x+\varepsilon)$ as $\log
\varepsilon+\log(1+x/\varepsilon)$ and then expand in powers of $x/\varepsilon
.$ We then substitute $x=x_{t,\lambda}^{\ast}x_{t,\lambda},$ and then apply
the trace term by term, giving
\begin{equation}
S(t,\lambda,\varepsilon)=\log\varepsilon+\sum_{n=1}^{\infty}\frac{(-1)^{n-1}%
}{n\varepsilon^{n}}\tau\lbrack(x_{t,\lambda}^{\ast}x_{t,\lambda})^{n}].
\label{eq:SPower}%
\end{equation}

We now wish to differentiate the right-hand side of (\ref{eq:SPower}) term by
term in $t$. We will see shortly that when we differentiate inside the sum,
the resulting series still converges for $\varepsilon>\Vert x_{t,\lambda
}^{\ast}x_{t,\lambda}\Vert$. Furthermore, since the map $t\mapsto x_{t}$ is
continuous in the operator norm topology, $\Vert x_{t}\Vert$ is a locally
bounded function of $t$. Hence, the series of derivatives converges locally
uniformly in $t$. This, together with the pointwise convergence of the
original series, will show that term-by-term differentiation is valid.

If we differentiate inside the sum in (\ref{eq:SPower}), we obtain
\begin{equation}
\sum_{n=1}^{\infty}\frac{(-1)^{n-1}}{n\varepsilon^{n}}\frac{\partial}{\partial
t}\tau\lbrack(x_{t,\lambda}^{\ast}x_{t,\lambda})^{n}]. \label{eq:SeriestoDiff}%
\end{equation}
By Lemma \ref{Dt:moments}, the above power series becomes
\begin{align}
&  \frac{1}{2}\sum_{n=2}^{\infty}\sum_{m=1}^{n-1}\frac{(-1)^{n}}%
{\varepsilon^{n}}\tau\lbrack x_{t,\lambda}(x_{t,\lambda}^{\ast}x_{t,\lambda
})^{m-1}]\tau\lbrack x_{t,\lambda}(x_{t,\lambda}^{\ast}x_{t,\lambda}%
)^{n-m-1}]\nonumber\\
&  +\frac{1}{2}\sum_{n=2}^{\infty}\sum_{m=1}^{n-1}\frac{(-1)^{n}}%
{\varepsilon^{n}}\tau\lbrack x_{t,\lambda}^{\ast}(x_{t,\lambda}^{\ast
}x_{t,\lambda})^{m-1}]\tau\lbrack x_{t,\lambda}^{\ast}(x_{t,\lambda}^{\ast
}x_{t,\lambda})^{n-m-1}]\nonumber\\
&  +\sum_{n=1}^{\infty}\sum_{m=1}^{n}\frac{(-1)^{n-1}}{\varepsilon^{n}}%
\tau\lbrack(x_{t,\lambda}^{\ast}x_{t,\lambda})^{n-m}]\tau\lbrack(x_{t,\lambda
}^{\ast}x_{t,\lambda})^{m-1}]. \label{threeTerms}%
\end{align}
Note that the constant term $1$ is in the last term in (\ref{threeTerms}). The
first term in (\ref{threeTerms}) may be rewritten as
\begin{align*}
&  \frac{1}{2}\frac{1}{\varepsilon^{2}}\sum_{n=0}^{\infty}\sum_{m=0}^{n}%
\frac{(-1)^{n}}{\varepsilon^{n}}\tau\lbrack x_{t,\lambda}(x_{t,\lambda}^{\ast
}x_{t,\lambda})^{m}]\tau\lbrack x_{t,\lambda}(x_{t,\lambda}^{\ast}%
x_{t,\lambda})^{n-m}]\\
=  &  \frac{1}{2}\left(  \sum_{k=0}^{\infty}\frac{(-1)^{k}}{\varepsilon^{k+1}%
}\tau\lbrack x_{t,\lambda}(x_{t,\lambda}^{\ast}x_{t,\lambda})^{k}]\right)
\left(  \sum_{l=0}^{\infty}\frac{(-1)^{l}}{\varepsilon^{l+1}}\tau\lbrack
x_{t,\lambda}(x_{t,\lambda}^{\ast}x_{t,\lambda})^{l}]\right) \\
=  &  \frac{1}{2}\tau\lbrack x_{t,\lambda}(x_{t,\lambda}^{\ast}x_{t,\lambda
}+\varepsilon)^{-1}]^{2}.
\end{align*}

The second term in (\ref{threeTerms}) differs from the first term only by
replacing the $x_{t,\lambda}$ by $x_{t,\lambda}^{\ast}$ in the two trace
terms, and is therefore computed as
\[
\frac{1}{2}\tau\lbrack x_{t,\lambda}^{\ast}(x_{t,\lambda}^{\ast}x_{t,\lambda
}+\varepsilon)^{-1}]^{2}.
\]
A similar computation expresses the last term in (\ref{threeTerms}) as
\[
\sum_{n=1}^{\infty}\sum_{m=1}^{n}\frac{(-1)^{n-1}}{\varepsilon^{n}}\tau
\lbrack(x_{t,\lambda}^{\ast}x_{t,\lambda})^{n-m}]\tau\lbrack(x_{t,\lambda
}^{\ast}x_{t,\lambda})^{m-1}]\ =\varepsilon\tau\lbrack(x_{t,\lambda}^{\ast
}x_{t,\lambda}+\varepsilon)^{-1}]^{2}.
\]
This shows that the series in (\ref{eq:SeriestoDiff}) converges to the right
hand side of (\ref{eq:DS.Dt}). It follows that (\ref{eq:DS.Dt}) holds for all
$\varepsilon>\Vert x_{t,\lambda}^{\ast}x_{t,\lambda}\Vert$.

Thus, for all $\varepsilon>\max_{s\leq t}\Vert x_{s,\lambda}^{\ast
}x_{s,\lambda}\Vert$, we have
\begin{align}
S(t,\lambda,\varepsilon)  &  =S(0,\lambda,\varepsilon)+\int_{0}^{t}\left\{
\frac{1}{2}\tau\lbrack x_{s,\lambda}(x_{s,\lambda}^{\ast}x_{s,\lambda
}+\varepsilon)^{-1}]^{2}\right. \nonumber\\
&  \left.  +\frac{1}{2}\tau\lbrack x_{s,\lambda}^{\ast}(x_{s,\lambda}^{\ast
}x_{s,\lambda}+\varepsilon)^{-1}]^{2}+\varepsilon\tau\lbrack(x_{s,\lambda
}^{\ast}x_{s,\lambda}+\varepsilon)^{-1}]^{2}\right\}  \,ds.
\label{eq:SIntegral}%
\end{align}
The right hand side of (\ref{eq:SIntegral}) is analytic in $\varepsilon$ for
all $\varepsilon>0$. We now claim that the left hand side of
(\ref{eq:SIntegral}) is also analytic. At each $\varepsilon>0$, we have the
operator-valued power series expansion
\[
\log(x_{t,\lambda}^{\ast}x_{t,\lambda}+\varepsilon+h)=\log(x_{t,\lambda}%
^{\ast}x_{t,\lambda}+\varepsilon)+\sum_{n=1}^{\infty}\frac{(-1)^{n-1}h^{n}}%
{n}(x_{t,\lambda}^{\ast}x_{t,\lambda}+\varepsilon)^{-n}%
\]
for $|h|<\Vert(x_{t,\lambda}^{\ast}x_{t,\lambda}+\varepsilon)^{-1}\Vert$.
Taking the trace gives
\[
S(t,\lambda,\varepsilon+h)=\log(x_{t,\lambda}^{\ast}x_{t,\lambda}%
+\varepsilon)+\sum_{n=1}^{\infty}\frac{(-1)^{n-1}h^{n}}{n}\tau\lbrack
(x_{t,\lambda}^{\ast}x_{t,\lambda}+\varepsilon)^{-n}]
\]
for $|h|<\Vert(x_{t,\lambda}^{\ast}x_{t,\lambda}+\varepsilon)^{-1}\Vert$. This
shows $S(t,\lambda,\cdot)$ is analytic on the positive real line. Since both
sides of (\ref{eq:SIntegral}) define an analytic function for $\varepsilon>0$
and they agree for all large $\varepsilon$, they are indeed equal for all
$\varepsilon>0$. Now, the conclusion of the proposition follows from
differentiating both sides of (\ref{eq:SIntegral}) with respect to $t.$
\end{proof}

\begin{lemma}
\label{Deri.Formulas}The partial derivatives of $S$ with respect to
$\varepsilon$ and $\lambda$ are given by the following formulas.
\begin{align*}
\frac{\partial S}{\partial\lambda}  &  =-\tau\lbrack x_{t,\lambda}^{\ast
}(x_{t,\lambda}^{\ast}x_{t,\lambda}+\varepsilon)^{-1}]\\
\frac{\partial S}{\partial\bar{\lambda}}  &  =-\tau\lbrack x_{t,\lambda
}(x_{t,\lambda}^{\ast}x_{t,\lambda}+\varepsilon)^{-1}]\\
\frac{\partial S}{\partial\varepsilon}  &  =\tau\lbrack(x_{t,\lambda}^{\ast
}x_{t,\lambda}+\varepsilon)^{-1}].
\end{align*}

\end{lemma}

\begin{proof}
By Lemma 1.1 in Brown's paper \cite{Br}, the derivative of the trace of a
logarithm is given by
\begin{equation}
\frac{d}{du}\tau\lbrack\log(f(u))]=\tau\left[  f(u)^{-1}\frac{df}{du}\right]
. \label{diffLog}%
\end{equation}
The lemma follows from applying this formula.
\end{proof}

Now we are ready to prove Theorem \ref{PDE.thm}.

\begin{proof}
[Proof of Theorem \ref{PDE.thm}]By Lemma \ref{DS.Dt},
\[
\frac{\partial S}{\partial t}=\frac{1}{2}\tau\lbrack x_{t,\lambda
}(x_{t,\lambda}^{\ast}x_{t,\lambda}+\varepsilon)^{-1}]^{2}+\frac{1}{2}%
\tau\lbrack x_{t,\lambda}^{\ast}(x_{t,\lambda}^{\ast}x_{t,\lambda}%
+\varepsilon)^{-1}]^{2}+\varepsilon\tau\lbrack(x_{t,\lambda}^{\ast
}x_{t,\lambda}+\varepsilon)^{-1}]^{2}.
\]
Using Lemma \ref{Deri.Formulas}, the above displayed equation can be written
as
\[
\frac{\partial S}{\partial t}=\frac{1}{2}\left(  \frac{\partial S}%
{\partial\lambda}\right)  ^{2}+\frac{1}{2}\left(  \frac{\partial S}%
{\partial\bar{\lambda}}\right)  ^{2}+\varepsilon\left(  \frac{\partial
S}{\partial\varepsilon}\right)  ^{2}.
\]
Now, (\ref{thePDE}) follows from applying the definition of Cauchy--Riemann
operators to the above equation. The initial condition holds because
$x_{t}=x_{0}$ when $t=0$.
\end{proof}

\section{The Hamilton--Jacobi analysis\label{HJ.sec}}

\subsection{The Hamilton--Jacobi method}

We define a \textquotedblleft Hamiltonian\textquotedblright\ function
$H:\mathbb{R}^{6}\rightarrow\mathbb{R}$ by replacing the derivatives $\partial
S/\partial a,$ $\partial S/\partial b,$ and $\partial S/\partial\varepsilon$
on the right-hand side of the PDE in Theorem \ref{PDE.thm} by
\textquotedblleft momentum\textquotedblright\ variables $p_{a},$ $p_{b},$ and
$p_{\varepsilon},$ and then reversing the overall sign. Thus, we define%
\begin{equation}
H(a,b,\varepsilon,p_{a},p_{b},p_{\varepsilon})=-\frac{1}{4}(p_{a}^{2}%
-p_{b}^{2})-\varepsilon p_{\varepsilon}^{2}, \label{theHamiltonian}%
\end{equation}
where in this case, $H$ happens to be independent of $a$ and $b.$ We then
introduce Hamilton's equations for the Hamiltonian $H,$ namely%
\begin{equation}
\frac{du}{dt}=\frac{\partial H}{\partial p_{u}};\quad\frac{dp_{u}}{dt}%
=-\frac{\partial H}{\partial u}, \label{HamSystem}%
\end{equation}
where $u$ ranges over the set $\{a,b,\varepsilon\}.$ We will use the notation%
\[
\lambda(t)=a(t)+ib(t).
\]

\begin{notation}
We use the notation
\[
p_{a,0},~p_{b,0},~p_{0}%
\]
for the initial values of $p_{a},$ $p_{b},$ and $p_{\varepsilon},$ respectively.
\end{notation}

In the Hamilton--Jacobi analysis, the initial momenta are determined by the
initial positions $\lambda_{0}$ and $\varepsilon_{0}$ by means of the
following formula:%
\begin{equation}
p_{a,0}=\frac{\partial}{\partial a_{0}}S(0,\lambda_{0},\varepsilon_{0});\quad
p_{b,0}=\frac{\partial}{\partial b_{0}}S(0,\lambda_{0},\varepsilon_{0});\quad
p_{0}=\frac{\partial}{\partial\varepsilon_{0}}S(0,\lambda_{0},\varepsilon
_{0}). \label{initMomentaGen}%
\end{equation}
Now, the formula for $S(0,\lambda,\varepsilon)$ in Theorem \ref{PDE.thm} may
be written more explicitly as%
\[
S(0,\lambda,\varepsilon)=\int_{\mathbb{R}}\log(\left\vert x-\lambda\right\vert
^{2}+\varepsilon)~d\mu(x),
\]
where $\mu$ is the law of $x_{0}$, as in (\ref{muDef}). We thus obtain the
following formula for the initial momenta:
\begin{align}
p_{a,0}  &  =\int_{\mathbb{R}}\frac{2(a_{0}-x)}{(a_{0}-x)^{2}+b_{0}%
^{2}+\varepsilon_{0}}~d\mu(x)\nonumber\\
p_{b,0}  &  =\int_{\mathbb{R}}\frac{2b_{0}}{(a_{0}-x)^{2}+b_{0}^{2}%
+\varepsilon_{0}}~d\mu(x)\label{initialMomenta}\\
p_{0}  &  =\int_{\mathbb{R}}\frac{1}{(a_{0}-x)^{2}+b_{0}^{2}+\varepsilon_{0}%
}~d\mu(x).\nonumber
\end{align}
Provided we assume $\varepsilon_{0}>0,$ the integrals are convergent.

\begin{proposition}
\label{HJ.prop}Suppose we have a solution to the Hamiltonian system on a time
interval $[0,T]$ such that $\varepsilon(t)>0$ for all $t\in\lbrack0,T].$ Then
we have
\begin{equation}
S(t,\lambda(t),\varepsilon(t))=S(0,\lambda_{0},\varepsilon_{0})+tH_{0},
\label{firstHJ}%
\end{equation}
where%
\[
H_{0}=H(a_{0},b_{0},\varepsilon_{0},p_{a,0},p_{b,0},p_{0}).
\]
We also have
\begin{equation}
\frac{\partial S}{\partial u}(t,\lambda(t),\varepsilon(t))=p_{u}(t)
\label{secondHJ}%
\end{equation}
for all $u\in\{a,b,\varepsilon\}.$
\end{proposition}

We refer to (\ref{firstHJ}) and (\ref{secondHJ}) as the first and second
Hamilton--Jacobi formulas, respectively.

\begin{proof}
The reader may consult Section 6.1 of \cite{DHKBrown} for a concise statement
and derivation of the general Hamilton--Jacobi method. (See also the book of
Evans \cite{Evans}.) The general form of the first Hamilton--Jacobi formula,
when applied to this case, reads as%
\[
S(t,\lambda(t),\varepsilon(t))=S(0,\lambda_{0},\varepsilon_{0})-tH_{0}%
+\int_{0}^{t}\sum_{u\in\{a,b,\varepsilon\}}p_{u}\frac{\partial H}{\partial
p_{u}}~ds.
\]
In our case, because the Hamiltonian is homogeneous of degree two in the
momentum variables, $\sum_{u\in\{a,b,\varepsilon\}}p_{u}\frac{\partial
H}{\partial p_{u}}$ is equal to $2H.$ Since $H$ is a constant of motion, the
general formula reduces to (\ref{firstHJ}). Meanwhile, (\ref{secondHJ}) is an
immediate consequence of the general form of the second Hamilton--Jacobi formula.
\end{proof}

\subsection{Solving the ODEs}

We now solve the Hamiltonian system (\ref{HamSystem}) with Hamiltonian given
by (\ref{theHamiltonian}). We start by noting several helpful constants of motion.

\begin{proposition}
The quantities%
\[
H,~p_{a},~p_{b},~\varepsilon p_{\varepsilon}^{2}%
\]
are constants of motion, meaning that they are constant along any solution of
Hamilton's equations (\ref{HamSystem}).
\end{proposition}

\begin{proof}
The Hamiltonian is always a constant of motion in any Hamiltonian system. The
quantities $p_{a}$ and $p_{b}$ are constants of motion because $H$ is
independent of $a$ and $b.$ And finally, $\varepsilon p_{\varepsilon}^{2}$ is
a constant of motion because it equals $-\frac{1}{4}(p_{a}^{2}-p_{b}^{2})-H.$
\end{proof}

We now obtain solutions to (\ref{HamSystem}), where at the moment, we allow
arbitrary initial momenta, not necessarily given by (\ref{initMomentaGen}).

\begin{proposition}
\label{solveODE.prop}Consider the Hamiltonian system (\ref{HamSystem}) with
Hamiltonian (\ref{theHamiltonian}) and initial conditions
\[
(a_{0},b_{0},\varepsilon_{0},p_{a,0},p_{b,0},p_{0}),
\]
with $p_{0}>0.$ Then the solution to the system exists up to time
\[
t_{\ast}=1/p_{0}.
\]
Up until that time, we have%
\begin{align*}
p_{a}(t)  &  =p_{a,0}\\
a(t)  &  =a_{0}-\frac{1}{2}p_{a,0}t.\\
p_{b}(t)  &  =p_{b,0}\\
b(t)  &  =b_{0}+\frac{1}{2}p_{b,0}t\\
p_{\varepsilon}(t)  &  =\frac{p_{0}}{1-p_{0}t}\\
\varepsilon(t)  &  =\varepsilon_{0}\left(  1-p_{0}t\right)  ^{2}.
\end{align*}
If $\varepsilon_{0}>0$ then $\varepsilon(t)$ remains positive for all
$t<t_{\ast}.$
\end{proposition}

\begin{proof}
We begin by noting that
\[
\dot{p}_{\varepsilon}=-\frac{\partial H}{\partial\varepsilon}=p_{\varepsilon
}^{2}.
\]
We may solve this separable equation as
\[
-\left(  \frac{1}{p_{\varepsilon}(t)}-\frac{1}{p_{0}}\right)  =t,
\]
from which the claimed formula for $p_{\varepsilon}(t)$ follows. We then note
that%
\begin{align*}
\frac{d\varepsilon}{dt}  &  =\frac{\partial H}{\partial p_{\varepsilon}}\\
&  =-2\varepsilon p_{\varepsilon}\\
&  =-2\varepsilon\frac{p_{0}}{1-p_{0}t}.
\end{align*}
This equation is also separable and may easily be integrated to give the
claimed formula for $\varepsilon(t).$

The formulas for $p_{a}$ and $p_{b}$ simply amount to saying that they are
constants of motion, and the formulas for $a$ and $b$ are then easily obtained.
\end{proof}

We now specialize the initial conditions to the form occurring in the
Hamilton--Jacobi method, that is, where the initial momenta are given by
(\ref{initialMomenta}). We note that the formulas in (\ref{initialMomenta})
can be written as
\begin{equation}
p_{b,0}=2b_{0}p_{0} \label{pb0}%
\end{equation}
and%
\begin{equation}
p_{a,0}=2a_{0}p_{0}-2p_{1}, \label{pa0}%
\end{equation}
where%
\begin{equation}
p_{1}=\int_{\mathbb{R}}\frac{x}{(a_{0}-x)^{2}+b_{0}^{2}+\varepsilon_{0}}%
~d\mu(x). \label{p1def}%
\end{equation}

\begin{proposition}
\label{solveODEspecial.prop}Suppose $a_{0},$ $b_{0},$ and $\varepsilon_{0}$
are chosen in such a way that $p_{0}=1/t,$ so that the lifetime $t_{\ast}$ of
the system equals $t.$ Then we have%
\begin{align*}
\lim_{s\rightarrow t^{-}}a(s)  &  =tp_{1}\\
\lim_{s\rightarrow t^{-}}b(s)  &  =2b_{0}\\
\lim_{s\rightarrow t^{-}}\varepsilon(s)  &  =0,
\end{align*}
where $p_{1}$ is as in (\ref{p1def}).
\end{proposition}

\begin{proof}
The result follows easily from the formulas in Proposition \ref{solveODE.prop}%
, after using the relations (\ref{pb0}) and (\ref{pa0})\ and setting
$p_{0}=1/t.$
\end{proof}

\begin{definition}
\label{lifetimes.def}Let $t_{\ast}(\lambda_{0},\varepsilon_{0})$ denote the
lifetime of the solution, namely
\[
t_{\ast}(\lambda_{0},\varepsilon_{0})=\frac{1}{p_{0}}=\left(  \int%
_{\mathbb{R}}\frac{d\mu(x)}{(a_{0}-x)^{2}+b_{0}^{2}+\varepsilon_{0}}\right)
^{-1},
\]
and let
\[
T(\lambda_{0}):=\lim_{\varepsilon_{0}\rightarrow0^{+}}t_{\ast}(\lambda
_{0},\varepsilon_{0})=\left(  \int_{\mathbb{R}}\frac{d\mu(x)}{(a_{0}%
-x)^{2}+b_{0}^{2}}\right)  ^{-1}.
\]

\end{definition}

We note that if $b_{0}=0$ then the integral in the definition of
$T(a_{0}+ib_{0})$ may be infinite for certain values of $a_{0}.$ Thus, it is
possible for $T(a_{0}+ib_{0})$ to equal 0 when $b_{0}=0.$

\begin{proposition}
\label{smallEpsilon0.prop}Let%
\[
\lambda(t;\lambda_{0},\varepsilon_{0})
\]
denote the solution to the system (\ref{HamSystem}) with $\lambda
(0)=\lambda_{0}$ and $\varepsilon(0)=\varepsilon_{0}$, and with initial
momenta given by (\ref{initialMomenta}). Suppose $\lambda_{0}$ satisfies
$T(\lambda_{0})>t.$ Then%
\[
\lim_{\varepsilon_{0}\rightarrow0^{+}}\lambda(t;\lambda_{0},\varepsilon
_{0})=\lambda_{0}-t\int_{\mathbb{R}}\frac{1}{\lambda_{0}-x}~d\mu(x),
\]
provided that $\lambda_{0}$ does not belong the closed support of $\mu.$
\end{proposition}

\begin{proof}
Using Proposition \ref{solveODE.prop}, we find that
\begin{align*}
\lambda(t;\lambda_{0},\varepsilon_{0})  &  =a(t)+ib(t)\\
&  =\lambda_{0}-\frac{t}{2}(p_{a,0}-ip_{b,0}).
\end{align*}
In the limit as $\varepsilon_{0}$ tends to zero, we have (provided
$\lambda_{0}$ is not in $\mathrm{supp}(\mu)\subset\mathbb{R}$)%
\[
p_{a,0}-ip_{b,0}=\int_{\mathbb{R}}\frac{2(a_{0}-x)}{(a_{0}-x)^{2}+b_{0}^{2}%
}~d\mu(x)-i\int_{\mathbb{R}}\frac{2b_{0}}{(a_{0}-x)^{2}+b_{0}^{2}}~d\mu(x).
\]
It is then easy to check that
\[
p_{a,0}-ip_{b,0}=2\int_{\mathbb{R}}\frac{1}{a_{0}+ib_{0}-x}~d\mu(x),
\]
which gives the claimed formula.
\end{proof}

\section{The domains\label{domain.sec}}

\subsection{The domain in the $\lambda_{0}$-plane\label{lambda0Domain.sec}}

We now define the first of two domains we will be interested in. When we apply
the Hamilton--Jacobi method in Section \ref{outside.sec}, we will try to find
solutions with $\varepsilon(t)$ very close to zero. Based on the formula for
$\varepsilon(t)$ in Proposition \ref{solveODE.prop}, it seems that we can make
$\varepsilon(t)$ small by making $\varepsilon_{0}$ small. The difficulty with
this approach, however, is that if we fix some $\lambda_{0}$ and let
$\varepsilon_{0}$ tend to zero, the lifetime of the path may be smaller than
$t.$ Thus, if the small-$\varepsilon_{0}$ lifetime of the path---as computed
by the function $T$ in Definition \ref{lifetimes.def}---is smaller than $t,$
the simple approach of letting $\varepsilon_{0}$ tend to zero will not work.
This observation motivates the following definition.

\begin{definition}
\label{lambdaT.def}Let $T$ be the function defined in Definition
\ref{lifetimes.def}. We then define a domain $\Lambda_{t}\subset\mathbb{C}$ by%
\[
\Lambda_{t}=\left\{  \left.  \lambda_{0}\in\mathbb{C}\right\vert T(\lambda
_{0})<t\right\}  .
\]
Explicitly, a point $\lambda_{0}=a_{0}+ib_{0}$ belongs to $\Lambda_{t}$ if and
only if%
\begin{equation}
\int_{\mathbb{R}}\frac{d\mu(x)}{(a_{0}-x)^{2}+b_{0}^{2}}>\frac{1}{t}.
\label{LambdaTdef2}%
\end{equation}

\end{definition}

This domain appeared originally in the work Biane \cite{BianeConvolution}, for
reasons that we will explain in Section \ref{bianeResult.sec}. The domain
$\Lambda_{t}$ also plays a crucial role in work of the second author with
Zhong \cite{HZ}. In Section \ref{lambdaDomain.sec}, we will consider another
domain $\Omega_{t},$ whose closure will be the support of the Brown measure of
$x_{0}+i\sigma_{t}.$ See Figure \ref{regionplots.fig} for plots of
$\Lambda_{t}$ and the corresponding domain $\Omega_{t}$.

We give now a more explicit description of the domain $\Lambda_{t}.$

\begin{proposition}
\label{vt.prop}For each $t>0,$ define a function $v_{t}:\mathbb{R}%
\rightarrow\lbrack0,\infty)$ as follows. For each $a_{0}\in\mathbb{R},$ if%
\begin{equation}
\int_{\mathbb{R}}\frac{1}{(a_{0}-x)^{2}}~d\mu(x)>\frac{1}{t},
\label{vtPositive}%
\end{equation}
let $v_{t}(a_{0})$ be the unique positive number such that%
\begin{equation}
\int_{\mathbb{R}}\frac{1}{(a_{0}-x)^{2}+v_{t}(a_{0})^{2}}~d\mu(x)=\frac{1}{t}.
\label{vtDef}%
\end{equation}
If, on the other hand,%
\begin{equation}
\int_{\mathbb{R}}\frac{1}{(a_{0}-x)^{2}}~d\mu(x)\leq\frac{1}{t},
\label{vtZero}%
\end{equation}
set $v_{t}(a_{0})=0.$

Then the function $v_{t}:\mathbb{R}\rightarrow\lbrack0,\infty)$ is continuous
and the domain $\Lambda_{t}$ may be described as%
\begin{equation}
\Lambda_{t}=\left\{  \left.  a_{0}+ib_{0}\in\mathbb{C}\right\vert ~\left\vert
b_{0}\right\vert <v_{t}(a_{0})\right\}  , \label{LambdaTdef3}%
\end{equation}
so that
\begin{equation}
\Lambda_{t}\cap\mathbb{R}=\left\{  \left.  a_{0}\in\mathbb{R}\right\vert
v_{t}(a_{0})>0\right\}  . \label{LambdaTIntersectR}%
\end{equation}

\end{proposition}

See Figure \ref{vtplots.fig} for some plots of the function $v_{t}.$

\begin{proof}
We first note that for any fixed $a_{0},$ the integral%
\begin{equation}
\int_{\mathbb{R}}\frac{1}{(a_{0}-x)^{2}+v^{2}}~d\mu(x) \label{integralWithV}%
\end{equation}
is a strictly decreasing function of $v\geq0$ and that the integral tends to
zero as $v$ tends to infinity. Thus, whenever condition (\ref{vtPositive})
holds, it is easy to see that there is a unique positive number $v_{t}(a_{0})$
for which (\ref{vtDef}) holds. Continuity of $v_{t}$ is established in
\cite[Lemma 2]{BianeConvolution}.

Using the monotonicity of the integral in (\ref{integralWithV}), it is now
easy to see that the characterization of the domain $\Lambda_{t}$ in
(\ref{LambdaTdef3}) is equivalent to the characterization in
(\ref{LambdaTdef2}).
\end{proof}

\begin{figure}[ptb]%
\centering
\includegraphics[
height=3.1185in,
width=3.5613in
]%
{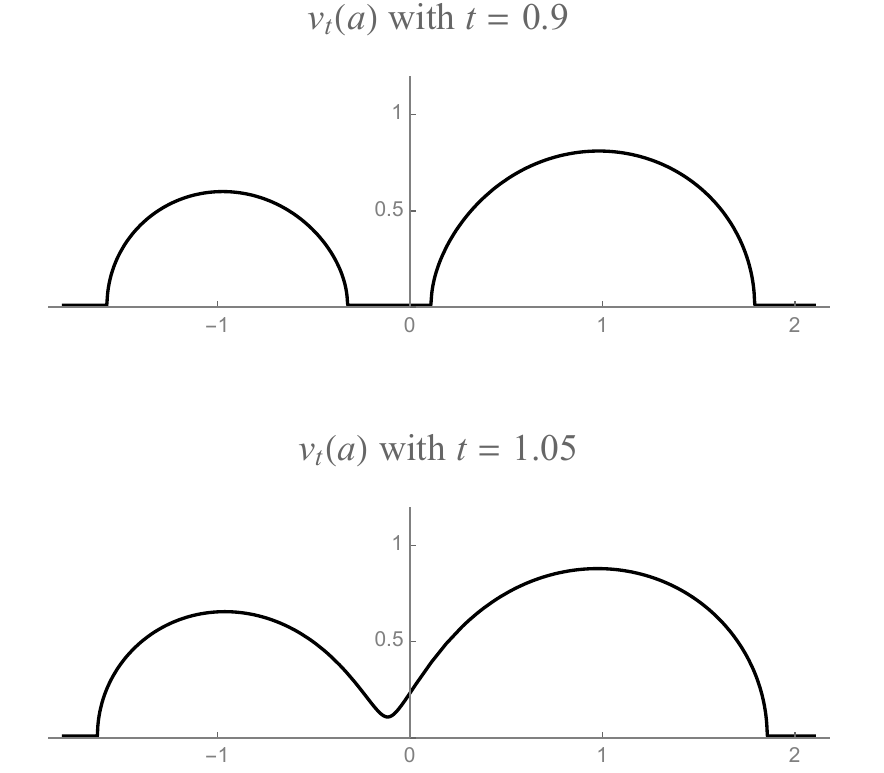}%
\caption{The function $v_{t}(a)$ for the case in which $\mu=\frac{1}{3}%
\delta_{-1}+\frac{2}{3}\delta_{1}.$}%
\label{vtplots.fig}%
\end{figure}

\subsection{The result of Biane\label{bianeResult.sec}}

We now explain how the domain $\Lambda_{t}$ arose in the work of Biane
\cite{BianeConvolution}. The results of Biane will be needed to formulate one
of our main results (Theorem \ref{push.thm}).

For any operator $A\in\mathcal{A},$ we let $G_{A}$ denote the Cauchy transform
of $A,$ also known as the Stieltjes transform or holomorphic Green's function,
defined as%
\begin{equation}
G_{A}(z)=\tau\lbrack(z-A)^{-1}] \label{GreensFunctionDef}%
\end{equation}
for all $z\in\mathbb{C}$ outside the spectrum of $A.$ Then $G_{A}$ is
holomorphic on its domain. If $A$ is self-adjoint, we can recover the
distribution of $A$ from its Cauchy transform by the Stieltjes inversion
formula. Even if $A$ is not self-adjoint, $G_{A}$ determines the holomorphic
moments of the Brown measure $\mathrm{Brown}(A)$ of $A,$ that is, the
integrals of $\lambda^{n}$ with respect to $\mathrm{Brown}(A).$ (We emphasize
that these holomorphic moments do not, in general, determine the Brown measure itself.)

Let $x_{0}$ be a self-adoint element of $\mathcal{A}$ and let $\sigma_{t}%
\in\mathcal{A}$ be a semicircular Brownian motion freely independent of
$x_{0}.$ Define a function $H_{t}$ by%
\begin{equation}
H_{t}(\lambda_{0})=\lambda_{0}+tG_{x_{0}}(\lambda_{0}). \label{htDef}%
\end{equation}
The significance of this function is from the following result of Biane
\cite{BianeConvolution}, which shows that the Cauchy transform of
$x_{0}+\sigma_{t}$ is related to the Cauchy transform of $x_{0}$ by the
formula%
\begin{equation}
G_{x_{0}+\sigma_{t}}(H_{t}(\lambda_{0}))=G_{x_{0}}(\lambda_{0}),
\label{bianeIdentity}%
\end{equation}
for $\lambda_{0}$ in an appropriate set, which we will specify shortly. Note
that this result is for the Cauchy transform of the self-adjoint operator
$x_{0}+\sigma_{t}$, not for $x_{0}+i\sigma_{t}.$

We now explain the precise domain (taken to be in the upper half-plane for
simplicity) on which the identity (\ref{bianeIdentity}) holds. Let%
\begin{equation}
\Delta_{t}=\left\{  \left.  a_{0}+ib_{0}\right\vert b_{0}>v_{t}(a_{0}%
)\right\}  , \label{DeltaDef}%
\end{equation}
which is just the set of points in the upper half-plane outside the closure of
$\Lambda_{t}.$ The boundary of $\Delta_{t}$ is then the graph of $v_{t}$:%
\[
\partial\Delta_{t}=\left\{  a_{0}+i\left.  v_{t}(a_{0})\right\vert a_{0}%
\in\mathbb{R}\right\}  .
\]

\begin{theorem}
[Biane]\label{BianeH.thm}First, the function $H_{t}$ is an injective conformal
map of $\Delta_{t}$ onto the upper half-plane. Second, $H_{t}$ maps
$\partial\Delta_{t}$ homeomorphically onto the real line. Last, the identity
(\ref{bianeIdentity}) holds for all $\lambda_{0}$ in $\Delta_{t}.$ Thus, we
may write%
\[
G_{x_{0}+\sigma_{t}}(\lambda)=G_{x_{0}}(H_{t}^{-1}(\lambda))
\]
for all $\lambda$ in the upper half-plane, where the inverse function
$H_{t}^{-1}$ is chosen to map into $\Delta_{t}.$
\end{theorem}

See Lemma 4 and Proposition 2 in \cite{BianeConvolution}. In the terminology
of Voiculescu \cite{Voi1,Voi2}, we may say that $H_{t}^{-1}$ is one of the
\textbf{subordination functions} for the sum $x_{0}+\sigma_{t},$ meaning that
one can compute $G_{x_{0}+\sigma_{t}}$ from $G_{x_{0}}$ by composing with
$H_{t}^{-1}.$ Since $x_{0}+\sigma_{t}$ is self-adjoint, one can then compute
the distribution of $x_{0}+\sigma_{t}$ from its Cauchy transform. We remark
that Biane denotes the map $H_{t}^{-1}$ by $F_{t}$ on p. 710 of
\cite{BianeConvolution}.

\subsection{The domain in the $\lambda$-plane\label{lambdaDomain.sec}}

Our strategy in applying the Hamilton--Jacobi method will be in two stages. In
the first stage, we attempt to make $\varepsilon(t)$ close to zero by taking
$\varepsilon_{0}$ close to zero. For this strategy to work, we must have
$\lambda_{0}$ outside the closure of the domain $\Lambda_{t}$ introduced in
Section \ref{lambda0Domain.sec}. We will then solve the system of ODEs
(\ref{HamSystem}) in the limit as $\varepsilon_{0}$ approaches zero, using
Proposition \ref{smallEpsilon0.prop}. Let us define a map $J_{t}$ by
\begin{equation}
J_{t}(\lambda_{0})=\lambda_{0}-tG_{x_{0}}(\lambda_{0}), \label{jtDef}%
\end{equation}
which differs from the function $H_{t}$ in Section \ref{bianeResult.sec} by a
change of sign. (See Section \ref{jnpwz.sec} for a different perspective on
how this function arises.) With this notation, Proposition
\ref{smallEpsilon0.prop} says that if $\lambda(0)=\lambda_{0}$ and
$\varepsilon_{0}$ approaches zero, then%
\[
\lambda(t)=J_{t}(\lambda_{0}),
\]
provided that $\lambda_{0}$ is outside the closure of $\Lambda_{t}.$ Thus, the
first stage of our analysis will allow us to compute the Brown measure at
points of the form $J_{t}(\lambda_{0})$ with $\lambda_{0}\notin\overline
{\Lambda}_{t}.$ We will find that the Brown measure is zero in a neighborhood
of any such point. A second stage of the analysis will then be required to
compute the Brown measure at points inside $\overline{\Lambda}_{t}.$

The discussion the previous paragraph motivates the following definition.

\begin{definition}
\label{omegaT.def}For each $t>0,$ define a domain $\Omega_{t}$ in $\mathbb{C}$
by%
\[
\Omega_{t}=[J_{t}(\Lambda_{t}^{c})]^{c}.
\]
That is to say, the complement of $\Omega_{t}$ is the image under $J_{t}$ of
the complement of $\Lambda_{t}.$
\end{definition}

See Figure \ref{regionplots.fig} for plots of the domains $\Lambda_{t}$ and
$\Omega_{t}.$%

\begin{figure}[ptb]%
\centering
\includegraphics[
height=6.0165in,
width=5.6273in
]%
{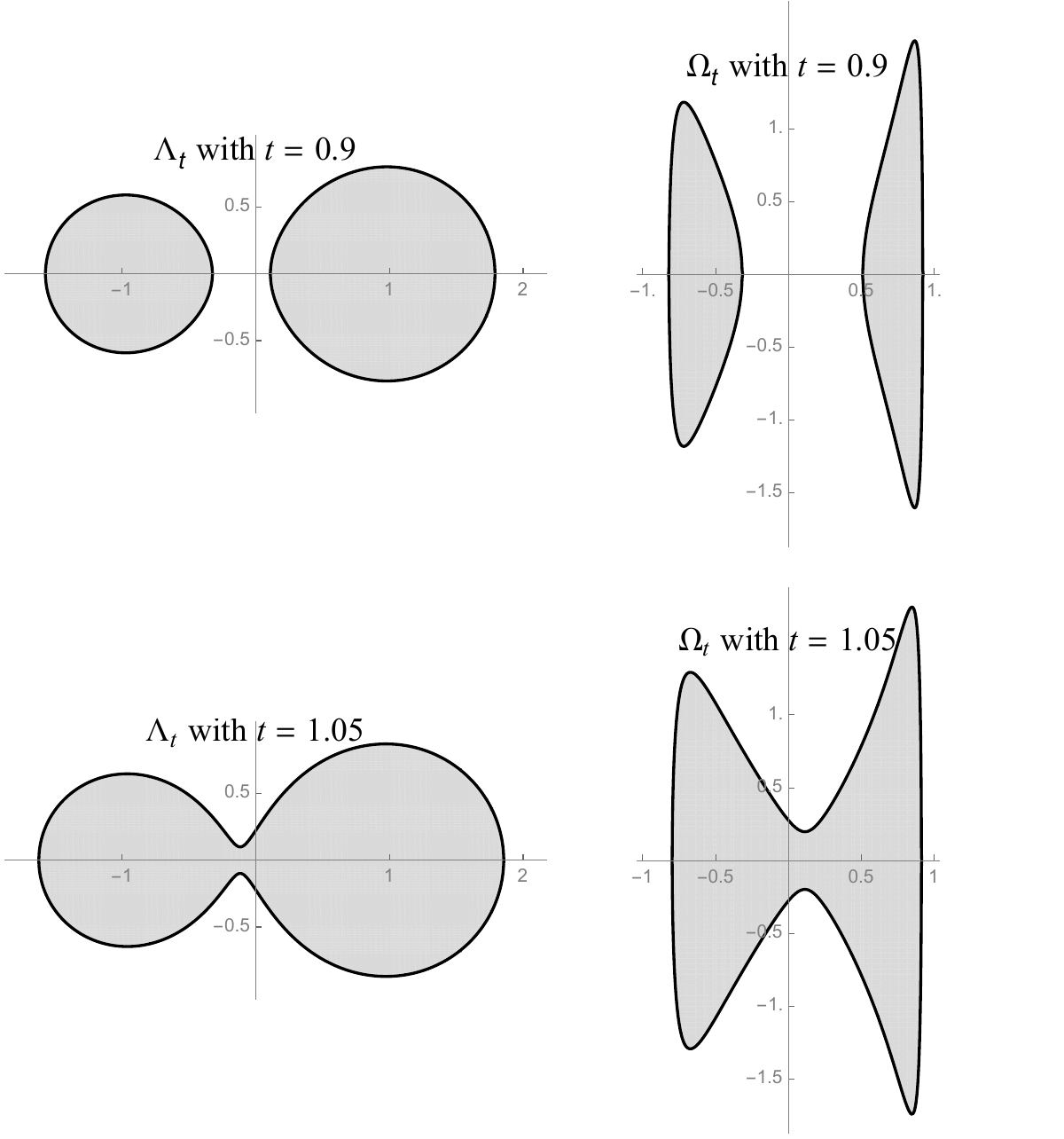}%
\caption{The regions $\Lambda_{t}$ and $\Omega_{t}$ for $\mu=\frac{1}{3}%
\delta_{-1}+\frac{2}{3}\delta_{1}$.}%
\label{regionplots.fig}%
\end{figure}

We recall our standing assumption that $\mu$ is not a $\delta$-measure and we
remind the reader that the set $\Delta_{t}$ in (\ref{DeltaDef}) is the region
above the graph of $v_{t}$ so that $\overline{\Delta}_{t}$ is the set of
points on or above the graph of $v_{t}.$

\begin{proposition}
\label{Jprops.prop}The following results hold.

\begin{enumerate}
\item \label{jInject.point}The map $J_{t}$ is well-defined, continuous, and
injective on $\overline{\Delta}_{t}.$

\item \label{datDt.point}Define a function $a_{t}:\mathbb{R}\rightarrow
\mathbb{R}$ by%
\begin{equation}
a_{t}(a_{0})=\operatorname{Re}[J_{t}(a_{0}+iv_{t}(a_{0}))]. \label{atDef}%
\end{equation}
Then at any point $a_{0}$ with $v_{t}(a_{0})>0,$ the function $a_{t}$ is
differentiable and satisfies%
\[
0<\frac{da_{t}}{da_{0}}<2.
\]

\item \label{atHomeo.point}The function $a_{t}$ is continuous and strictly
increasing and maps $\mathbb{R}$ onto $\mathbb{R}.$

\item \label{graphs.point}The map $J_{t}$ maps the graph of $v_{t}$ to the
graph of a function, which we denote by $b_{t}.$ The function $b_{t}$
satisfies
\begin{equation}
b_{t}(a_{t}(a_{0}))=2v_{t}(a_{0}) \label{eq:btDef}%
\end{equation}
for all $a_{0}\in\mathbb{R}.$

\item \label{aboveGraphs.point}The map $J_{t}$ takes the region above the
graph of $v_{t}$ \emph{onto} the region above the graph of $b_{t}.$

\item \label{OmegaGraph.point}The set $\Omega_{t}$ defined in Definition
\ref{omegaT.def} may be computed as%
\[
\Omega_{t}=\left\{  \left.  a+ib\in\mathbb{C}\right\vert ~\left\vert
b\right\vert <b_{t}(a)\right\}  .
\]

\end{enumerate}
\end{proposition}

Since $J_{t}(z)=2z-H_{t}(z)$ and $H_{t}(a_{0}+iv_{t}(a_{0}))$ is real, we see
that $a_{t}(a_{0})=2a_{0}-H_{t}(a_{0}+iv_{t}(a_{0}))$. Lemma 5 of
\cite{BianeConvolution} and Theorem 3.14 of \cite{HZ} show that $0<H_{t}%
^{\prime}(a_{0}+iv_{t}(a_{0})\leq2$, which means $0\leq a_{t}^{\prime}%
(a_{0})<2$. Thus, Point \ref{datDt.point} improves the result to
$0<a_{t}^{\prime}(a_{0})<2$.

The proof requires $\mu$ to have more than one point in its support in order
to prove $a_{t}^{\prime}(a_{0}) \neq0$. When $\mu=\delta_{0}$, it can be
computed that $J_{t}(z)=z-\frac{t}{z}$ and $a_{0}+iv_{t}(a_{0})$ is the upper
semicircle of radius $\sqrt{t}$. Therefore, $\operatorname{Re}[J_{t}%
(a_{0}+iv_{t}(a_{0}))]=0$ for all $a_{0}\in\Lambda_{t}\cap\mathbb{R}$, and its
derivative is constantly $0$ on $\Lambda_{t}\cap\mathbb{R}$.

The proof is similar to the proof in \cite{BianeConvolution} of similar
results about the map $H_{t}$.

\begin{proof}
Continuity of $J_{t}$ on $\overline{\Delta}_{t}$ follows from \cite[Lemma
3]{BianeConvolution}, which shows continuity of $G_{x_{0}}$ on $\overline
{\Delta}_{t}.$ To show injectivity of $J_{t}$, suppose, toward a
contradiction, that $J_{t}(z_{1})=J_{t}(z_{2})$, for some $z_{1}\neq z_{2}$ in
$\overline{\Delta}_{t}$. Then, using the definition (\ref{jtDef}) of $J_{t}$,
we have
\[
t(G_{\mu}(z_{2})-G_{\mu}(z_{1}))=z_{2}-z_{1}.
\]
This shows
\[
t\int_{\mathbb{R}}\frac{z_{1}-z_{2}}{(z_{1}-x)(z_{2}-x)}\,d\mu(x)=z_{2}%
-z_{1}.
\]
Since we are assuming that $z_{1}$ and $z_{2}$ are distinct, we can divide by
$z_{1}-z_{2}$ to obtain
\begin{equation}
\int_{\mathbb{R}}\frac{d\mu(x)}{(z_{1}-x)(z_{2}-x)}=-\frac{1}{t}.
\label{eq:IntProduct}%
\end{equation}

Since $z_{1},z_{2}\in\overline{\Delta}_{t}$, we have $T(z_{1})\leq1/t$ and
$T(z_{2})\leq1/t$. Thus, by the Cauchy--Schwarz inequality,
\[
\left\vert \int_{\mathbb{R}}\frac{d\mu(x)}{(z_{1}-x)(z_{2}-x)}\right\vert
^{2}\leq\int_{\mathbb{R}}\frac{d\mu(x)}{|z_{1}-x|^{2}}\int_{\mathbb{R}}%
\frac{d\mu(x)}{|z_{2}-x|^{2}}\leq\frac{1}{t^{2}}.
\]
By (\ref{eq:IntProduct}), we have equality in the above Cauchy--Schwarz
inequality. Therefore, there exists an $\alpha\in\mathbb{C}$ such that the
relation
\begin{equation}
\frac{1}{\bar{z}_{2}-x}=\frac{\alpha}{z_{1}-x}, \label{CSeq1}%
\end{equation}
or, equivalently,%
\begin{equation}
(\alpha-1)x=\alpha\bar{z}_{2}-z_{1} \label{CSeq2}%
\end{equation}
holds for $\mu$-almost every $x$. Since $\mu$ is assumed not to be a $\delta
$-measure, we must have $\alpha=1,$ or else $x$ would equal the constant value
$(\alpha\bar{z}_{2}-z_{1})/(\alpha-1)$ for $\mu$-almost every $x.$ With
$\alpha=1,$ we find that $z_{1}=\bar{z}_{2}.$ But now if we substitute
$z_{1}=\bar{z}_{2}$ into (\ref{eq:IntProduct}), we obtain
\[
\int_{\mathbb{R}}\frac{d\mu(x)}{|z_{1}-x|^{2}}=-\frac{1}{t}%
\]
which is impossible. This shows $z_{1}$ and $z_{2}$ cannot be distinct and
Point \ref{jInject.point} is established.

For Point \ref{datDt.point}, fix $a_{0}$ with $v_{t}(a_{0})>0$. We compute
that
\[
G_{x_{0}}^{\prime}(\lambda_{0})=-\int_{\mathbb{R}}\frac{d\mu(x)}{(\lambda
_{0}-x)^{2}}%
\]
so that
\begin{equation}
\left\vert G_{x_{0}}^{\prime}(a_{0}+iv_{t}(a_{0}))\right\vert \leq
\int_{\mathbb{R}}\frac{d\mu(x)}{(a_{0}-x)^{2}+v_{t}(a_{0})^{2}}=\frac{1}{t}.
\label{eq:CauchyBound}%
\end{equation}
We claim that this inequality must be strict. Otherwise, we would have
equality in the ``putting the absolute value inside the integral'' inequality.
This would mean, by the proof of Theorem 1.33 of \cite{Rudin}, that
\[
\frac{1}{(\lambda_{0}-x)^{2}}%
\]
would have the same phase for $\mu$-almost every $x$. But since $\lambda_{0}$
is in the upper half plane, the phase of $\lambda_{0}-x$ increases from $0$ to
$\pi$ as $x$ increases from $-\infty$ to $\infty$. Thus, the phase of
$1/(\lambda_{0}-x)^{2}$ decreases from $2\pi$ to $0$ as $x$ increases from
$-\infty$ to $\infty$. Therefore, $1/(\lambda_{0}-x)^{2}$ cannot have the same
phase $\mu$-almost every $x$ unless $\mu$ is a $\delta$-measure.

Now,
\[
\frac{d}{da_{0}}G_{x_{0}}(a_{0}+iv_{t}(a_{0}))=G_{x_{0}}^{\prime}(a_{0}%
+iv_{t}(a_{0}))\left(  1+i\frac{dv_{t}(a_{0})}{da_{0}}\right)  .
\]
Since (\ref{eq:CauchyBound}) is a strict inequality,
\begin{equation}
\left\vert \frac{d}{da_{0}}G_{x_{0}}(a_{0}+iv_{t}(a_{0}))\right\vert
^{2}<\frac{1}{t^{2}}\left(  1+\left(  \frac{dv_{t}(a_{0})}{da_{0}}\right)
^{2}\right)  . \label{eq:CauchyBound2}%
\end{equation}
Since
\[
\operatorname{Im}[G_{x_{0}}(a_{0}+iv_{t}(a_{0}))]=-v_{t}(a_{0})\int%
_{\mathbb{R}}\frac{d\mu(x)}{(a_{0}-x)^{2}+v_{t}(a_{0})^{2}}=-\frac{v_{t}%
(a_{0})}{t},
\]
we have
\[
\left(  \frac{d}{da_{0}}\operatorname{Im}[G_{x_{0}}(a_{0}+iv_{t}%
(a_{0}))]\right)  ^{2}=\frac{1}{t^{2}}\frac{dv_{t}(a_{0})}{a_{0}}%
\]
and (\ref{eq:CauchyBound2}) becomes
\[
\left(  \frac{d}{da_{0}}\operatorname{Re}[G_{x_{0}}(a_{0}+iv_{t}%
(a_{0}))]\right)  ^{2}<\frac{1}{t^{2}}.
\]
This shows, using the definition $a_{t}(a_{0})=\operatorname{Re}[J_{t}%
(a_{0}+iv_{t}(a_{0}))]$,
\[
a_{t}^{\prime}(a_{0})=1-t\frac{d}{da_{0}}\operatorname{Re}[G_{x_{0}}%
(a_{0}+iv_{t}(a_{0}))]\in(0,2),
\]
as claimed.

We now turn to Point \ref{atHomeo.point}. To show that the function
$a_{t}(a_{0})$ is strictly increasing with $a_{0},$ we use two observations.
First, by Point \ref{datDt.point}, $a_{t}$ is increasing at any point $a_{0}$
where $v_{t}(a_{0})>0.$ Second, when $v_{t}(a_{0})=0,$ we have%
\begin{equation}
a_{t}(a_{0})=a_{0}-t\int_{\mathbb{R}}\frac{1}{a_{0}-x}~d\mu(x).
\label{atRealCase}%
\end{equation}
We claim that the right-hand side of (\ref{atRealCase}) is an increasing
function of $a_{0}.$ Suppose that $a_{0}<a_{1}$ and $v_{t}(a_{0})=v_{t}%
(a_{1})=0$. We compute
\begin{equation}
a_{t}(a_{1})-a_{t}(a_{0})=(a_{1}-a_{0})\left(  1+t\int\frac{1}{(a_{1}%
-x)(a_{0}-x)}\,d\mu(x)\right)  . \label{atDiff}%
\end{equation}
By Cauchy--Schwarz inequality and (\ref{vtZero}),
\[
\left\vert \int\frac{1}{(a_{1}-x)(a_{0}-x)}\,d\mu(x)\right\vert ^{2}\leq
\int\frac{d\mu(x)}{(a_{1}-x)^{2}}\int\frac{d\mu(x)}{(a_{0}-x)^{2}}\leq\frac
{1}{t^{2}};
\]
the Cauchy--Schwarz inequality is indeed strict by the reasoning leading to
(\ref{CSeq1}) and (\ref{CSeq2}). This proves that the right-hand side of
(\ref{atDiff}) is positive and we have established our claim.

Consider any two points $a_{0}$ and $a_{1}$ with $a_{0}<a_{1}$; we wish to
show that $a_{t}(a_{0})<a_{t}(a_{1}).$

We consider four cases, corresponding to whether $v_{t}(a_{0})$ and
$v_{t}(a_{1})$ are zero or positive. If $v_{t}(a_{0})$ and $v_{t}(a_{1})$ are
both zero, we use (\ref{atRealCase}) and immediately conclude that
$a_{t}(a_{0})<a_{t}(a_{1}).$ If $v_{t}(a_{0})=0$ but $v_{t}(a_{1})>0,$ then
let $\alpha$ be the infimum of the interval $I$ around $a_{1}$ on which
$v_{t}$ is positive, so that $v_{t}(\alpha)=0$ and $a_{0}\leq\alpha.$ Then
$a_{t}(a_{0})\leq a_{t}(\alpha)$ by (\ref{atRealCase})\ and $a_{t}%
(\alpha)<a_{t}(a_{1})$ by the positivity of $a_{t}^{\prime}$ on $I.$ The
remaining cases are similar; the case where both $v_{t}(a_{0})$ and
$v_{t}(a_{1})$ are positive can be subdivided into two cases depending on
whether or not $a_{0}$ and $a_{1}$ are in the same interval of positivity of
$v_{t}.$

Finally, we show that $a_{t}$ maps $\mathbb{R}$ onto $\mathbb{R}.$ Since
$x_{0}$ is assumed to be bounded, the law $\mu$ of $x_{0}$ is compactly
supported. It then follows easily from the condition (\ref{vtZero}) for
$v_{t}$ to be zero that $v_{t}(a_{0})=0$ whenever $\left\vert a_{0}\right\vert
$ is large enough. Thus, for $\left\vert a_{0}\right\vert $ large, the formula
(\ref{atRealCase}) applies, and we can easily see that $\lim_{a_{0}%
\rightarrow-\infty}a_{t}(a_{0})=-\infty$ and $\lim_{a_{0}\rightarrow+\infty
}a_{t}(a_{0})=+\infty.$

For Point \ref{graphs.point}, it follows easily from Point \ref{atHomeo.point}
and the definition (\ref{atDef}) of $a_{t}$ that $J_{t}$ maps the graph of
$v_{t}$ to the graph of a function. When then note that (\ref{eq:btDef}) holds
when $v_{t}(a_{0})=0$---both sides are zero. To establish (\ref{eq:btDef})
when $v_{t}(a_{0})>0$, we compute that
\begin{align*}
\operatorname{Im}[J_{t}(a_{0}+iv_{t}(a_{0}))]  &  =v_{t}(a_{0}%
)-t\operatorname{Im}\int_{\mathbb{R}}\frac{1}{a_{0}+iv_{t}(a_{0})-x}~d\mu(x)\\
&  =v_{t}(a_{0})+tv_{t}(a_{0})\int_{\mathbb{R}}\frac{1}{(a_{0}-x)^{2}%
+v_{t}(a_{0})^{2}}~d\mu(x)\\
&  =2v_{t}(a_{0}),
\end{align*}
by the defining property (\ref{vtDef}) of $v_{t}.$

For Point \ref{aboveGraphs.point}, we note that the graph of $v_{t},$ together
with the point at infinity, forms a Jordan curve in the Riemann sphere, with
the region above the graph as the interior of the disk---and similarly with
$v_{t}$ replaced by $b_{t}.$ Since $J_{t}(\lambda_{0})$ tends to infinity as
$\lambda_{0}$ tends to infinity, $J_{t}$ defines a continuous map of the
closed disk bounded by $\mathrm{graph}(v_{t})\cup\{\infty\}$ to the closed
disk bounded by $\mathrm{graph}(b_{t})\cup\{\infty\},$ and this map is a
homeomorphism on the boundary. By an elementary topological argument, $J_{t}$
must map the closed disk \textit{onto} the closed disk.

Finally, for Point \ref{OmegaGraph.point}, we use the description of
$\Lambda_{t}$ in Proposition \ref{vt.prop} as the region bounded by the graphs
of $v_{t}$ and $-v_{t}.$ The complement of $\Lambda_{t}$ thus consists of the
region on or above the graph of $v_{t}$ or on or below the graph of $-v_{t}.$
By Point \ref{aboveGraphs.point} and the fact that $J_{t}$ commutes the
complex conjugation, $J_{t}$ will map the complement of $\Lambda_{t}$ to the
region on or above the graph of $b_{t}$ or on or below the graph of $-b_{t}.$
Thus, from Definition \ref{omegaT.def}, $\Omega_{t}$ will be the region
bounded by the graphs of $b_{t}$ and $-b_{t}.$
\end{proof}

\subsection{The method of Janik, Nowak, Papp, Wambach, and
Zahed\label{jnpwz.sec}}

We now discuss the work of Janik, Nowak, Papp, Wambach, and Zahed \cite{JNPWZ}
which gives a nonrigorous but illuminating method of computing the support of
the Brown measure of $x_{0}+i\sigma_{t}.$ (See especially Section V of
\cite{JNPWZ}.) This method does not say anything about the Brown measure
besides what its support should be. Furthermore, it is independent of the
method used by Jarosz and Nowak in \cite{JN1,JN2} and discussed in Section
\ref{jn.sec}.

Recall the definition of the Cauchy transform of an operator $A$ in
(\ref{GreensFunctionDef}). We note that if $\lambda$ is outside the spectrum
of $A,$ then we may safely put $\varepsilon=0$ in the function%
\[
S_{A}(\lambda,\varepsilon):=\tau\lbrack\log((A-\lambda)^{\ast}(A-\lambda
)+\varepsilon)].
\]
Using the formula (\ref{diffLog}) for the derivative of the trace of the
logarithm, we can easily compute that%
\[
\frac{\partial}{\partial\lambda}S_{A}(\lambda,0)=\tau\lbrack(\lambda
-A)^{-1}]=G_{A}(\lambda).
\]
But since $G_{A}(\lambda)$ depends holomorphically on $\lambda,$ we find that%
\begin{align*}
\Delta_{\lambda}S_{A}(\lambda,0)  &  =4\frac{\partial}{\partial\bar{\lambda}%
}\frac{\partial}{\partial\lambda}S_{A}(\lambda,0)\\
&  =4\frac{\partial}{\partial\bar{\lambda}}G_{A}(\lambda)\\
&  =0,
\end{align*}
so that the Brown measure is zero. This argument shows that the Brown measure
is zero outside the spectrum of $A.$

Now, in the case $A=x_{0}+i\sigma_{t},$ the authors of \cite{JNPWZ} attempt to
determine the maximum set on which the function
\[
\frac{\partial}{\partial\lambda}S(t,\lambda,0)
\]
remains holomorphic. We start with Biane's subordination function identity
(\ref{bianeIdentity}), which we rewrite as follows. Let $\sigma$ be a fixed
semicircular element, so that the law of $\sigma_{t}$ is the same as that of
$\sqrt{t}\sigma.$ Then set $u=\sqrt{t},$ so that (\ref{bianeIdentity}) reads
as
\[
\tau\lbrack\{\lambda+u^{2}G_{x_{0}}(\lambda)-(x_{0}+u\sigma)\}^{-1}]=G_{x_{0}%
}(\lambda).
\]
We then formally analytically continue to $u=i\sqrt{t},$ giving%
\[
\tau\lbrack\{\lambda-tG_{x_{0}}(\lambda)-(x_{0}+i\sqrt{t}\sigma)\}^{-1}%
]=G_{x_{0}}(\lambda).
\]
Thus,%
\[
G_{x_{0}+\sigma_{t}}(\lambda+tG_{x_{0}}(\lambda))=G_{x_{0}}(\lambda
)=G_{x_{0}+i\sigma_{t}}(\lambda-tG_{x_{0}}(\lambda)).
\]
In terms of the maps $H_{t}$ and $J_{t}$ defined in (\ref{htDef}) and
(\ref{jtDef}), respectively, we then have
\[
G_{x_{0}+i\sigma_{t}}(J_{t}(\lambda))=G_{x_{0}+\sigma_{t}}(H_{t}(\lambda))
\]
or%
\begin{equation}
G_{x_{0}+i\sigma_{t}}(J_{t}(H_{t}^{-1}(\lambda)))=G_{x_{0}+\sigma_{t}}%
(\lambda). \label{nowakID}%
\end{equation}
We also note that from the definitions (\ref{htDef}) and (\ref{jtDef}), we
have $J_{t}(\lambda)=2\lambda-H_{t}(\lambda),$ so that%
\begin{equation}
J_{t}(H_{t}^{-1}(z))=2H_{t}^{-1}(z)-z. \label{twoMapsId}%
\end{equation}

Then since the right-hand side of (\ref{nowakID}) is holomorphic on the whole
upper half-plane, the authors of \cite{JNPWZ} argue that the identity
(\ref{nowakID}) actually holds on the whole upper half-plane. If that claim
actually holds, we will have the identity
\begin{equation}
G_{x_{0}+i\sigma_{t}}(z)=G_{x_{0}+\sigma_{t}}(H_{t}(J_{t}^{-1}(z)))
\label{nowakID2}%
\end{equation}
for all $z$ in the range of $J_{t}\circ H_{t}^{-1},$ namely for all $z$ (in
the upper half-plane) outside the closure of $\Omega_{t}.$ An exactly parallel
argument then applies in the lower half-plane. The authors thus wish to
conclude that $G_{x_{0}+i\sigma_{t}}$ is defined and holomorphic on the
complement of $\overline{\Omega}_{t},$ which would show that the Brown measure
of $x_{0}+i\sigma_{t}$ is zero there.

We emphasize that the argument for (\ref{nowakID}) is rigorous for all
sufficiently large $\lambda,$ simply because the quantity $\tau\lbrack
(\lambda-A)^{-1}]$ depends holomorphically on both the complex number
$\lambda$ and the operator $A.$ But just because the right-hand side of the
identity extends holomorphically to the upper half-plane does not by itself
mean that the identity continues to hold on the whole upper half-plane. Thus,
the argument in \cite{JNPWZ} is not entirely rigorous. Nevertheless, it
certainly gives a natural explanation of how the domain $\Omega_{t}$ arises.

The identities (\ref{nowakID}) and (\ref{nowakID2}) already indicate a close
relationship between the operators $x_{0}+i\sigma_{t}$ and $x_{0}+\sigma_{t}.$
In Section \ref{pushforward.sec}, we will find an even closer relationship:
The push-forward of the Brown measure of $x_{0}+i\sigma_{t}$ under a certain
map $Q_{t}:\overline{\Omega}_{t}\rightarrow\mathbb{R}$ is precisely the law of
$x_{0}+\sigma_{t}.$ The map $Q_{t}$ is constructed as follows: It is the
unique map of $\overline{\Omega}_{t}$ to $\mathbb{R}$ that agrees with
$H_{t}\circ J_{t}^{-1}$ on $\partial\Omega_{t}$ and maps vertical segments in
$\overline{\Omega}_{t}$ to points in $\mathbb{R}.$

\section{Outside the domain\label{outside.sec}}

In this section, we show that the Brown measure of $x_{0}+i\sigma_{t}$ is zero
in the complement of the closure of the domain $\Omega_{t}$ in Definition
\ref{omegaT.def}. We outline our strategy in Section \ref{outsideOutline.sec}
and then give a rigorous argument in Section \ref{outsideRigorous.sec}.

\subsection{Outline\label{outsideOutline.sec}}

Our goal is to compute the Laplacian with respect to $\lambda$ of the
function
\[
s_{t}(\lambda)=\lim_{\varepsilon\rightarrow0^{+}}S(t,\lambda,\varepsilon).
\]
We use the Hamilton--Jacobi method of Proposition \ref{HJ.prop}, which gives
us a formula for $S(t,\lambda(t),\varepsilon(t)).$ Since (Proposition
\ref{solveODE.prop}) $\varepsilon(t)=\varepsilon_{0}(1-p_{0}t)^{2}$, we can
attempt to make $\varepsilon(t)$ approach 0 by letting $\varepsilon_{0}$
approach zero. This strategy, however, can only succeed if the lifetime of the
path remains at least $t$ in the limit as $\varepsilon_{0}\rightarrow0.$ Thus,
we must take $\lambda_{0}$ for which $T(\lambda_{0})\geq t,$ where $T$ is as
in Definition \ref{lifetimes.def}. We therefore consider $\lambda_{0}$ in
$\overline{\Lambda}_{t}^{c},$ where $\Lambda_{t}$ is as in Definition
\ref{lambdaT.def}.

If we formally put $\varepsilon_{0}=0$, then $\varepsilon(t)=0$, and, by
Proposition \ref{smallEpsilon0.prop}, we have
\begin{equation}
\lambda(t)=J_{t}(\lambda_{0}). \label{JtLambda0}%
\end{equation}
Now, by Proposition \ref{Jprops.prop}, $J_{t}$ maps $\Lambda_{t}^{c}$
injectively onto $\Omega_{t}^{c}.$ Thus, for any $\lambda\in\overline{\Omega
}_{t}^{c},$ we may choose $\lambda_{0}=J_{t}^{-1}(\lambda)$. Then, if we
\emph{formally} apply the Hamilton--Jacobi formula (\ref{firstHJ}) with
$\varepsilon_{0}=0$, we get
\[
S(t,\lambda,0)=\int_{\mathbb{R}}\log(|J_{t}^{-1}(\lambda)-x|^{2}%
)\,d\mu(x)-\frac{t}{4}(p_{a,0}^{2}-p_{b,0}^{2}),
\]
where, with $\varepsilon_{0}=0,$ the initial momenta in (\ref{initialMomenta})
may be computed as%
\begin{align*}
p_{a,0}  &  =2\int_{\mathbb{R}}\frac{(a_{0}-x)}{(a_{0}-x)^{2}+b_{0}^{2}}%
~d\mu(x)=2\operatorname{Re}\int_{\mathbb{R}}\frac{1}{\lambda_{0}-x}~d\mu(x)\\
p_{b,0}  &  =2\int_{\mathbb{R}}\frac{b_{0}}{(a_{0}-x)^{2}+b_{0}^{2}}%
~d\mu(x)=-2\operatorname{Im}\int_{\mathbb{R}}\frac{1}{\lambda_{0}-x}~d\mu(x).
\end{align*}

Thus,
\begin{equation}
S(t,\lambda,0)=\int_{\mathbb{R}}\log(|J_{t}^{-1}(\lambda)-x|^{2}%
)\,d\mu(x)-t\operatorname{Re}[G_{x_{0}}(J_{t}^{-1}(\lambda_{0}))^{2}],
\label{eq:OutsideS}%
\end{equation}
where%
\[
G_{x_{0}}(\lambda)=\tau((\lambda-x_{0})^{-1})=\int_{\mathbb{R}}\frac
{1}{\lambda-x}~d\mu(x).
\]
The right-hand side of (\ref{eq:OutsideS}) is the composition of a harmonic
function and a holomorphic function and is therefore harmonic. We thus wish to
conclude that Brown measure of $x_{0}+i\sigma_{t}$ is zero outside
$\overline{\Omega}_{t}.$

The difficulty with the preceding argument is that the function $S(t,\lambda
,\varepsilon)$ is only known ahead of time to be defined for $\varepsilon>0.$
Thus, the PDE in Theorem \ref{PDE.thm} is only known to hold when
$\varepsilon>0$ and the Hamilton--Jacobi formula is only valid when
$\varepsilon(t)$ remains positive. We are therefore not allowed to set
$\varepsilon_{0}=0$ in the Hamilton--Jacobi formula (\ref{firstHJ}). Now, if
$\lambda$ is outside the spectrum of $x_{0}+i\sigma_{t},$ then we can see that
$S(t,\lambda,\varepsilon)$ continues to make sense for $\varepsilon=0$ and
even for $\varepsilon$ slightly negative, and the PDE and Hamilton--Jacobi
formula presumably apply. But of course we do not know that every point in
$\overline{\Omega}_{t}$ is outside the spectrum of $x_{0}+i\sigma_{t}$; if we
did, an elementary property of the Brown measure would already tell us that
the Brown measure is zero there.

If, instead, we let $\varepsilon_{0}$ \textit{approach} zero from above, we
find that
\begin{equation}
\lim_{\varepsilon_{0}\rightarrow0^{+}}S(t,\lambda(t),\varepsilon
(t))=\int_{\mathbb{R}}\log(|J_{t}^{-1}(\lambda)-x|^{2})\,d\mu(x)-t\mathrm{Re}%
\,[G_{x_{0}}(J_{t}^{-1}(\lambda))^{2}]. \label{outsideLim}%
\end{equation}
Now, as $\varepsilon_{0}\rightarrow0^{+},$ we can see that $\lambda(t)$
approaches $\lambda$ and $\varepsilon(t)$ approaches 0. But there is still a
difficulty, because the function $s_{t}(\lambda)$ is defined as the limit of
$S(t,\lambda,\varepsilon)$ as $\varepsilon$ tends to zero \textit{with
}$\lambda$\textit{ fixed}. But on the left-hand side of (\ref{outsideLim}),
$\lambda=\lambda(t)$ is not fixed, because it depends on $\varepsilon_{0}.$ To
overcome this difficulty, we will use the inverse function theorem to show
that, for each $t>0,$ the function $S$ has an extension to a neighborhood of
$(t,\lambda,0)$ that is continuous in the $\lambda$- and $\varepsilon
$-variables. Thus, the limit of $S$ along \textit{any} path approaching
$(t,\lambda,0)$ is the same as the limit with $\lambda$ fixed and
$\varepsilon$ tending to zero.

\subsection{Rigorous treatment\label{outsideRigorous.sec}}

In this section, we establish the following rigorous version of
(\ref{eq:OutsideS}), which shows that the support of the Brown measure of
$x_{0}+i\sigma_{t}$ is contained in $\overline{\Omega}_{t}$. Recall that
$s_{t}(\lambda)$ is the limit of $S(t,\lambda,\varepsilon)$ as $\varepsilon$
approaches zero from above with $\lambda$ fixed.

\begin{theorem}
\label{thm:outside}If $\lambda$ is not in $\overline{\Omega}_{t},$ we have
\begin{equation}
s_{t}(\lambda)=\int_{\mathbb{R}}\log(|J_{t}^{-1}(\lambda)-x|^{2}%
)\,d\mu(x)-t\mathrm{Re}\,[G_{x_{0}}(J_{t}^{-1}(\lambda))^{2}] \label{stResult}%
\end{equation}
and $\Delta s_{t}(\lambda)=0$.
\end{theorem}

The theorem will follow from the argument in Section \ref{outsideOutline.sec},
once the following regularity result is established.

\begin{proposition}
\label{outsideRegularity.prop}Fix a time $t>0$ and a point $\lambda^{\ast}%
\in\overline{\Omega}_{t}^{c}.$ Then the function $(\lambda,\varepsilon)\mapsto
S(t,\lambda,\varepsilon)$ extends to a real analytic function defined in a
neighborhood of $(\lambda^{\ast},0)$ inside $\mathbb{C}\times\mathbb{R}.$
\end{proposition}

We will need the following preparatory result.

\begin{lemma}
\label{notIntersect.lem}If $\lambda_{0}$ is not in $\overline{\Lambda}_{t}$,
there is a neighborhood of $\lambda_{0}$ in $\overline{\Lambda}_{t}^{c}$ that
does not intersect $\mathrm{supp}(\mu)$.
\end{lemma}

The result of this lemma does not hold if we replace $\overline{\Lambda}%
_{t}^{c}$ by $\Lambda_{t}^{c}$. As a counter-example, if $\mu=3x^{2}\,dx$ on
$[0,1],$ then using the criterion (\ref{LambdaTIntersectR}) for $\Lambda
_{t}\cap\mathbb{R},$ we find that $0\in\Lambda_{t}^{c}$ for small enough $t,$
but $0\in\mathrm{supp}(\mu)$.

\begin{proof}
It is clear that the statement of this lemma holds unless $\lambda_{0}%
\in\mathbb{R}$ since $x_{0}$ is self-adjoint.

Consider, then, a point $\lambda_{0}\in\overline{\Lambda}_{t}^{c}%
\cap\mathbb{R}.$ Choose an interval $(\alpha,\beta)$ around $\lambda_{0}$
contained in $\overline{\Lambda}_{t}^{c}\cap\mathbb{R}.$ We claim that
$\mu((\alpha,\beta))$ must be zero. To see, note that since the points in
$(\alpha,\beta)$ are outside $\Lambda_{t},$ we have (Definition
\ref{lambdaT.def})%
\[
\int_{\mathbb{R}}\frac{1}{(a_{0}-x)^{2}}\,d\mu(x)\leq\frac{1}{t}%
\]
for all $a_{0}\in(\alpha,\beta)$. If we integrate the above integral with
respect to the Lebesgue measure in $a_{0}$, we have
\[
\int_{\alpha}^{\beta}\int_{\mathbb{R}}\frac{1}{(a_{0}-x)^{2}}\,d\mu
(x)\,da_{0}<\infty.
\]
We may then reverse the order of integration and restrict the integral with
respect to $\mu$ to $(\alpha,\beta)$ to get
\begin{equation}
\int_{\alpha}^{\beta}\int_{\alpha}^{\beta}\frac{1}{(a_{0}-x)^{2}}%
\,da_{0}\,d\mu(x)<\infty. \label{eq:vt0contra}%
\end{equation}
But
\[
\int_{\alpha}^{\beta}\frac{1}{(a_{0}-x)^{2}}\,da_{0}=\infty
\]
for all $x\in(\alpha,\beta)$. Thus, the only way (\ref{eq:vt0contra}) can hold
is if $\mu((\alpha,\beta)=0.$
\end{proof}

We now work toward the proof of Proposition \ref{outsideRegularity.prop}. In
light of the formulas for the solution path in Proposition \ref{solveODE.prop}%
, we consider the map $V_{t}$ given by $V_{t}(a_{0},b_{0},\varepsilon
_{0})=(a_{t},b_{t},\varepsilon_{t}),$ where
\begin{align*}
a_{t}(a_{0},b_{0},\varepsilon_{0})  &  =a_{0}-\frac{t}{2}p_{a,0}\\
b_{t}(a_{0},b_{0},\varepsilon_{0})  &  =b_{0}+\frac{t}{2}p_{b,0}\\
\varepsilon_{t}(a_{0},b_{0},\varepsilon_{0})  &  =\varepsilon_{0}%
(1-p_{0}t)^{2}.
\end{align*}
This map is initially defined for $\varepsilon_{0}>0,$ which guarantees that
the integrals (\ref{initialMomenta}) defining $p_{a,0}$ and $p_{b,0}$ are
convergent, even if $b_{0}=0.$ But if $\lambda_{0}$ is in $\overline{\Lambda
}_{t}^{c},$ Lemma \ref{notIntersect.lem} guarantees that $\lambda_{0}$ is
outside the closed support of $\mu,$ so that the integrals are convergent when
$\varepsilon_{0}=0$ and even when $\varepsilon_{0}$ is slightly negative.
Thus, for any $\lambda_{0}\in\overline{\Lambda}_{t}^{c}$, we can extend
$V_{t}$ to a neighborhood of $(\lambda_{0},0)$ using the same formula. We note
that when $\varepsilon_{0}=0,$ we have
\begin{equation}
a_{t}(a_{0},b_{0},0)+ib_{t}(a_{0},b_{0},0)=J_{t}(a_{0}+ib_{0}),
\label{atPlusIbt}%
\end{equation}
as in (\ref{JtLambda0}).

\begin{lemma}
\label{outsideIVT.lem}If $\lambda_{0}$ is not in $\overline{\Lambda}_{t},$ the
Jacobian matrix of $V_{t}$ at $(\lambda_{0},0)$ is invertible.
\end{lemma}

\begin{proof}
If we vary $a_{0}$ or $b_{0}$ with $\varepsilon_{0}$ held equal to $0,$ then
$\varepsilon$ remains equal to zero, so that
\[
\frac{\partial\varepsilon_{t}}{\partial a_{0}}(\lambda_{0},0)=\frac
{\partial\varepsilon_{t}}{\partial b_{0}}(\lambda_{0},0)=0.
\]
Meanwhile, from the formula for $\varepsilon_{t},$ we obtain
\[
\frac{\partial\varepsilon_{t}}{\partial\varepsilon_{0}}(\lambda_{0}%
,0)=(1-tp_{0})^{2}.
\]
Thus, using (\ref{atPlusIbt}), we find that the Jacobian matrix of $V$ at
$(a_{0},b_{0},0)$ has the form
\[
\left(
\begin{array}
[c]{cc}%
K & \ast\\
0 & (1-tp_{0})^{2}%
\end{array}
\right)  ,
\]
where $K$ is the $2\times2$ Jacobian matrix of the map $J_{t}$.

Since $\lambda_{0}\in\overline{\Lambda}_{t}^{c},$ we have $T(\lambda
_{0})=1/p_{0}>1/t,$ so that $1-tp_{0}>0.$ Furthermore, since $J_{t}$ is
injective on $\overline{\Lambda}_{t}^{c},$ its complex derivative must be
nonzero at $\lambda_{0},$ so that $K$ is invertible. We can then see that the
Jacobian matrix of $V_{t}$ at $(t,\lambda_{0},0)$ has nonzero determinant.
\end{proof}

We are now ready for the proof of our regularity result.

\begin{proof}
[Proof of Propositon \ref{outsideRegularity.prop}]Define a function
$\mathrm{HJ}$ by the right-hand side of the first Hamilton--Jacobi formula
(\ref{firstHJ}), namely,%
\begin{equation}
\mathrm{HJ}(a_{0},b_{0},\varepsilon_{0},t)=S(0,\lambda_{0},\varepsilon
_{0})-t\left[  \frac{1}{4}(p_{a,0}^{2}-p_{b,0}^{2})-\varepsilon_{0}p_{0}%
^{2}\right]  . \label{HJdef}%
\end{equation}
Now take $\lambda^{\ast}\in\overline{\Omega}_{t}^{c}$ and let $\lambda
_{0}^{\ast}=J_{t}^{-1}(\lambda^{\ast}),$ so that $\lambda_{0}^{\ast}%
\in\overline{\Lambda}_{t}^{c}.$ By Lemma \ref{outsideIVT.lem} and the inverse
function theorem, $V_{t}$ has an analytic inverse in a neighborhood $U$ of
$(\lambda^{\ast},0).$ By shrinking $U$ if necessary, we can assume that the
$\lambda_{0}$-component of $V^{-1}(\lambda,\varepsilon)$ lies in
$\overline{\Lambda}_{t}^{c}$ for all $(\lambda,\varepsilon)$ in $U.$ We now
claim that for each fixed $t>0,$ the map%
\begin{equation}
(\lambda,\varepsilon)\mapsto\mathrm{HJ}\circ V_{t}^{-1}(\lambda,\varepsilon)
\label{Sextension}%
\end{equation}
gives the desired analytic extension of $S(t,\cdot,\cdot)$ to a neighborhood
of $(\lambda^{\ast},0).$

We first note that $\mathrm{HJ}\circ V_{t}^{-1}$ is smooth, where we use Lemma
\ref{notIntersect.lem} to guarantee that the momenta in the definition of
$\mathrm{HJ}$ are well defined. We then argue that for all $(\lambda
,\varepsilon)$ in $U$ with $\varepsilon>0,$ the value of $\mathrm{HJ}\circ
V_{t}^{-1}(\lambda,\varepsilon)$ agrees with $S(t,\lambda,\varepsilon).$ To
see this, note first that if $(\lambda,\varepsilon)\in U$ has $\varepsilon>0,$
then the $\varepsilon_{0}$-component of $V_{t}^{-1}(\lambda,\varepsilon)$ must
be positive, as is clear from the formula for $\varepsilon_{t}(a_{0}%
,b_{0},\varepsilon_{0}).$ Since, also, the $\lambda_{0}$-component of
$V_{t}^{-1}(\lambda,\varepsilon)$ is in $\overline{\Lambda}_{t}^{c},$ the
small-$\varepsilon_{0}$ lifetime of the path is at least $t,$ so that when
$\varepsilon_{0}>0,$ the lifetime is greater than $t.$ Thus, for
$(\lambda,\varepsilon)$ in $U$ with $\varepsilon>0,$ the first
Hamilton--Jacobi formula (\ref{firstHJ}) tells us that, indeed, $S(t,\lambda
,\varepsilon)=\mathrm{HJ}(V_{t}^{-1}(\lambda,\varepsilon)).$
\end{proof}

We now come to the proof of our main result.

\begin{proof}
[Proof of Theorem \ref{thm:outside}]Once we know that (\ref{Sextension}) gives
an analytic extension of $S(t,\cdot,\cdot),$ we conclude that the function
$s_{t}$ defined as%
\[
s_{t}(\lambda)=\lim_{\varepsilon\rightarrow0}S(t,\lambda,\varepsilon)
\]
can be computed as%
\begin{equation}
s_{t}(\lambda)=\mathrm{HJ}\circ V_{t}^{-1}(\lambda,0). \label{stFormula}%
\end{equation}
The point of this observation is that because $\mathrm{HJ}\circ V_{t}^{-1}$ is
analytic (in particular, continuous), we can compute $s_{t}(\lambda)$ by
taking the limit of $\mathrm{HJ}\circ V_{t}^{-1}(\delta,\varepsilon)$ along
\textit{any} path ending at $(\lambda,0),$ rather than having to fix $\lambda$
and let $\varepsilon$ tend to zero.

Fix a point $\lambda$ in $\overline{\Omega}_{t}^{c}$ and let $\lambda
_{0}=J_{t}^{-1}(\lambda)$, so that $T(\lambda_{0})\geq t.$ Then for any
$\varepsilon_{0}>0,$ the lifetime of the path with initial conditions
$(\lambda_{0},\varepsilon_{0})$ will be greater than $t$ and the first
Hamilton--Jacobi formula (\ref{firstHJ}) tells us that
\[
S(t,\lambda(t),\varepsilon(t))=S(0,\lambda_{0},\varepsilon_{0})-t\left[
\frac{1}{4}(p_{a,0}^{2}-p_{b,0}^{2})-\varepsilon_{0}p_{0}^{2}\right]  .
\]
As $\varepsilon_{0}\rightarrow0$, we find that $\lambda(t)\rightarrow
J_{t}(\lambda_{0})=\lambda$ and $\varepsilon(t)\rightarrow0.$ Thus, by
(\ref{stFormula}) and the continuity of $\mathrm{HJ}\circ V_{t}^{-1},$ we have%
\[
s_{t}(\lambda)=S(0,\lambda_{0},0)-t\lim_{\varepsilon_{0}\rightarrow0}\left[
\frac{1}{4}(p_{a,0}^{2}-p_{b,0}^{2})-\varepsilon_{0}p_{0}^{2}\right]  ,
\]
which gives the claimed expression (\ref{stResult}).

Now, if $\lambda$ is outside of $\overline{\Omega}_{t},$ then $J_{t}%
^{-1}(\lambda)$ is outside of $\bar{\Lambda}_{t},$ which means (Lemma
\ref{notIntersect.lem}) that $J_{t}^{-1}(\lambda)$ is outside the support of
the measure $\mu.$ It is then easy to see that $s_{t}$ is a composition of a
harmonic function and a holomorphic function, which is harmonic.
\end{proof}

\section{Inside the domain\label{inside.sec}}

\subsection{Outline}

In Section \ref{outside.sec}, we computed the Brown measure in the complement
of $\overline{\Omega}_{t}$ and found that it is zero there. Our strategy was
to apply the Hamilton--Jacobi formulas with $\lambda_{0}$ in the complement of
$\overline{\Lambda}_{t}$ and $\varepsilon_{0}$ chosen to be very small, so
that $\lambda(t)$ is in the complement of $\overline{\Omega}_{t}$ and
$\varepsilon(t)$ is also very small. If, on the other hand, we take
$\lambda_{0}$ inside $\Lambda_{t}$, then (by definition) $T(\lambda_{0})<t,$
meaning that the small-$\varepsilon_{0}$ lifetime of the path is less than
$t.$ Thus, for $\lambda_{0}\in\Lambda_{t}$ and $\varepsilon_{0}$ small, the
Hamilton--Jacobi formulas are not applicable at time $t.$

In this section, then, we will use a different strategy. We recall from
Proposition \ref{solveODE.prop} that $\varepsilon(t)=\varepsilon_{0}%
(1-p_{0}t)^{2}.$ Thus, an alternative way to make $\varepsilon(t)$ small is to
take $\varepsilon_{0}>0$ and arrange for $p_{0}$ to be close to $1/t.$ Thus,
for each point $\lambda$ in $\Omega_{t},$ we will try to find $\lambda_{0}%
\in\Lambda_{t}$ and $\varepsilon_{0}>0$ so that $p_{0}=1/t$ and $\lambda
(t)=\lambda.$ (If $p_{0}=1/t$ then the solution to the system of ODEs blows up
at time $t,$ so that technically we are not allowed to apply the
Hamilton--Jacobi formulas at time $t.$ But we will gloss over this point for
now and return to it in Section \ref{insideTechnical.sec}.)

Once we have understood how to choose $\lambda_{0}$ and $\varepsilon_{0}$ as
functions of $\lambda\in\Omega_{t},$ we will then apply the Hamilton--Jacobi
method to compute the Brown measure inside $\Omega_{t}.$ Specifically, we will
use the second Hamilton--Jacobi formula (\ref{secondHJ}) to compute the first
derivatives of $S(t,\lambda,0)$ with respect to $a$ and $b.$ We then compute
the second derivatives to get the density of the Brown measure.

\subsubsection{Mapping onto $\Omega_{t}$}

We first describe how to choose $\lambda_{0}$ and $\varepsilon_{0}>0$ as
functions of $\lambda\in\Omega_{t}$ so that $\lambda(t)=\lambda$ and
$\varepsilon(t)=0.$ If $a_{0}+ib_{0}\in\Lambda_{t},$ then $\left\vert
b_{0}\right\vert <v_{t}(a_{0}).$ Then from the defining property (\ref{vtDef})
of the function $v_{t},$ we see that if we take%
\begin{equation}
\varepsilon_{0}=\varepsilon_{0}^{t}(a_{0}):=v_{t}(a_{0})^{2}-b_{0}^{2},
\label{epsilon0fromA0}%
\end{equation}
then $\varepsilon_{0}$ is positive and plugging this value of $\varepsilon
_{0}$ into the formula (\ref{initialMomenta}) for $p_{0}$ gives%
\[
p_{0}=\int_{\mathbb{R}}\frac{1}{(a_{0}-x)^{2}+v_{t}(a_{0})^{2}}~d\mu
(x)=\frac{1}{t},
\]
as desired.

It remains to see how to choose $\lambda_{0}$ so that (with $\varepsilon_{0}$
given by (\ref{epsilon0fromA0})) we will have $\lambda(t)=\lambda.$ Since
$p_{0}=1/t,$ Proposition \ref{solveODEspecial.prop} applies:
\begin{align}
a(t)  &  =t\int_{\mathbb{R}}\frac{1}{(a_{0}-x)^{2}+v_{t}(a_{0})^{2}}%
~d\mu(x)\label{p1condition}\\
b(t)  &  =2b_{0}. \label{bEquals}%
\end{align}
If we want $\lambda(t)$ to equal $\lambda=a+ib,$ then (\ref{bEquals})
immediately tells us that we should choose $b_{0}=b/2.$ We will show in
Section \ref{sect:surjectivity} that (\ref{p1condition}) can be solved for
$a_{0}$ as a function of $a$ and $t$; we use the notation $a_{0}^{t}(a)$ for
the solution.

\begin{summary}
\label{surject.summary}For all $\lambda=a+ib\in\Omega_{t},$ the following
procedure shows how to choose $\lambda_{0}=a_{0}+ib_{0}\in\Lambda_{t}$ and
$\varepsilon_{0}>0$ so that, with these initial conditions, we will have
$\lambda(t)=\lambda$ and $\varepsilon(t)=0.$ First, we use the condition%
\[
\int_{\mathbb{R}}\frac{1}{(a_{0}-x)^{2}+v_{t}(a_{0})^{2}}~d\mu(x)=\frac{1}{t}%
\]
to determine $v_{t}$ as a function of $a_{0}.$ Second, use the condition
\[
\int_{\mathbb{R}}\frac{x}{(a_{0}-x)^{2}+v_{t}(a_{0})^{2}}~d\mu(x)=\frac{a}{t}%
\]
to determine $a_{0}$ as a function $a_{0}^{t}$ of $a.$ Then we take
\begin{align*}
b_{0}  &  =b/2\\
\varepsilon_{0}  &  =v_{t}(a_{0}^{t}(a))^{2}-b_{0}^{2}.
\end{align*}

\end{summary}

\subsubsection{Computing the Brown measure}

Using the choices for $\lambda_{0}$ and $\varepsilon_{0}$ in Summary
\ref{surject.summary}, we then apply the second Hamilton--Jacobi formula
(\ref{secondHJ}). Since $\lambda(t)=\lambda$ and $\varepsilon(t)=0$ and
$p_{b}$ is a constant of motion,%
\[
\frac{\partial S}{\partial b}(t,\lambda,0)=p_{b}(t)=p_{b,0}.
\]
But since, by (\ref{pb0}), $p_{b,0}=2b_{0}p_{0},$ we obtain%
\[
\frac{\partial S}{\partial b}(t,\lambda,0)=2b_{0}p_{0}=\frac{b}{t},
\]
since we are assuming that $p_{0}=1/t.$

Similarly,%
\begin{align*}
\frac{\partial S}{\partial a}(t,\lambda,0)  &  =p_{a}(t)\\
&  =p_{a,0}\\
&  =2a_{0}p_{0}-2p_{1}.\\
&  =\frac{2a_{0}^{t}(a)}{t}-\frac{2a}{t},
\end{align*}
where we have used (\ref{p1condition}) and the formula (\ref{pa0}) for
$p_{a,0}.$

\begin{conclusion}
\label{sDerivatives.conclusion}The preceding argument suggests that for
$\lambda=a+ib\in\Omega_{t},$ we should have%
\begin{align*}
\frac{\partial s_{t}}{\partial a}  &  =\frac{2}{t}(a_{0}^{t}(a)-a)\\
\frac{\partial s_{t}}{\partial b}  &  =\frac{b}{t}.
\end{align*}
If this is correct, then the density of the Brown measure in $\Omega_{t}$ is
readily computed as%
\[
\frac{1}{4\pi}\left(  \frac{\partial^{2}S}{\partial a^{2}}+\frac{\partial
^{2}S}{\partial b^{2}}\right)  (t,\lambda,0)=\frac{1}{2\pi t}\left(
\frac{da_{0}^{t}(a)}{da}-\frac{1}{2}\right)  ,
\]
as claimed in Theorem \ref{intro.thm}. In particular, the density of the Brown
measure in $\Lambda_{t}$ would be independent of $b=\operatorname{Im}\lambda.$
\end{conclusion}

\subsubsection{Technical issues\label{insideTechnical.sec}}

The preceding argument is not rigorous, since the Hamilton--Jacobi formulas
are only known to hold as long as $\varepsilon(s)$ remains positive for all
$0\leq s\leq t.$ That is to say, if $\varepsilon(t)=0$ then we are not allowed
to use the formulas at time $t.$ We can try to work around this point by
letting $\varepsilon_{0}$ \textit{approach} the value $\varepsilon_{0}%
^{t}(\lambda_{0}):=v_{t}(a_{0})^{2}-b_{0}^{2}$ in (\ref{epsilon0fromA0}) from
above. Then we have a situation similar to the one in (\ref{outsideLim}),
namely%
\begin{align}
\lim_{\varepsilon_{0}\rightarrow\varepsilon_{0}^{t}(\lambda_{0})^{+}}%
\frac{\partial S}{\partial a}(t,\lambda(t),\varepsilon(t))  &  =\frac{2}%
{t}(a_{0}^{t}(a)-a)\label{InsideLim1}\\
\lim_{\varepsilon_{0}\rightarrow\varepsilon_{0}^{t}(\lambda_{0})^{+}}%
\frac{\partial S}{\partial b}(t,\lambda(t),\varepsilon(t))  &  =\frac{b}{t},
\label{InsideLim2}%
\end{align}
where%
\[
\lim_{\varepsilon_{0}\rightarrow\varepsilon_{0}^{t}(\lambda_{0})}%
\lambda(t)=\lambda;\quad\lim_{\varepsilon_{0}\rightarrow\varepsilon_{0}%
^{t}(\lambda_{0})}\varepsilon(t)=0.
\]

But the Brown measure is computed by first evaluating the limit%
\[
s_{t}(\lambda):=\lim_{\varepsilon\rightarrow0^{+}}S(t,\lambda,\varepsilon),
\]
where the limit is taken as $\varepsilon\rightarrow0$ \textit{with }$\lambda
$\textit{\ fixed}, and then taking the distributional Laplacian with respect
to $\lambda.$ Since $\lambda(t)$ is not fixed in (\ref{InsideLim1}) and
(\ref{InsideLim2}), it is not clear that these limits are actually computing
$\partial s_{t}/\partial a$ and $\partial s_{t}/\partial b.$ The main
technical challenge of this section is, therefore, to establish enough
regularity of $S$ near $(t,\lambda,\varepsilon)$ to verify that $\partial
s_{t}/\partial a$ and $\partial s_{t}/\partial b$ are actually given by the
right-hand sides of (\ref{InsideLim1}) and (\ref{InsideLim2}).

\subsection{Surjectivity\label{sect:surjectivity}}

In this section, we show that the procedure in Summary \ref{surject.summary}
actually gives a continuous map of $\Lambda_{t}$ onto $\Omega_{t}.$ Given any
$\lambda_{0}\in\Lambda_{t}$, choose $\varepsilon_{0}=\varepsilon_{0}%
^{t}(\lambda_{0})$ as in (\ref{epsilon0fromA0}), so that
\[
\lim_{s\rightarrow t}\varepsilon(s)=0.
\]
Then define
\[
U_{t}(\lambda_{0})=\lim_{s\rightarrow t}\lambda(s).
\]
By Proposition \ref{solveODEspecial.prop}, we have
\[
U_{t}(a_{0}+ib_{0})=a_{t}(a_{0})+2ib_{0}%
\]
where
\begin{equation}
a_{t}(a_{0})=t\int_{\mathbb{R}}\frac{1}{(a_{0}-x)^{2}+v_{t}(a_{0})^{2}}%
\,d\mu(x). \label{atFormulas}%
\end{equation}
Since we assume $\lambda_{0}\in\Lambda_{t},$ we have $v_{t}(a_{0})>0$ and we
therefore have an alternative formula:%
\begin{equation}
a_{t}(a_{0})=\operatorname{Re}[J_{t}(a_{0}+iv_{t}(a_{0}))]. \label{atAlt}%
\end{equation}
It is a straightforward computation to check that the right-hand sides of
(\ref{atFormulas}) and (\ref{atAlt}) agree, using that the identity
(\ref{vtDef}) holds when $v_{t}(a_{0})>0.$

The main result in this section is stated in the following theorem. We remind
the reader of the definition (\ref{jtDef}) of the map $J_{t}.$

\begin{theorem}
\label{thm:lambda0tolambda}The following results hold.

\begin{enumerate}
\item The map $U_{t}$ extends continuously to $\overline{\Lambda}_{t}.$ This
extension is the unique continuous map of $\overline{\Lambda}_{t}$ into
$\overline{\Omega}_{t}$ that (a) agrees with $J_{t}$ on $\partial\Lambda_{t}$
and (b) maps each vertical segment in $\overline{\Lambda}_{t}$ linearly to a
vertical segment in $\overline{\Omega}_{t}.$

\item The map $U_{t}$ is a homeomorphism from $\Lambda_{t}$ onto $\Omega_{t}$.
\end{enumerate}
\end{theorem}

Most of what we need to prove the theorem is already in Proposition
\ref{Jprops.prop}.

\begin{proof}
Proposition \ref{Jprops.prop} showed that the right-hand side of (\ref{atAlt})
is continuous for all $a_{0}\in\mathbb{R}.$ Using this formula for $a_{t},$ we
see that $U_{t}$ actually extends continuously to the whole complex plane. It
is then a simple computation to check that
\[
\operatorname{Im}J_{t}(a_{0}\pm iv_{t}(a_{0}))=\pm2v_{t}(a_{0}).
\]
This formula, together with (\ref{atAlt}), shows that $U_{t}$ agrees with
$J_{t}$ for all points in $\partial\Lambda_{t}$ having nonzero imaginary
parts. Then points in $\partial\Lambda_{t}$ on the real axis are limits of
points in $\partial\Lambda_{t}$ with nonzero imaginary parts. Thus, $U_{t}$
indeed agrees with $J_{t}$ on $\partial\Lambda_{t}.$ Also $U_{t}$ is linear on
each vertical segment. Since $\Lambda_{t}$ is bounded by the graphs of $v_{t}$
and $-v_{t},$ it is easy to see that $U_{t}$ is the \textit{unique} map with
these two properties.

By Proposition \ref{Jprops.prop}, $\Omega_{t}$ is bounded by the graphs
$b_{t}$ and $-b_{t},$ where the graph of $b_{t}$ is the image of the graph of
$v_{t}$ under $J_{t}.$ From this result, it follows easily that $U_{t}$ is a homeomorphism.
\end{proof}

We conclude this section by giving bounds on the real parts of points in
$\Omega_{t}$, in terms of the law $\mu$ of $x_{0}.$

\begin{proposition}
\label{prop:OmegatBound} Let
\[
M=\sup~\mathrm{supp}(\mu),\quad m=\inf~\mathrm{supp}(\mu).
\]
Then
\[
m<\inf(\Omega_{t}\cap\mathbb{R})\quad\mathrm{and}\quad\sup(\Omega_{t}%
\cap\mathbb{R})<M.
\]
In particular, every point $\lambda$ in $\overline{\Omega}_{t}$ has
$m<\operatorname{Re}\lambda<M.$
\end{proposition}

\begin{proof}
Let $\tilde{a}_{0}=\sup(\Lambda_{t}\cap\mathbb{R})$. Then $v_{t}(\tilde{a}%
_{0})=0$, which means (Proposition \ref{vt.prop}) that
\[
\int_{\mathbb{R}}\frac{d\mu(x)}{(\tilde{a}_{0}-x)^{2}}\leq\frac{1}{t}.
\]
Then
\[
\sup(\Omega_{t}\cap\mathbb{R})=a_{t}(\tilde{a}_{0})=t\int_{\mathbb{R}}%
\frac{x\,d\mu(x)}{(\tilde{a}_{0}-x)^{2}}\leq M.
\]
Because of our standing assumption that $\mu$ is not a $\delta$-measure, this
inequality is strict. The inequality for $\inf(\Omega_{t}\cap\mathbb{R})$ can
be proved similarly.
\end{proof}

\subsection{Regularity\label{inside.IVT}}

Define a function $\tilde{S}$ by
\[
\tilde{S}(t,\lambda,z)=S(t,\lambda,z^{2})
\]
for $z>0$.

\begin{proposition}
\label{insideRegularity.prop}Fix a time $t>0$ and a point $\lambda^{\ast}%
\in\Omega_{t}.$ Then the function $(\lambda,z)\mapsto\tilde{S}(t,\lambda,z)$
extends to a real analytic function defined in a neighborhood of
$(\lambda^{\ast},0)$ inside $\mathbb{C}\times\mathbb{R}.$
\end{proposition}

Once the proposition is established, the function $s_{t}(\lambda
):=\lim_{\varepsilon\rightarrow0^{+}}S(t,\lambda,\varepsilon)$ can be computed
as $s_{t}(\lambda)=\tilde{S}(t,\lambda,0).$ Since $\tilde{S}(t,\lambda,z)$ is
smooth in $\lambda$ and $z,$ we can compute $s_{t}$ (or any of its
derivatives) at $\lambda^{\ast}$ by evaluating $\tilde{S}(t,\lambda,z)$ (or
any of its derivatives) along any path where $\lambda\rightarrow\lambda^{\ast
}$ and $z\rightarrow0.$ Thus, Proposition \ref{insideRegularity.prop} will
allow us to make rigorous the argument leading to
\ref{sDerivatives.conclusion}. Specifically, we will be able to conclude that
the left-hand sides of (\ref{InsideLim1}) and (\ref{InsideLim2}) are actually
equal to $\partial s_{t}/\partial a$ and $\partial s_{t}/\partial b,$ respectively.

\begin{remark}
The function $S$ itself does not have a smooth extension of the same sort that
$\tilde{S}$ does. Indeed, since $\sqrt{\varepsilon}p_{\varepsilon}$ is a
constant of motion, the second Hamilton--Jacobi formula (\ref{secondHJ}) tells
us that $\partial S/\partial\varepsilon$ must blow up like $1/\sqrt
{\varepsilon}$ as we approach $(t,\lambda^{\ast},0)$ along a solution of the
system (\ref{HamSystem}). The same reasoning tells us that the extended
$\tilde{S}$ does \emph{not} satisfy $\tilde{S}(t,\lambda,z)=S(t,\lambda
,z^{2})$ for $z<0$. Indeed, since $\sqrt{\varepsilon}p_{\varepsilon}$ is a
constant of motion, $\frac{\partial\tilde{S}}{\partial z}(t,\lambda
,z)=2\sqrt{\varepsilon}\frac{\partial S}{\partial\varepsilon}(t,\lambda
,z^{2})$ has a nonzero limit as $z\rightarrow0$. Thus, $\tilde{S}$ cannot have
a smooth extension that is even in $z$.
\end{remark}

To prove Proposition \ref{insideRegularity.prop}, we will use a strategy
similar to the one in Section \ref{outsideRigorous.sec}. For each $t>0,$ we
define a map
\[
W_{t}(a_{0},b_{0},\varepsilon_{0})=(a_{t},b_{t},z_{t})
\]
by
\begin{align*}
a_{t}  &  =a(t,a_{0},b_{0},\varepsilon_{0})\\
b_{t}  &  =b(t,a_{0},b_{0},\varepsilon_{0})\\
z_{t}  &  =\sqrt{\varepsilon(t,a_{0},b_{0},\varepsilon_{0})}%
\end{align*}
where $a$, $b$, $\varepsilon$ are defined as in Proposition
\ref{solveODE.prop}. The last component $z_{t}$ can be expressed explicitly
as
\[
z_{t}=\sqrt{\varepsilon_{0}}\left(  1-tp_{0}\right)  .
\]

The map $W_{t}$ is initially defined only for%
\[
\varepsilon_{0}>\varepsilon_{0}^{t}(\lambda_{0}):=v_{t}(a_{0})^{2}-b_{0}^{2}.
\]
This condition guarantees that $p_{0}<1/t,$ so that the lifetime of the path
is greater than $t.$ But for each $t>0$ and $\lambda_{0}\in\Lambda_{t},$ we
can extend $W_{t}$ to a neighborhood of $(a_{0},b_{0},\varepsilon_{0}%
^{t}(\lambda_{0}))$, simply by using the same formulas. We note that if
$\varepsilon_{0}>\varepsilon_{0}^{t}(\lambda_{0}),$ then $p_{0}<1/t$ so that
$z_{t}>0$; and if $\varepsilon_{0}<\varepsilon_{0}^{t}(\lambda_{0})$ then
$p_{0}>1/t$ so that $z_{t}<0.$

\begin{lemma}
\label{insideIVT.lem}For all $t>0$ and $\lambda_{0}=a_{0}+ib_{0}\in\Lambda
_{t},$ the Jacobian of $W_{t}$ at $(a_{0},b_{0},\varepsilon_{0}^{t}%
(\lambda_{0}))$ is invertible, where $\varepsilon_{0}^{t}(\lambda_{0}%
)=v_{t}(a_{0})^{2}-b_{0}^{2}.$
\end{lemma}

\begin{proof}
We introduce the notations
\[%
\begin{split}
q_{0}  &  =\int_{\mathbb{R}}\frac{d\mu(x)}{((a_{0}-x)^{2}+v_{t}(a_{0}%
)^{2})^{2}}\\
q_{1}  &  =\int_{\mathbb{R}}\frac{(a_{0}-x)\,d\mu(x)}{((a_{0}-x)^{2}%
+v_{t}(a_{0})^{2})^{2}}\\
q_{2}  &  =\int_{\mathbb{R}}\frac{(a_{0}-x)^{2}\,d\mu(x)}{((a_{0}-x)^{2}%
+v_{t}(a_{0})^{2})^{2}}.
\end{split}
\]
Note that $q_{0}>0$ and $q_{2}>0$. When $\varepsilon_{0}=\varepsilon_{0}%
^{t}(\lambda_{0})$, we can write $p_{0}$ in terms of $q_{0}$ and $q_{2}$ as
\begin{equation}
p_{0}=q_{2}+v_{t}(a_{0})^{2}q_{0}. \label{p:in.q}%
\end{equation}

We now compute the Jacobian matrix of $W_{t}$ at the point $(\lambda
_{0},\varepsilon_{0}^{t}(\lambda_{0})),$ with $\lambda_{0}\in\Lambda_{t}.$
Using the formulas (\ref{initialMomenta}) for $p_{a,0}$ and $p_{b,0},$ we can
compute that
\[%
\begin{array}
[c]{lll}%
\frac{\partial p_{a,0}}{\partial a_{0}}=2(-q_{2}+v_{t}(a_{0})^{2}q_{0}) &
\frac{\partial p_{a,0}}{\partial b_{0}}=-4b_{0}q_{1} & \frac{\partial p_{a,0}%
}{\partial\varepsilon_{0}}=-2q_{1}\\[3pt]%
\frac{\partial p_{0}}{\partial a_{0}}=-2q_{1} & \frac{\partial p_{0}}{\partial
b_{0}}=-2b_{0}q_{0} & \frac{\partial p_{0}}{\varepsilon_{0}}=-q_{0}%
\end{array}
.
\]
The Jacobian of $W_{t}$ at $(a_{0},b_{0},\varepsilon_{0}^{t}(\lambda_{0}))$
then has the following form,
\[
DW_{t}=%
\begin{pmatrix}
t(\frac{1}{t}+q_{2}-v_{t}(a_{0})^{2}q_{0}) & 2tb_{0}q_{1} & tq_{1}\\
-2tb_{0}q_{1} & t(\frac{1}{t}+p_{0}-2b_{0}^{2}q_{0}) & -tb_{0}q_{0}\\
2t\sqrt{\varepsilon_{0}}q_{1} & 2t\sqrt{\varepsilon_{0}}b_{0}q_{0} &
\frac{1-tp_{0}}{2\sqrt{\varepsilon_{0}}}+t\sqrt{\varepsilon_{0}}q_{0}%
\end{pmatrix}
.
\]
Since $\varepsilon_{0}=\varepsilon_{0}^{t}(\lambda_{0}),$ we have $1/t=p_{0}$,
and by (\ref{p:in.q}), the $(1,1)$-entry can be simplified to $2tq_{2}$. The
Jacobian matrix $DW_{t}$ then simplifies to
\[
DW_{t}=2t%
\begin{pmatrix}
q_{2} & b_{0}q_{1} & \frac{1}{2}q_{1}\\[3pt]%
-b_{0}q_{1} & (q_{2}+(v_{t}(a_{0})^{2}-b_{0}^{2})q_{0}) & -\frac{1}{2}%
b_{0}q_{0}\\[3pt]%
\sqrt{\varepsilon_{0}}q_{1} & \sqrt{\varepsilon_{0}}b_{0}q_{0} & \frac{1}%
{2}\sqrt{\varepsilon_{0}}q_{0}%
\end{pmatrix}
.
\]

We compute the determinant of $DW_{t}$ by first adding $-2b_{0}$ times the
third column to the second column and then using a cofactor expansion along
the second column. The result is
\[
\det DW_{t}=4t^{3}\sqrt{\varepsilon_{0}}(q_{2}+v_{t}(a_{0})^{2}q_{0}%
)(q_{2}q_{0}-q_{1}^{2})
\]
Now, by the Cauchy--Schwarz inequality,
\[
q_{0}q_{2}-q_{1}^{2}\geq0
\]
and it cannot be an equality unless $\mu$ is a $\delta$-measure. Therefore, we
conclude that $\det DW_{t}$ is positive, establishing the proposition.
\end{proof}

\begin{proof}
[Proof of Proposition \ref{insideRegularity.prop}]The proof is extremely
similar to the proof of Proposition \ref{outsideRegularity.prop}; the desired
extension is given by the map%
\[
(\lambda,z)\mapsto\mathrm{HJ}(W_{t}^{-1}(\lambda,z)),
\]
where $\mathrm{HJ}$ is the Hamilton--Jacobi function in (\ref{HJdef}). Take
$z>0$ and let $(\mu_{0},\delta_{0})=W_{t}^{-1}(\lambda,z).$ Then we must have
$\delta_{0}>\varepsilon_{0}^{t}(\mu_{0})$ or else $z_{t}(\mu_{0},\delta
_{0})=z$ would be negative. Thus, the lifetime of the path will be greater
than $t$ and the Hamilton--Jacobi formula will apply. Thus, the
Hamilton--Jacobi formula (\ref{firstHJ}) shows that $\mathrm{HJ}(W_{t}%
^{-1}(\lambda,z))$ agrees with $\tilde{S}(t,\lambda,z)$ for $z>0.$
\end{proof}

\subsection{Computing the Brown measure\label{computeBrown.sec}}

Using Proposition \ref{insideRegularity.prop}, we can show that the left-hand
sides of (\ref{InsideLim1})\ and (\ref{InsideLim2}) are actually equal to
$\partial s_{t}/\partial a$ and $\partial s_{t}/\partial b.$

\begin{corollary}
\label{cor:stDeri}For any $\lambda=a+ib$ in $\Omega_{t}$ we have
\[
\frac{\partial s_{t}}{\partial a}=\frac{2}{t}(a_{0}^{t}(a)-a),\quad
\frac{\partial s_{t}}{\partial b}=\frac{b}{t},
\]
so that
\[
\Delta s_{t}(\lambda)=\frac{2}{t}\left(  \frac{da_{0}^{t}(a)}{da}-\frac{1}%
{2}\right)  .
\]

\end{corollary}

\begin{proof}
Fix $t>0$ and $\lambda^{\ast}\in\Omega_{t}.$ By Proposition
\ref{insideRegularity.prop}, the function%
\[
s_{t}(\lambda):=\lim_{\varepsilon\rightarrow0^{+}}S(t,\lambda,\varepsilon)
\]
may be computed, for $\lambda$ in a neighborhood of $\lambda^{\ast},$ as%
\[
s_{t}(\lambda)=\tilde{S}(t,\lambda,0).
\]
Since $\tilde{S}$ is smooth, we can evaluate $\tilde{S}$ or any of its
derivatives at $(t,\lambda,0)$ by taking limits along any path we choose with
$t$ fixed. Thus, the left-hand sides of (\ref{InsideLim1}) and
(\ref{InsideLim2}) are actually equal to $\partial s_{t}/\partial a$ and
$\partial s_{t}/\partial b.$ The formula for $\Delta s_{t}$ then follows by
taking second derivatives with respect to $a$ and to $b$ and simplifying.
\end{proof}

We now establish our main result, a formula for the Brown measure of
$x_{0}+i\sigma_{t}.$

\begin{theorem}
\label{thm:main} The open set $\Omega_{t}$ is a set of full measure for the
Brown measure of $x_{0}+i\sigma_{t}.$ Inside $\Omega_{t},$ the Brown measure
is absolutely continuous with a strictly positive density $w_{t}$ given by
\begin{equation}
w_{t}(\lambda)=\frac{1}{2\pi t}\left(  \frac{da_{0}^{t}(a)}{da}-\frac{1}%
{2}\right)  ,\quad\lambda=a+ib. \label{eq:density}%
\end{equation}
Since $w_{t}(\lambda)$ is independent of $b,$ we see that $w_{t}$ is constant
along the vertical segments inside $\Omega_{t}.$
\end{theorem}

\begin{proof}
Corollary \ref{cor:stDeri} shows that in $\Omega_{t},$ the Brown measure of
$x_{0}+i\sigma_{t}$ has a density given by (\ref{eq:density}). It then follows
from Point \ref{datDt.point} of Proposition \ref{Jprops.prop} that
\[
\frac{da_{0}^{t}(a)}{da}>\frac{1}{2},
\]
showing that $w_{t}$ is positive in $\Omega_{t}$. It remains to show that
$\Omega_{t}$ is a set of full Brown measure. Since the Brown measure is zero
outside $\overline{\Omega}_{t},$ we see that $\Omega_{t}$ will have full
measure provided that the boundary of $\Omega_{t}$ has measure zero. While it
may be possible to prove this directly using the strategy in Section 7.4 of
\cite{DHKBrown}, we instead use the approach used in \cite{HZ}.

In Theorem \ref{push.thm}, we will consider a probability measure $\rho_{t}$
on $\Lambda_{t}.$ We will then show that the push-forward of $\rho_{t}$ under
the map $U_{t}:\Lambda_{t}\rightarrow\Omega_{t}$ agrees with $\mathrm{Brown}%
(x_{0}+i\sigma_{t})$ on $\Omega_{t}.$ Since the preimage of $\Omega_{t}$ under
$U_{t}$ is $\Lambda_{t}$ and $\rho_{t}(\Lambda_{t})=1,$ we see that the
$\mathrm{Brown}(x_{0}+i\sigma_{t})$ assigns full measure to $\Omega_{t}.$
\end{proof}

\section{Two results about push-forwards of the Brown
measure\label{pushforward.sec}}

In this section, we show how $\mathrm{Brown}(x_{0}+i\sigma_{t})$ is related to
two other measure by means of pushing forward under appropriate maps. To
motivate one of our results, let us consider the case that $x_{0}%
=\tilde{\sigma}_{s}$, a semicircular element of variance $s$ freely
independent of $\sigma_{t}.$ It is known (see \cite[Example 5.3]{BL} and
Section \ref{sect:ellipticLaw}) that in this case, the Brown measure of
$\tilde{\sigma}_{s}+i\sigma_{t}$ is uniformly distributed on an ellipse. It
follows that the distribution of $\operatorname{Re}\lambda$ with respect to
$\mathrm{Brown}(\tilde{\sigma}_{s}+i\sigma_{t})$ is semicircular---which is
the same (up to scaling by a constant) as the distribution of $\tilde{\sigma
}_{s}+\sigma_{t}.$ (See Figure \ref{ellipse.fig}.)

Point 2 of Theorem \ref{push.thm} generalizes the preceding result to the case
of arbitrary $x_{0},$ in which the map $\lambda\mapsto const.\operatorname{Re}%
\lambda$ is replaced by a certain map $Q_{t}:\overline{\Omega}_{t}%
\rightarrow\mathbb{R}.$ When the distribution of $x_{0}$ is semicircular,
$Q_{t}(\lambda)$ is a multiple of the real part of $\lambda,$ as in
(\ref{QtSemicircular}).%

\begin{figure}[ptb]%
\centering
\includegraphics[
height=2.2364in,
width=4.8282in
]%
{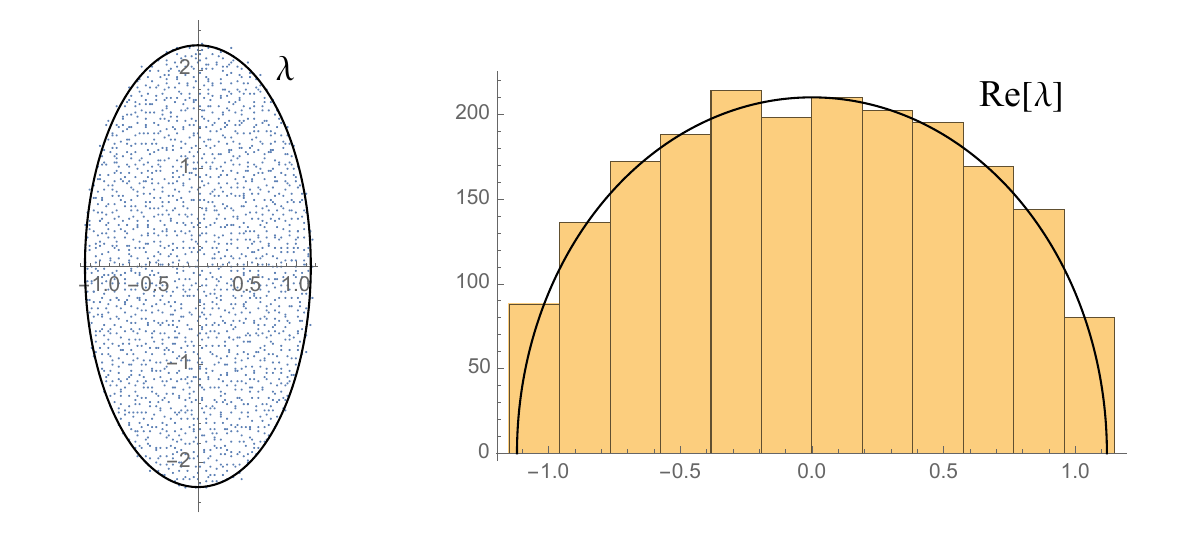}%
\caption{A random matrix approximation to the Brown measure of $x_{0}%
+i\sigma_{t}$ when $x_{0}$ is semicircular (left) and the distribution of the
real parts of the eigenvalues (right).}%
\label{ellipse.fig}%
\end{figure}

Recall that $c_{t}$ denotes the circular Brownian motion. We will make use of
the map $U_{t}:\overline{\Lambda}_{t}\rightarrow\overline{\Omega}_{t}$
described in Section \ref{sect:surjectivity}, and another map $Q_{t}%
:\overline{\Omega}_{t}\rightarrow\mathbb{R}$ which we now define. Recall from
Sections \ref{bianeResult.sec} and \ref{lambdaDomain.sec} that the inverse of
the map $J_{t}$ takes $\partial\Omega_{t}$ to $\partial\Lambda_{t}$ and that
the map $H_{t}$ takes $\partial\Lambda_{t}$ to $\mathbb{R},$ so that
$H_{t}\circ J_{t}^{-1}$ takes $\partial\Omega_{t}$ to $\mathbb{R}.$

\begin{definition}
Let $Q_{t}:\overline{\Omega}_{t}\rightarrow\mathbb{R}$ be the unique map that
agrees with $H_{t}\circ J_{t}^{-1}$ on $\partial\Omega_{t}$ and maps vertical
segments in $\overline{\Omega}_{t}$ to points in $\mathbb{R}.$
\end{definition}

The map $Q_{t}$ is visualized in Figure \ref{qmap.fig}. In the case that
$x_{0}$ is semicircular with variance $s,$ one can easily use the computations
in Section \ref{sect:ellipticLaw} to show that
\begin{equation}
Q_{t}(a+ib)=\frac{s+t}{s}a. \label{QtSemicircular}%
\end{equation}

In general, we may compute $Q_{t}$ more explicitly as follows. We first map
$a+ib$ to the point $a+ib_{t}(a)$ on $\partial\Omega_{t}.$ Next, we compute
\[
J_{t}^{-1}(a+ib_{t}(a))=a_{0}^{t}(a)+iv_{t}(a_{0}^{t}(a)).
\]
Next, we use the identity $H(J_{t}^{-1}(z))=2J_{t}^{-1}(z)-z$ in
(\ref{twoMapsId}). Finally, we recall that $H_{t}\circ J_{t}^{-1}$ is
real-valued on $\partial\Omega_{t}.$ Thus,%
\begin{align*}
Q_{t}(a+ib)  &  =\operatorname{Re}\{2J_{t}^{-1}[a+ib_{t}(a)]-(a+ib_{t}(a))\}\\
&  =2a_{0}^{t}(a)-a.
\end{align*}

\begin{theorem}
\label{push.thm}The following results hold.

\begin{enumerate}
\item The push-forward of the Brown measure of $x_{0}+c_{t}$ under the map
$U_{t}$ in Theorem \ref{thm:lambda0tolambda} is the Brown measure of
$x_{0}+i\sigma_{t}$.

\item The push-forward of the Brown measure of $x_{0}+i\sigma_{t}$ under the
map $Q_{t}$ is the law of $x_{0}+\sigma_{t}$.
\end{enumerate}
\end{theorem}

\begin{proof}
By Theorem 3.9 in \cite{HZ}, the Brown measure $\rho_{t}$ of $x_{0}+c_{t}$ can
be written as
\begin{align*}
d\rho_{t}  &  =\frac{1}{\pi t}\left(  1-\frac{t}{2}\frac{d}{da_{0}}%
\int_{\mathbb{R}}\frac{x\,d\mu(x)}{(a_{0}-x)^{2}+v_{t}(a_{0})^{2}}\right)
~da_{0}\,db_{0}\\
&  =\frac{1}{\pi t}\left(  1-\frac{1}{2}\frac{da_{t}}{da_{0}}\right)
~da_{0}\,db_{0}%
\end{align*}
for $a_{0}+ib_{0}\in\Lambda_{t}$. Now, under the map $U_{t},$ we have
$a=a_{t}(a_{0})$ and $b=2b_{0}.$ Thus,
\begin{align*}
d\rho_{t}  &  =\frac{1}{\pi t}\left(  1-\frac{1}{2}\frac{da_{t}}{da_{0}%
}\right)  ~\frac{da_{0}^{t}}{da}da\,\frac{db}{2}\\
&  =\frac{1}{2\pi t}\left(  \frac{da_{0}^{t}}{da}-\frac{1}{2}\right)  ~da~db
\end{align*}
for $a+ib\in\Omega_{t}$. This last expression in the formula for the
restriction of the Brown measure to $\Omega_{t}.$

Since $\rho_{t}$ is a probability measure on $\Lambda_{t},$ we find that the
Brown measure of $x_{0}+i\sigma_{t}$ assigns mass 1 to $\Omega_{t}$, as noted
in the proof of Theorem \ref{thm:main}. Thus, there is no mass of
$\mathrm{Brown}(x_{0}+i\sigma_{t})$ anywhere else and the pushforward of
$\rho_{t}$ under $U_{t}$ is precisely $\mathrm{Brown}(x_{0}+i\sigma_{t}).$

To prove Point 2, we consider the unique map $\Psi_{t}:\overline{\Lambda}%
_{t}\rightarrow\mathbb{R}$ that agrees with $H_{t}$ on $\partial\Lambda_{t}$
and is constant along vertical segments in $\Lambda_{t}.$ Then $Q_{t}=\Psi
_{t}\circ U_{t}^{-1}$. (Both $Q_{t}$ and $H_{t}\circ U_{t}^{-1}$ agree with
$\Psi_{t}\circ J_{t}^{-1}$ on $\partial\Omega_{t}$ and are constant along
vertical segments inside $\overline{\Omega}_{t}.$) By Point 1, the
push-forward of the Brown measure of $x_{0}+i\sigma_{t}$ under $U_{t}^{-1}$ is
the Brown measure $\rho_{t}$ of $x_{0}+c_{t}$. By Theorem 3.13 of \cite{HZ},
the push-forward of the Brown measure $\rho_{t}$ by $\Psi_{t}$ of $x_{0}%
+c_{t}$ is the law of $x_{0}+\sigma_{t}$ and Point 2 follows.
\end{proof}

\section{The method of Jarosz and Nowak\label{jn.sec}}

\subsection{The formula for the Brown measure}

We now describe a different approach to computing the Brown measure of
$x_{0}+i\sigma_{t},$ developed by Jarosz and Nowak in the physics literature
\cite{JN1,JN2}. As discussed in the introduction, the method is not rigorous
as written, but could conceivably be made rigorous using the general framework
developed by Belinschi, Mai, and Speicher in \cite{BMS}. (See also related
results in \cite{BSS}.) We emphasize, however, that (so far as we know) no
explicit computation of the case of $x_{0}+i\sigma_{t}$ has been made using
the framework in \cite{BMS}.

Jarosz and Nowak work with an operator of the form $H_{1}+iH_{2},$ where
$H_{1}$ and $H_{2}$ are arbitrary freely independent elements. Then on p.
10118 of \cite{JN2}, they present an algorithm by which the \textquotedblleft
nonholomorphic Green's function\textquotedblright\ of $H_{1}+iH_{2}$ may be
computed. In the notation of this paper, the nonholomorphic Green's function
is the function $\partial s_{t}/\partial\lambda,$ so that the Brown measure
may be computed by taking the $\bar{\lambda}$-derivative:%
\[
\frac{1}{\pi}\frac{\partial}{\partial\bar{\lambda}}\frac{\partial s_{t}%
}{\partial\lambda}=\frac{1}{4\pi}\Delta_{\lambda}s_{t}(\lambda).
\]

Examples are presented in Section 6.1 of \cite{JN1} in which $H_{2}$ is
semicircular and $H_{2}$ has various different distributions. We now work out
their algorithm in detail in the case that $H_{1}=x_{0}$ is an
\textit{arbitrary} self-adjoint element and $H_{2}=\sigma_{t}.$

We refer to \cite{JN1,JN2} for the framework used in the algorithm, involving
\textquotedblleft quaternionic Green's functions.\textquotedblright\ We
present only the final algorithm for computation, described in Eqs. (75)--(79)
of \cite{JN2}, and we specialize to the case $H_{1}=x_{0}$ and $H_{2}%
=\sigma_{t}.$ The algorithm, adapted to our notation, is as follows. We fix a
complex number $\lambda=a+ib.$ Then we introduce three unknown quantities,
complex numbers $g$ and $g^{\prime}$ and a real number $m.$ These are supposed
to satisfy three equations:%
\begin{align}
B_{x_{0}}(g)  &  =a+\frac{m}{g}\label{algorithm1}\\
B_{\sigma_{t}}(g^{\prime})  &  =b+\frac{1-m}{g^{\prime}}\label{algorithm2}\\
\left\vert g\right\vert  &  =\left\vert g^{\prime}\right\vert ,
\label{algorithm3}%
\end{align}
where $B_{x_{0}}$ and $B_{\sigma_{t}}$ are the \textquotedblleft Blue's
functions,\textquotedblright\ that is, the inverse functions of the Cauchy
transforms of $x_{0}$ and $\sigma_{t},$ respectively. We are supposed to solve
these equations for $g$ and $g^{\prime}$ as functions of $a$ and $b$. Once
this is done, we have%
\begin{equation}
\frac{\partial s_{t}}{\partial\lambda}(a,b)=\operatorname{Re}%
g-i\operatorname{Re}g^{\prime}. \label{algorithm4}%
\end{equation}
Since $\partial/\partial\lambda=(\partial/\partial a-i~\partial/\partial
b)/2,$ (\ref{algorithm4}) may be written equivalently as%
\begin{align}
\frac{\partial s_{t}}{\partial a}(a,b)  &  =2\operatorname{Re}g;\label{dSda}\\
\frac{\partial s_{t}}{\partial b}  &  =2\operatorname{Re}g^{\prime}.
\label{dSdb}%
\end{align}
That is to say, the real parts of $g$ and $g^{\prime}$ determine the
derivatives of $s_{t}$ with respect to $a$ and $b,$ respectively.

\begin{proposition}
\label{JNmethod1.prop}The Jarosz--Nowak method when applied to $x_{0}%
+i\sigma_{t}$ gives the following result. We try to solve the equation%
\[
g=G_{x_{0}}(a+t\bar{g})
\]
for $g$ as a function of $a$ and $t,$ with the solution denoted $g_{t}(a).$
Then, inside the support of the Brown measure, its density $\rho_{t}$ is a
function of $a$ and $t$ only, namely%
\begin{equation}
\rho_{t}(a)=\frac{1}{4\pi}\left(  \frac{1}{t}+2\frac{d}{da}\operatorname{Re}%
g_{t}(a)\right)  . \label{rhoFormula}%
\end{equation}

\end{proposition}

\begin{proof}
It is known that the Blue's function of $\sigma_{t}$ is given by
$B_{\sigma_{t}}(g^{\prime})=1/g^{\prime}+tg^{\prime}.$ (This statement is
equivalent to saying that the $R$-transform of $\sigma_{t}$ is given by
$R(z)=tz,$ as in \cite[Example 3.4.4]{VDN}.) Plugging this expression into
(\ref{algorithm2}) and simplifying, we obtain a quadratic equation:%
\[
t(g^{\prime})^{2}-bg^{\prime}+m=0,
\]
whose roots are%
\begin{equation}
g^{\prime}=\frac{b\pm\sqrt{b^{2}-4mt}}{2t}. \label{roots}%
\end{equation}
Assuming (as Jarosz and Nowak implicitly do) that these roots are complex, we
find that
\begin{equation}
\operatorname{Re}(g^{\prime})=\frac{b}{2t}. \label{ReRoot}%
\end{equation}
Thus, without even using (\ref{algorithm1}) or (\ref{algorithm3}), we find by
(\ref{dSdb}) that%
\[
\frac{\partial s_{t}}{\partial b}=\frac{b}{t},
\]
which agrees with what we found in Corollary \ref{cor:stDeri}.

Meanwhile, assuming still that the roots in (\ref{roots}) are complex, we find
that
\begin{align*}
\left\vert g^{\prime}\right\vert ^{2}  &  =\frac{1}{4t^{2}}(b^{2}+4mt-b^{2})\\
&  =\frac{m}{t}.
\end{align*}
Then (\ref{algorithm3}) says that $\left\vert g\right\vert ^{2}=\left\vert
g^{\prime}\right\vert ^{2}=m/t,$ so that $m=t\left\vert g\right\vert ^{2}.$
Thus, after replacing $m$ by $t\left\vert g\right\vert ^{2},$
(\ref{algorithm1}) becomes%
\[
B_{x_{0}}(g)=a+t\bar{g}.
\]
Since $B_{x_{0}}$ is the inverse function to $G_{x_{0}},$ this equation may be
rewritten as%
\begin{equation}
g=G_{x_{0}}(a+t\bar{g}). \label{EqnForG}%
\end{equation}
We hope that this equation will implicitly determine the complex number $g$ as
a function of $a$ (and $t$). We therefore write $g$ as $g_{t}(a).$

We then substitute the expression for $g$ into (\ref{dSda}), giving%
\[
\frac{\partial s_{t}}{\partial a}=2\operatorname{Re}g_{t}(a).
\]
The density $\rho_{t}$ of the Brown measure is then computed as
\begin{align*}
\rho_{t}(a,b)  &  =\frac{1}{4\pi}\left(  \frac{\partial^{2}s_{t}}{\partial
b^{2}}+\frac{\partial^{2}s_{t}}{\partial a^{2}}\right) \\
&  =\frac{1}{4\pi}\left(  \frac{1}{t}+2\frac{d}{da}\operatorname{Re}%
g_{t}(a)\right)  ,
\end{align*}
as claimed.
\end{proof}

\begin{proposition}
\label{JNmethod2.prop}In the Jarosz--Nowak method, the quantity
$\operatorname{Re}g_{t}(a)$ may be computed as%
\[
\operatorname{Re}g_{t}(a)=\frac{1}{t}(a_{0}^{t}(a)-a),
\]
where the function $a_{0}^{t}$ is as in Summary \ref{surject.summary}. Thus,
the formula (\ref{rhoFormula}) for the Brown measure in the Jarosz--Nowak
method agrees with what we found in Theorem \ref{thm:main}.
\end{proposition}

\begin{proof}
The imaginary part of the equation (\ref{EqnForG}) for $g$ says that%
\begin{align*}
\operatorname{Im}g  &  =\operatorname{Im}\int_{\mathbb{R}}\frac{d\mu
(x)}{a+t\bar{g}-x}\\
&  =t\operatorname{Im}g\int_{\mathbb{R}}\frac{d\mu(x)}{(a+t\operatorname{Re}%
g-x)^{2}+t^{2}(\operatorname{Im}g)^{2}}.
\end{align*}
Thus, at least when $\operatorname{Im}g\neq0,$ we get%
\begin{equation}
\int\frac{d\mu(x)}{(a+t\operatorname{Re}g-x)^{2}+t^{2}(\operatorname{Im}%
g)^{2}}=\frac{1}{t}. \label{Pequals}%
\end{equation}
We may now apply the equation (\ref{vtDef}) that defines the function $v_{t}$
with $a$ replaced by $a+t\operatorname{Re}g,$ giving
\begin{equation}
t\operatorname{Im}g=\pm v_{t}(a+t\operatorname{Re}g). \label{tImg}%
\end{equation}

We now look at the real part of (\ref{EqnForG}):%
\begin{align*}
\operatorname{Re}g  &  =\operatorname{Re}\int_{\mathbb{R}}\frac{d\mu
(x)}{a+t\bar{g}-x}\\
&  =(a+t\operatorname{Re}g)\int_{\mathbb{R}}\frac{1}{(a+t\operatorname{Re}%
g-x)^{2}+t^{2}(\operatorname{Im}g)^{2}}~d\mu(x)\\
&  -\int_{\mathbb{R}}\frac{x}{(a+t\operatorname{Re}g-x)^{2}+t^{2}%
(\operatorname{Im}g)^{2}}~d\mu(x).
\end{align*}
Using (\ref{Pequals}) and (\ref{tImg}) this equation simplifies to
\begin{equation}
a=t\int_{\mathbb{R}}\frac{x}{(a+t\operatorname{Re}g-x)^{2}+v_{t}%
(a+t\operatorname{Re}g)^{2}}~d\mu(x). \label{aEq}%
\end{equation}

Now, if we let%
\begin{equation}
a_{0}=a+t\operatorname{Re}g, \label{JNa0}%
\end{equation}
then (\ref{aEq}) is just the equation for $a$ in terms of $a_{0}$ that we
found in our Hamilton--Jacobi analysis (Proposition \ref{solveODEspecial.prop}%
). Thus,%
\[
\operatorname{Re}g=\frac{1}{t}(a_{0}-a),
\]
as claimed.
\end{proof}

\subsection{The support of the Brown measure}

We now examine the condition for the boundary of the support of the Brown
measure, as given in Eq. (80) of \cite{JN2}:%
\begin{equation}
(\operatorname{Re}g)^{2}+(\operatorname{Re}g^{\prime})^{2}=\left\vert
g\right\vert ^{2}. \label{JNbdy}%
\end{equation}

\begin{proposition}
\label{JNmethod3.prop}The condition for a point $a+ib$ to be on the boundary
in the Jarosz--Nowak method is that
\[
b=2v_{t}(a_{0}^{t}(a)).
\]
Such points are precisely the boundary points of our domain $\Omega_{t}.$
\end{proposition}

\begin{proof}
We cancel $(\operatorname{Re}g)^{2}$ from both sides of (\ref{JNbdy}), leaving
us with%
\[
(\operatorname{Re}g^{\prime})^{2}=(\operatorname{Im}g)^{2}.
\]
Now, we have found in (\ref{ReRoot}) that $\operatorname{Re}g^{\prime}=b/(2t)$
and in (\ref{tImg}) that $\operatorname{Im}g=\pm v_{t}(a+t\operatorname{Re}%
g)/t.$ But in (\ref{JNa0}), we have identified $a+t\operatorname{Re}g$ with
$a_{0}^{t}(a).$ Thus, the condition (\ref{JNbdy}) for the boundary reads%
\[
\frac{b}{2t}=\pm\frac{v_{t}(a_{0}^{t}(a))}{t},
\]
or $b=2v_{t}(a_{0}^{t}(a)),$ which is the condition for the boundary of
$\Omega_{t}$ (Point \ref{graphs.point} of Proposition \ref{Jprops.prop}).
\end{proof}

\section{Examples\label{examples.sec}}

In this section, we compute three examples, in which the law of $x_{0}$ is
semicircular, Bernoulli, or uniform. Additional examples, computed by a
different method, were previously worked out by Jarosz and Nowak in
\cite[Section 6.1]{JN1}.

We also mention that we can take $x_{0}$ to have the form $x_{0}=y_{0}%
+\tilde{\sigma}_{s},$ where $\tilde{\sigma}_{s}$ is another semicircular
Brownian motion and $y_{0},$ $\tilde{\sigma}_{s},$ and $\sigma_{t}$ are all
freely independent. Thus, our results allow one to determine the Brown measure
for the sum of the elliptic element $\tilde{\sigma}_{s}+i\sigma_{t}$ and the
freely independent self-adjoint element $y_{0}$. The details of this analysis
will appear elsewhere.

\subsection{The elliptic law\label{sect:ellipticLaw}}

In our first example, the law $\mu$ of $x_{0}$ is a semicircular distribution
with variance $s.$ Then $x_{0}+i\sigma_{t}$ has the form of an
\textbf{elliptic element} $\tilde{\sigma}_{s}+i\sigma_{t}$, where
$\tilde{\sigma}_{\cdot}$ and $\sigma_{\cdot}$ are two freely independent
semicircular Brownian motions. The associated \textquotedblleft elliptical
law,\textquotedblright\ in various forms, has been studied extensively going
back to the work of Girko \cite{GirkoElliptic}. The elliptical case was worked
out by Jarosz and Nowak in \cite[Section 3.6]{JN2}. The Brown measure of
$\tilde{\sigma}_{s}+i\sigma_{t}$ was also computed by Biane and Lehner
\cite[Example 5.3]{BL} by a different method. We include this example as a
simple demonstration of the effectiveness of our method.

\begin{theorem}
\label{thm:EllipticLaw}The Brown measure of $\tilde{\sigma}_{s}+i\sigma_{t}$
is supported in the closure of the ellipse centered at the origin with
semi-axes $2s/\sqrt{s+t}$ and $2t/\sqrt{s+t}$. The density of the Brown
measure is constant
\[
\frac{1}{4\pi}\left(  \frac{1}{s}+\frac{1}{t}\right)
\]
in the domain.
\end{theorem}

We apply our result in this paper with $x_{0}=\tilde{\sigma}_{s}$. In the next
proposition we compute $\Lambda_{t}$.

\begin{proposition}
\label{prop:HtElliptic} We can parametrize the upper boundary curve of
$\Lambda_{t}$ by
\begin{equation}
a_{0}+iv_{t}(a_{0})=\frac{(2s+t)q+it\sqrt{4(s+t)-q^{2}}}{2(s+t)},\quad
q\in\lbrack-2\sqrt{s+t},2\sqrt{s+t}]; \label{ellipseUpper}%
\end{equation}
therefore, $\Lambda_{t}$ is the ellipse centered at the origin with semi-axes
$\frac{2s+t}{\sqrt{s+t}}$ and $\frac{t}{\sqrt{s+t}}$.
\end{proposition}

\begin{proof}
Recall that $\Delta_{t}$ denotes the region in the upper half plane above the
graph of $v_{t}$, so that $\overline{\Delta}_{t}$ is the region on or above
the graph of $v_{t}.$ Recall also that Biane \cite{BianeConvolution} has shown
that the function $H_{t}$ in (\ref{htDef}) maps $\overline{\Delta}_{t}$
injectively onto the closed upper half plane.

In the case at hand, the Cauchy transform of $\tilde{\sigma}_{s}$ is
$G_{\tilde{\sigma}_{s}}(z)=\left(  z-\sqrt{z^{2}-4s}\right)  /(2s)$. We can
then compute the function $H_{t}$ in (\ref{htDef}) as
\[
H_{t}(z)=z+tG_{\sigma_{s}}(z)=z+t\left(  \frac{z-\sqrt{z^{2}-4s}}{2s}\right)
,\quad z\in\overline{\Delta}_{t}.
\]
The inverse map $H_{t}^{-1}$ is then easily computed as%
\begin{equation}
H_{t}^{-1}(z)=\frac{(2s+t)z+t\sqrt{z^{2}-4(s+t)}}{2(s+t)},\quad
\operatorname{Im}z\geq0. \label{HtInverse}%
\end{equation}

The part of the graph of $v_{t}$ where $v_{t}>0$ comes from the values of
$H_{t}^{-1}$ on the real axis having nonzero imaginary part, that is, for real
numbers $q$ with $\left\vert q\right\vert <2\sqrt{s+t}.$ Plugging these
numbers into (\ref{HtInverse}) gives the claimed form (\ref{ellipseUpper}).
\end{proof}

\begin{proposition}
\label{prop:OmegatEllipse} The boundary curve of $\Omega_{t}$ can be
parametrized by
\[
a+ib_{t}(a)=\frac{sq+it\sqrt{4(s+t)-q^{2}}}{s+t},\quad q\in\lbrack-2\sqrt
{s+t},2\sqrt{s+t}].
\]
Consequently, $\Omega_{t}$ is an ellipse centered at the origin with semi-axes
$\frac{2s}{\sqrt{s+t}}$ and $\frac{2t}{\sqrt{s+t}}$.
\end{proposition}

\begin{proof}
By Proposition \ref{prop:HtElliptic}, the upper boundary of $\Lambda_{t}$ can
be parametrized by the curve in (\ref{ellipseUpper}). By Definition
\ref{omegaT.def}, we find the boundary curve of $\Omega_{t}$ by applying the
map $J_{t}$ in (\ref{jtDef}), which satisfies $J_{t}(z)=2z-H_{t}(z)$. Thus,
\begin{align*}
a+ib_{t}(a)  &  =J_{t}(a_{0}+iv_{t}(a_{0}))\\
&  =\frac{sq+it\sqrt{4(s+t)-q^{2}}}{s+t}.
\end{align*}
which traces an ellipse centered at the origin with semi-axes $\frac{2s}{s+t}$
and $\frac{2t}{\sqrt{s+t}}$.
\end{proof}

\begin{proof}
[Proof of Theorem \ref{thm:EllipticLaw}]The domain $\Omega_{t}$ is computed in
Proposition \ref{prop:OmegatEllipse}. By Proposition \ref{prop:OmegatEllipse}%
,
\begin{align*}
a  &  =\frac{sq}{s+t}\\
a_{0}^{t}(a)  &  =\frac{(2s+t)q}{2(s+t)}=\frac{(2s+t)}{2s}a.
\end{align*}
It follows that the density of the Brown measure is
\begin{align*}
\frac{1}{2\pi t}\left(  \frac{da_{0}^{t}(a)}{da}-\frac{1}{2}\right)   &
=\frac{1}{2\pi t}\left(  \frac{2s+t}{2s}-\frac{1}{2}\right) \\
&  =\frac{1}{4\pi}\left(  \frac{1}{s}+\frac{1}{t}\right)  ,
\end{align*}
as claimed.
\end{proof}

\subsection{Bernoulli case}

In our second example, the law $\mu$ of $x_{0}$ is Bernoulli distributed, with
mass $\alpha$ at $1$ and mass $\beta=1-\alpha$ at $-1$, for $0<\alpha<1.$ The
case $\alpha=1/2$ was previously analyzed in the paper of Stephanov
\cite{stephanov} and also in Section V of \cite{JN1} by different methods.

Denote by $Q(a)$ the quartic polynomial
\begin{equation}
-4a^{4}+4t(\alpha-\beta)a^{3}-(t^{2}+4t-8)a^{2}+2t(t-2)(\alpha-\beta
)a-(\alpha-\beta)^{2}t^{2}+4t-4. \label{eq:Qdef}%
\end{equation}
Then the domain $\Omega_{t}$ and the density of the Brown measure in this
example are computed in the following proposition.

\begin{proposition}
\label{prop:BernoulliResult} Any $\lambda\in\Omega_{t}$ satisfies $\left\vert
\operatorname{Re}\lambda\right\vert <1$. The domain $\Omega_{t}$ is given by
\[
\Omega_{t}=\left\{  a+ib\in\mathbb{C}\left\vert b^{2}<\frac{Q(a)}%
{(1-a^{2})^{2}}\right.  \right\}
\]
so that
\[
\Omega_{t}\cap\mathbb{R}=\{a\in\mathbb{R}|Q(a)>0\}.
\]
The density of the Brown measure in this Bernoulli case is given by
\[
w_{t}(\lambda)=\frac{1}{4\pi}\left(  -\frac{1}{t}+\frac{\beta}{(a-1)^{2}%
}+\frac{\alpha}{(a+1)^{2}}\right)  .
\]

\end{proposition}

See Figure \ref{bernoullidensity100.fig}.%

\begin{figure}[ptb]%
\centering
\includegraphics[
height=2.7691in,
width=3.5276in
]%
{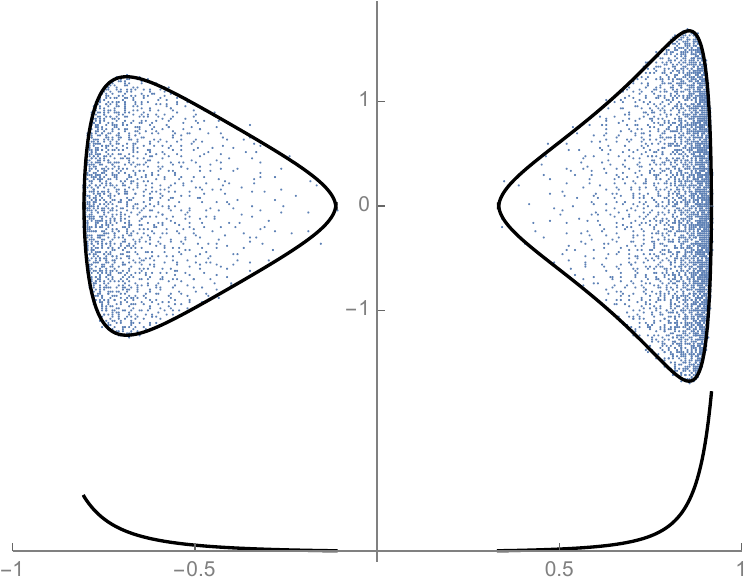}%
\caption{The domain $\Omega_{t}$ in the Bernoulli case with a simulation of
the eigenvalues (top), plotted with the density of the Brown measure (in
$\Omega_{t}$) as a function of $a$ (bottom). Shown for $\alpha=2/3$ and
$t=1.$}%
\label{bernoullidensity100.fig}%
\end{figure}

\begin{proof}
Recall the functions $a_{0}^{t}(a)$ and $b_{t}(a)$ defined by the four
equations (\ref{p0EqIntro})--(\ref{btDef}). We now compute these functions for
the Bernoulli case. The equations (\ref{p0EqIntro}) and (\ref{p1EqIntro}) take
the following form in the Bernoulli case:
\begin{align*}
\frac{\alpha}{(a_{0}-1)^{2}+v^{2}}+\frac{\beta}{(a_{0}+1)^{2}+v^{2}}  &
=\frac{1}{t}\\
\frac{\alpha}{(a_{0}-1)^{2}+v^{2}}-\frac{\beta}{(a_{0}+1)^{2}+v^{2}}  &
=\frac{a}{t}.
\end{align*}
We then introduce the new variables
\begin{align*}
A=  &  (a_{0}-1)^{2}+v^{2}\\
B=  &  (a_{0}+1)^{2}+v^{2}%
\end{align*}
so that
\begin{align*}
\frac{\alpha}{A}+\frac{\beta}{B}  &  =\frac{1}{t}\\
\frac{\alpha}{A}-\frac{\beta}{B}  &  =\frac{a}{t}.
\end{align*}
Then we can solve for $A$ and $B$ as
\[
A=\frac{2\alpha t}{1+a},\quad B=\frac{2\beta t}{1-a}.
\]

We can recover $a_{0}$ and $v^{2}$ as
\[
a_{0}=\frac{1}{4}(B-A)=\frac{t}{2}\frac{a+\beta-\alpha}{1-a^{2}}%
\]
and
\begin{align*}
v^{2}  &  =\frac{A+B}{2}-a_{0}^{2}-1\\
&  =\frac{\alpha t(1-a)+\beta t(1+a)}{1-a^{2}}-\left(  \frac{t}{2}%
\frac{a+\beta-\alpha}{1-a^{2}}\right)  ^{2}-1\\
&  =\frac{Q(a)}{4(1-a^{2})^{2}}%
\end{align*}
where $Q$ is defined in (\ref{eq:Qdef}). Recalling that $b_{t}(a)=2v$, we find
that
\begin{equation}
a_{0}^{t}(a)=\frac{t}{2}\frac{a+\beta-\alpha}{1-a^{2}} \label{eq:Beroullia0}%
\end{equation}
and
\begin{equation}
b_{t}(a)^{2}=\frac{Q(a)}{(1-a^{2})^{2}}. \label{eq:Bernoullivt}%
\end{equation}

Equation (\ref{eq:Bernoullivt}) gives the claimed form of the domain
$\Omega_{t}.$ The density of the Brown measure is then computed from
(\ref{eq:Beroullia0}) using the formula in Theorem \ref{thm:main}.
\end{proof}

\subsection{Uniform case}

In our third and final example, $\mu$ is uniformly distributed on $[-1,1]$.
The case in which $\mu$ is uniformly distributed on any interval can be
reduced to this case as follows. First, by shifting $x_{0}$ by a constant, we
can assume that the interval has the form $[-A,A].$ Once this is the case, we
write%
\[
x_{0}+i\sigma_{t}=A\left(  x_{0}/A+i\sigma_{t}/A\right)  ,
\]
where the law of $x_{0}/A$ is uniform on $[-1,1]$ and $\sigma_{t}/A$ has the
same $\ast$-distribution as $\sigma_{t/A^{2}}.$ Thus, to compute the Brown
measure in this case, we use the formulas below with $t$ replaced by $t/A^{2}$
and then scale the entire Brown measure by a factor of $A.$ 

When $\mu$ is uniform on $[-1,1]$---in particular, symmetric about $0$---the
Brown measure is symmetric about the imaginary axis. The key is again solving
the equations (\ref{p0EqIntro}) and\ (\ref{p1EqIntro}) that define the
functions $a_{0}^{t}(a)$ and $b_{t}(a).$

\begin{proposition}
\label{prop:UniformOmega} Let $v_{\max}$ be the smallest positive real number
$v$ such that
\[
\frac{1}{v}=\tan\left(  \frac{v}{t}\right)  .
\]
Also let $A_{0}^{t}(v)$ be given by
\[
A_{0}^{t}(v)=\sqrt{2v\cot\left(  \frac{2v}{t}\right)  +1-v^{2}}.
\]
for $\left\vert v\right\vert \leq v_{\max}$. (When $v=0$, we understand the
above formula as $A_{0}^{t}(0)=\sqrt{t+1}$.) Then the following results hold.

\begin{enumerate}
\item The domain $\Omega_{t}$ has only one connected component, and is given
by
\begin{equation}
\Omega_{t}=\{\pm A_{t}(b/2)+iy|~\left\vert y\right\vert <b,~\left\vert
b\right\vert \leq2v_{\max}\} \label{eq:OmegatUniformSet}%
\end{equation}
where
\[
A_{t}(v)=A_{0}^{t}(v)+\frac{t}{4}\log\left(  1-\frac{4A_{0}^{t}(v)}{(A_{0}%
^{t}(v)+1)^{2}+v^{2}}\right)  .
\]
That is, $\partial\Omega_{t}$ consists of the two curves
\[
\pm A_{t}(b/2)+ib,\quad\left\vert b\right\vert \leq2v_{\max}.
\]
In particular, the function $b_{t}$ defined in (\ref{eq:btDef}) is unimodal
with a peak at $0$.

\item The domain $\Omega_{t}$ satisfies
\[
\Omega_{t}\cap\mathbb{R}=\left\{  a\in\mathbb{R}\left\vert \left\vert
a\right\vert <\sqrt{t+1}-\frac{t}{2}\log\left(  \frac{\sqrt{t+1}+1}{\sqrt
{t+1}-1}\right)  \right.  \right\}  .
\]

\item The density of the Brown measure in $\Omega_{t}$ is (as always) a
function of $a$ and $t$ only, and the graph of this function is traced out by
the curve%
\begin{equation}
(A_{t}(v),W_{t}(v)),\quad\left\vert v\right\vert <v_{\max},\label{AtWt}%
\end{equation}
where
\[
W_{t}(v)=\frac{t^{2}+4(t+2)v^{2}-t(t+4v^{2})\cos\left(  4v/t\right)
-4tv\sin\left(  4v/t\right)  }{4\pi t\left(  -t^{2}+8v^{2}+t^{2}\cos\left(
4v/t\right)  \right)  }.
\]
To the extent that we can compute the height $b_{t}$ of $\Omega_{t}$ as a
function of $a,$ we can then compute the density of the Brown measure as a
function of $a$ by replacing $v$ by $b_{t}(a)/2$ in the above expression. That
is, the density $w_{t}(\lambda)$ is given by
\begin{equation}
w_{t}(\lambda)=W_{t}(b_{t}(a)/2).\label{wtWt}%
\end{equation}

\end{enumerate}
\end{proposition}

See Figure \ref{uniformdensity.fig}.%

\begin{figure}[ptb]%
\centering
\includegraphics[
height=2.808in,
width=3.5276in
]%
{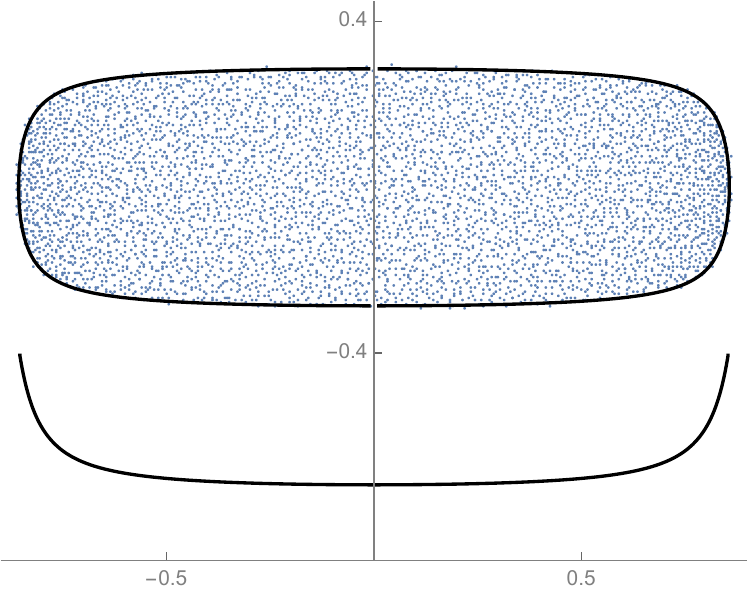}%
\caption{The domain $\Omega_{t}$ in the uniform case with a simulation of the
eigenvalues (top), plotted with the density of the Brown measure (in
$\Omega_{t}$) as a function of $a$ (bottom). Shown for $t=0.1.$}%
\label{uniformdensity.fig}%
\end{figure}

Before we prove Proposition \ref{prop:UniformOmega}, we need some computations
about $v_{t}$ and $\Lambda_{t}$ from the following proposition.

\begin{proposition}
The following results about $v_{t}$ and $\Lambda_{t}$ hold.

\begin{enumerate}
\item The function $v_{t}$ is unimodal, with a peak at $a_{0}=0$. The maximum
$v_{\max} = v_{t}(0)$ is the smallest positive real number $v$ such that
\[
\frac{1}{v}=\tan\left(  \frac{v}{t}\right)  .
\]
In particular, $v_{t}(0)< \frac{\pi t}{2}$.

\item The domain $\Lambda_{t}$ has only one connected component. Its boundary
can be described by the two curves
\[
\pm A_{0}^{t}(v)+iv,\quad\left\vert v\right\vert \leq v_{\max}%
\]
where
\begin{equation}
A_{0}^{t}(v)=\sqrt{2v\cot\left(  \frac{2v}{t}\right)  +1-v^{2}}.
\label{eq:Uniforma0}%
\end{equation}

\item The domain $\Lambda_{t}$ satisfies
\begin{equation}
\label{eq:LambdatUniform}\Lambda_{t}\cap\mathbb{R }= (-\sqrt{t+1}, \sqrt
{t+1}).
\end{equation}

\end{enumerate}
\end{proposition}

\begin{proof}
In the uniform case, (\ref{p0EqIntro}) takes the form
\begin{equation}
\int_{\mathbb{R}}\frac{d\mu(x)}{(a_{0}-x)^{2}+v^{2}}=\frac{\arctan\left(
\frac{1-a_{0}}{v}\right)  +\arctan\left(  \frac{1+a_{0}}{v}\right)  }%
{2v}=\frac{1}{t}. \label{eq:Uniformp0}%
\end{equation}
Using the addition law of inverse tangent,
\[
\arctan(A)+\arctan(B)=\arctan\left(  \frac{A+B}{1-AB}\right)  ,
\]
we can easily solve for $a_{0}^{2}$ as a function of $v$:
\begin{equation}
a_{0}^{2}=2v\cot\left(  \frac{2v}{t}\right)  +1-v^{2}. \label{eq:a02Inv}%
\end{equation}
Restricted to $a_{0}\geq0$, (\ref{eq:a02Inv}) defines $a_{0}=A_{0}^{t}$ as a
function of $v$ as in (\ref{eq:Uniforma0}).

The function $v_{t}(a_{0})$ cannot be represented as an elementary function of
$a_{0}$; we, however, have proved in the preceding paragraph that $v_{t}$
restricted to $a_{0}\geq0$ in $\Lambda_{t}$ has an inverse $A_{0}^{t}$. The
function $v_{t}$ then must be strictly decreasing from $0$ to $\sup
(\Lambda_{t}\cap\mathbb{R})$; by symmetry, it is strictly increasing from
$\inf(\Lambda_{t}\cap\mathbb{R})$ to $0$. In particular, $v_{t}$ is unimodal
with global maximum at $a_{0}=0$. Putting $a_{0}=0$ in (\ref{eq:Uniformp0}),
the maximum $v_{\max}=v_{t}(0)$ is the smallest positive real number $v$ such
that
\[
\frac{1}{v}=\tan\left(  \frac{v}{t}\right)  .
\]
Thus, $v_{t}(0)=v_{\max}<\frac{\pi t}{2}$. This proves Point 1.

Since $v_{t}$ is unimodal, the domain $\Lambda_{t}$ has only one connected
component. By Definition \ref{lambdaT.def}, $\Lambda_{t}$ is symmetric about
the real axis. In our case, $\Lambda_{t}$ is also symmetric about the
imaginary axis and the right hand side of (\ref{eq:Uniforma0}) defines an even
function of $v$. Thus, the boundary of $\Lambda_{t}$ can be described by the
curves
\[
\partial\Lambda_{t}=\{(\pm A_{0}^{t}(v),v)|~\left\vert v\right\vert \leq
v_{\max}\},
\]
which is Point 2. Since $a_{0}(v)^{2}\rightarrow t+1$ as $v\rightarrow0$,
(\ref{eq:LambdatUniform}) holds, which proves Point 3.
\end{proof}

\begin{proof}
[Proof of Proposition \ref{prop:UniformOmega}]In the uniform case, the
equation (\ref{p1EqIntro}) takes the form
\begin{align*}
a  &  =t\int_{\mathbb{R}}\frac{x}{(a_{0}-x)^{2}+v_{t}(a_{0})^{2}}~d\mu(x)\\
&  =a_{0}+\frac{t}{4}\log\left(  1-\frac{4a_{0}}{(a_{0}+1)^{2}+v_{t}%
(a_{0})^{2}}\right)  .
\end{align*}
For $a\geq0$ in $\Omega_{t}$, we can express $a$ as a function $a=A_{t}(v)$ of
$v=v_{t}(a_{0})$ using (\ref{eq:Uniforma0}) as
\begin{equation}
A_{t}(v)=A_{0}^{t}(v)+\frac{t}{4}\log\left(  1-\frac{4A_{0}^{t}(v)}{(A_{0}%
^{t}(v)+1)^{2}+v^{2}}\right)  \label{eq:Uniforma}%
\end{equation}
for $0<v\leq v_{\max}$.

Using the definition of $b_{t}$ in (\ref{eq:btDef}), $b_{t}(a)=2v_{t}%
(a_{0}^{t}(a))=2v$. Thus, $b_{t}$ is unimodal on $\Omega_{t}\cap\mathbb{R}$
with a peak at $0$. The domain $\Omega_{t}$ has only one connected component
whose boundary can be described by the two curves
\[
\pm A_{t}(b/2)+ib,\quad\left\vert b\right\vert \leq2v_{\max}.
\]
Thus, (\ref{eq:OmegatUniformSet}) follows. This proves Point 1.

Using (\ref{eq:LambdatUniform}), we can compute the limit as $v\rightarrow0$
in (\ref{eq:Uniforma}), from which the claimed from of $\Omega_{t}%
\cap\mathbb{R}$ follows, establishing Point 2.

Now, since both $a_{0}$ and $a$ are functions of $v$ when $a_{0}$ and $a$ are
nonnegative, we can compute
\begin{equation}
\frac{da_{0}^{t}(a)}{da}=\frac{dA_{0}^{t}(v)/dv}{dA_{t}(v)/dv}.
\label{eq:da0daUniform}%
\end{equation}
This result also holds for negative $a$ because $\Lambda_{t}$ and $\Omega_{t}$
are symmetric about the imaginary axis. The density of the Brown measure is
then computed using Theorem \ref{thm:main}. Using (\ref{eq:da0daUniform}), the
density can be expressed in terms of $v$ as
\[
\frac{1}{2\pi t}\left(  \frac{da_{0}^{t}(a)}{da}-\frac{1}{2}\right)
=\frac{t^{2}+4(t+2)v^{2}-t(t+4v^{2})\cos\left(  \frac{4v}{t}\right)
-4tv\sin\left(  \frac{4v}{t}\right)  }{4\pi t\left(  -t^{2}+8v^{2}+t^{2}%
\cos\left(  \frac{4v}{t}\right)  \right)  },
\]
establishing (\ref{AtWt}). Since $b_{t}(a)=2v$, we obtain (\ref{wtWt}),
completing the proof.
\end{proof}

\subsection*{Acknowledgments}

The authors thank Todd Kemp, Maciej Nowak, and Roland Speicher for useful
discussions. We also thank the referee for reading the paper thoroughly and
for making valuable suggestions and corrections.


\begin{thebibliography}{99}                                                                                               %


\bibitem {Bai}Z. D. Bai, Circular law, \textit{Ann. Probab.} \textbf{25}
(1997), 494--529.

\bibitem {BianeConvolution}P. Biane, On the free convolution with a
semi-circular distribution, \textit{Indiana Univ. Math. J.} \textbf{46}
(1997), 705--718.

\bibitem {BianeFreeIncr}P. Biane, Processes with free increments,
\textit{Math. Z.} \textbf{227} (1998), 143--174.

\bibitem {BL}P. Biane and F. Lehner, Computation of some examples of Brown's
spectral measure in free probability, \textit{Colloq. Math.} \textbf{90}
(2001), 181--211.

\bibitem {BS1}P. Biane and R. Speicher, Stochastic calculus with respect to
free Brownian motion and analysis on Wigner space, \textit{Probab. Theory
Related Fields} \textbf{112} (1998), 373--409.

\bibitem {BMS}S. T. Belinschi, T. Mai, and R. Speicher, Analytic subordination
theory of operator-valued free additive convolution and the solution of a
general random matrix problem, \textit{J. Reine Angew. Math.} \textbf{732}
(2017), 21--53.

\bibitem {BSS}S. T. Belinschi, P. \'{S}niady, and R. Speicher, Eigenvalues of
non-Hermitian random matrices and Brown measure of non-normal operators:
Hermitian reduction and linearization method, \textit{Linear Algebra Appl.}
\textbf{537} (2018), 48--83.

\bibitem {Br}L. G. Brown, Lidski\u{\i}'s theorem in the type II case.
\textit{In} Geometric methods in operator algebras (Kyoto, 1983), 1--35,
Pitman Res. Notes Math. Ser., 123, Longman Sci. Tech., Harlow, 1986.

\bibitem {BGNTW1}Z. Burda, J. Grela, M. A. Nowak, W. Tarnowski, and P.
Warcho\l , Dysonian dynamics of the Ginibre ensemble, \textit{Phys. Rev.
Letters} \textbf{113} (2014), article 104102.

\bibitem {BGNTW2}Z. Burda, J. Grela, M. A. Nowak, W. Tarnowski, and P.
Warcho\l , Unveiling the significance of eigenvectors in diffusing
non-Hermitian matrices by identifying the underlying Burgers dynamics,
\textit{Nucl. Phys. B} \textbf{897} (2015), 421--447.

\bibitem {DH}N. Demni and T. Hamdi, Support of the Brown measure of the
product of a free unitary Brownian motion by a free self-adjoint projection,
\textit{J. Funct. Anal.}, \textbf{282} (2022), Article 109362, 36 pp.

\bibitem {DHKBrown}B. K. Driver, B. C. Hall, and T. Kemp, The Brown measure of
the free multiplicative Brownian motion, arXiv:1903.11015 [math.PR].

\bibitem {Evans}L. C. Evans, Partial differential equations. Second edition.
Graduate Studies in Mathematics, 19. American Mathematical Society,
Providence, RI, 2010. xxii+749 pp.

\bibitem {GirkoCircular}V. L. Girko, The circular law, \textit{Theory Probab.
Appl.} \textbf{29} (1984), 694--706.

\bibitem {GirkoElliptic}V. L. Girko, Elliptic law, \textit{Theory Probab.
Appl.} \textbf{30} (1985), 677--690.

\bibitem {GNT}J. Grela, M. A. Nowak, and W. Tarnowski, Eikonal formulation of
large dynamical random matrix models, \textit{Phys. Rev. E} \textbf{104}
(2021), Article 054111, 12 pp.

\bibitem {GKZsingleRing}A. Guionnet, M. Krishnapur, and O. Zeitouni, The
single ring theorem, \textit{Ann. of Math.} (2) \textbf{174} (2011), 1189--1217.

\bibitem {PDEmethods}B. C. Hall, PDE methods in random matrix theory. In
Harmonic analysis and applications (M. T. Rassias, Ed.), 77-124, Springer 2021.

\bibitem {HHmult}B. C. Hall and C.-W. Ho, The Brown measure of a family of
free multiplicative Brownian motions, arXiv:2104.07859 [math.PR]

\bibitem {Ho1}C.-W. Ho, The Brown measure of the sum of a self-adjoint element
and an elliptic element, arXiv:2007.06100 [math.OA]

\bibitem {Ho2}C.-W. Ho, The Brown measure of unbounded variables with free
semicircular imaginary part, arXiv:2011.14222 [math.OA]

\bibitem {HZ}C.-W. Ho and P. Zhong, Brown Measures of free circular and
multiplicative Brownian motions with self-adjoint and unitary initial
conditions, \textit{J. Europ. Math. Soc.}, to appear.

\bibitem {JNPWZ}R. Janik, M. A. Nowak, G. Papp, J. Wambach, and I. Zahed,
Non-Hermitian random matrix models: free random variable approach,
\textit{Phys. Rev. E} \textbf{55} (1997), 4100--4106.

\bibitem {JN1}A. Jarosz and M. A. Nowak, A novel approach to non-Hermitian
random ratrix models, preprint arXiv:math-ph/0402057.

\bibitem {JN2}A Jarosz and M. A. Nowak, Random Hermitian versus random
non-Hermitian operators---unexpected links, \textit{J. Phys. A} \textbf{39}
(2006), 10107--10122.

\bibitem {Kargin}V. Kargin, Subordination for the sum of two random matrices,
\textit{Ann. Probab.} \textbf{43} (2015), 2119--2150.

\bibitem {KempLargeN}T. Kemp, The large-$N$ limits of Brownian motions on
$\mathbb{GL}_{N}$, \textit{Int. Math. Res. Not.}, (2016), 4012--4057.

\bibitem {MS}J. A. Mingo and R. Speicher, Free probability and random
matrices. Fields Institute Monographs, 35. Springer, New York; Fields
Institute for Research in Mathematical Sciences, Toronto, ON, 2017.

\bibitem {Nik}E. A. Nikitopoulos, It\^{o}'s formula for noncommutative $C^{2}$
functions of free It\^{o} processes with respect to circular Brownian motion,
arXiv:2011.08493 [math.OA].

\bibitem {PV}L. Pastur and V. Vasilchuk, On the law of addition of random
matrices, \textit{Comm. Math. Phys.} \textbf{214} (2000), 249--286.

\bibitem {Rudin}W. Rudin, Real and complex analysis, \textit{McGraw-Hill.} (1987).

\bibitem {stephanov}M. A. Stephanov, Random matrix model of QCD at finite
density and the nature of the quenched limit, \textit{Phys. Rev. Letters}
\textbf{76} (1996), 4472-4475.

\bibitem {TV}T. Tao and V. Vu, Random matrices: universality of ESDs and the
circular law. With an appendix by Manjunath Krishnapur, \textit{Ann. Probab.}
\textbf{38} (2010), 2023--2065.

\bibitem {Voi1}D. V. Voiculescu, Symmetries of some reduced free product
$C^{\ast}$-algebras. \textit{In} \textquotedblleft Operator algebras and their
connections with topology and ergodic theory (Bu\c{s}teni,
1983),\textquotedblright\ 556--588, Lecture Notes in Math., 1132, Springer,
Berlin, 1985.

\bibitem {Voi2}D. V. Voiculescu, Limit laws for random matrices and free
products, \textit{Invent. Math.} \textbf{104} (1991), 201--220.

\bibitem {Voi3}D. V. Voiculescu, The analogues of entropy and of Fisher's
information measure in free probability theory. I. Comm. Math. Phys. 155, 1
(1993), 71--92.

\bibitem {Voi4}D. V. Voiculescu, The coalgebra of the free difference quotient
and free probability, \textit{Int. Math. Res. Not.} \textbf{2000} (2000), 79--106.

\bibitem {VDN}D. V. Voiculescu, K. J. Dykema, and A. Nica, Free random
variables. A noncommutative probability approach to free products with
applications to random matrices, operator algebras and harmonic analysis on
free groups. CRM Monograph Series, 1. American Mathematical Society,
Providence, RI, 1992. vi+70.

\bibitem {Zhong2}P. Zhong, Brown measure of the sum of an elliptic operator
and a free random variable in a finite von Neumann algebra, arXiv:2108.09844 [math.OA]
\end{thebibliography}
\end{document}